\documentclass[11pt,a4paper]{article}
\usepackage{fullpage}
\usepackage{stmaryrd}
\usepackage{amsmath,amsfonts,amssymb,mathrsfs,amsthm,extarrows,textcomp,amscd}
\usepackage{enumerate,enumitem}
\usepackage{cases}
\usepackage{bm}
\usepackage{diagbox,multirow}
\usepackage{longtable}
\usepackage{graphicx,subfig}
\usepackage{algorithm,algorithmic}
\usepackage{caption}
\usepackage{xcolor}
\usepackage{cite}
\usepackage[colorlinks,linkcolor=blue,anchorcolor=red,citecolor=green]{hyperref}

\usepackage{authblk}
\theoremstyle{definition}

\newtheorem{theorem}{\bf Theorem}[section]
\newtheorem{lemma}{\bf Lemma}[section]

\newtheorem{remark}{\bf Remark}[section]

\numberwithin{equation}{section}

\numberwithin{table}{section}
\numberwithin{figure}{section}

\newcommand{\dx}{\,{\rm d}{\boldsymbol{x}}}
\newcommand{\ds}{\,{\rm d}{s}}
\newcommand{\curl}{\,{\rm curl}\,}
\newcommand{\Curl}{\,{\rm \bf{curl}}\,}

\newcommand{\dive}{\,{\rm div}\,}

\title{Decoupled Divergence-Free Neural Networks Basis Method for Incompressible Fluid Problems}

\author[1]{Jinbao Cheng %
\thanks{This author is now working in Banma Network Technology Co., Ltd., Shanghai 200030, China}}

\author[1]{Jianguo Huang %
\thanks{Corresponding author}}

\author[2,3]{Haoqin Wang %
\thanks{The work of this author was partially carried out at the Academy of Mathematics and Systems Science,
Chinese Academy of Sciences}}

\author[3]{Tao Zhou}

\affil[1]{\small School of Mathematical Sciences and MOE-LSC, Shanghai Jiao Tong University, \newline Shanghai 200240, China}
\affil[2]{\small School of Mathematics and Statistics, Yunnan University, Kunming 650500, China}
\affil[3]{\small Institute of Computational Mathematics and Scientific/Engineering Computing, \newline Academy of Mathematics and Systems Science, Chinese Academy of Sciences, Beijing 100190, China}

\affil[ ]{\footnotesize Emails: \texttt{jinbao\_cheng@163.com} (J. Cheng), \texttt{jghuang@sjtu.edu.cn} (J. Huang), \texttt{wanghaoqin@ynu.edu.cn} (H. Wang), \texttt{tzhou@lsec.cc.ac.cn} (T. Zhou)}

\date{}

\begin{document}
\maketitle

\begin{abstract}
We propose a decoupled divergence-free neural networks basis (Decoupled-DFNN) method for numerically solving incompressible flow problems, including the Stokes and Navier-Stokes equations. To ensure the divergence free property exactly, the velocity field is represented as the curl of a stream function in two dimensions and as the curl of a vector potential in three dimensions. Beyond classical stream-function or velocity-vorticity formulations, we further utilize the properties of the curl operator to derive two specific decoupled subproblems for the velocity (through the stream function or vector potential) and the pressure, respectively.  The proposed formulations enable a sequential solution strategy, in which the velocity and pressure are solved independently. To resolve the inherent nonlinearity of the Navier-Stokes equations, we employ a Gauss-Newton linearization strategy, transforming the nonlinear velocity subproblem into a sequence of linear subproblems. These decoupled subproblems for velocity and pressure are subsequently solved using the TransNet framework. Compared with existing methods, the proposed approach reduces computational cost while strictly preserving the incompressibility constraint.

\end{abstract}

\section{Introduction}\label{sec1}
Incompressible fluid problems have broad applications in science and engineering. However, except for a few special cases, it is impossible to get analytical solutions to these problems. To address these challenges, several classical numerical methods have been developed, including finite difference, finite element, and spectral methods \cite{Temam, Girault-Raviart, Canuto, Quarteroni-Valli, Quartapelle}.
In this work, we focus on a neural network-based numerical approach for solving the steady Stokes problem
\begin{subnumcases}{}
-\nu \Delta \bm{u} + \nabla p = \bm{f} \quad \text{in}\ \Omega, \label{SE-a}
\\
\dive \bm{u} = 0 \quad \text{in} \ \Omega, \label{SE-b}
\\
\bm{u} = \bm{0} \quad \text{on} \ \Gamma;   \label{SE-c}
\end{subnumcases}
and the Navier-Stokes problem
\begin{subnumcases}{}
-\nu \Delta \bm{u} + (\bm{u}\cdot \nabla) \bm{u} + \nabla p = \bm{f} \quad \text{in} \ \Omega, \ \label{NS-a}
\\
\dive \bm{u} = 0 \quad \text{in} \ \Omega, \ \label{NS-b}
\\
\bm{u} = \bm{0} \quad \text{on} \ \Gamma.  \label{NS-c}
\end{subnumcases}
Here, $\Omega$ is a bounded simply-connected domain with Lipschitz boundary $\Gamma$ in $\mathbb{R}^d\  (d=2,3)$, $\bm{u}$ is the non-dimensional velocity vector, $p$ is the non-dimensional pressure, the parameter $\nu = \mathrm{Re}^{-1}$ represents the kinematic viscosity with $\mathrm{Re}$ being the Reynolds number.

In recent years, machine learning techniques have been widely adopted in computational and applied mathematics, leading to the development of several neural network-based numerical methods for partial differential equations (PDEs). These approaches can be generally divided into two categories. The first category involves stochastic methods, where PDEs are reformulated as stochastic differential equations (SDEs) and solved using deep learning techniques, such as DeepBSDE \cite{E-2017}, deep backward schemes \cite{Hure-2020}, and martingale neural networks \cite{Cai-2026,Cai-2025}, among others. The second category is direct methods. The most well-known of these is the physics-informed neural network (PINN) method \cite{PINN}. A comprehensive review of PINN and its variants can be found in \cite{PIML,DeepXDE,Cuomo-2022}, and numerical analysis can be found in \cite{Acta-2024}. While PINNs typically aim to obtain classical solutions to PDEs, other methods have been developed to find weak solutions, such as the deep Ritz method \cite{deepRitz}, weak adversarial networks \cite{WAN}, and Friedrichs learning \cite{Chen-2023}. In these approaches, deep neural networks are utilized to parameterize the solutions to either SDEs or PDEs, where the network parameters are optimized to satisfy the governing physical laws. The great success of these methods is largely due to the universal approximation properties of neural networks \cite{Hornik-1989,Acta-2021}, particularly their ability to mitigate the ``curse of dimensionality'' \cite{Barron1993}. For a more detailed overview of machine learning in computational mathematics, we refer the reader to \cite{E-review1, E-review2}.

Although these methods have achieved significant success in various PDEs, their computational cost remains substantial, primarily due to the expensive non-convex optimization required for deep architectures. Consequently, one is tempted to utilize shallow neural networks as a solution ansatz to balance computational accuracy and efficiency. Specifically, the weights and biases of the hidden layer are randomly initialized and kept fixed, while the trainable parameters are limited to the weight coefficients of the output layer. By constructing the loss function based on the physics-informed residual minimization principle, a least-squares minimization problem with respect to the output layer weights is formed, which can be efficiently solved using existing linear or nonlinear least-squares solvers. In this way, the training process is transformed into a numerical linear algebra problem, which significantly reduces training time while preserving the approximation power of neural networks. This approach is known as the extreme learning machine (ELM) \cite{Huang-2006}, with early works including \cite{Sun-2019,Dwivedi-2020,Calabro-2021}. Subsequently, Dong and Li combined non-overlapping domain decomposition with local ELM to solve PDEs \cite{Dong-2021a}. In \cite{Chen-2022}, the partition of unity method was combined with randomized neural networks to produce the random feature method, and \cite{Shang-2023, Sun-2024} integrated traditional numerical schemes with randomized neural networks to develop new frameworks. Furthermore, \cite{Weng-2026} adopted adaptive strategies to design an efficient algorithm. A key issue in ELM and its variants is the choice of hyperparameters. Some efficient algorithms to address this have been proposed in \cite{Dong-2021b, Dong-2022}. Furthermore, \cite{TransNet} introduced a strategy for generating uniformly distributed neurons in the unit ball. To unify the previous methods, Huang and Wu propose the neural network basis method and its Gauss-Newton optimizer for solving general PDEs \cite{Huang-2025}. More recently, the work was extended to handle low regularity problems \cite{Huang-2026}.

In the context of fluid simulations, a number of specialized machine learning methods have also been proposed. The Stokes problem is reformulated in a first-order velocity-pressure-vorticity form, and convergence and error estimates are established for the two-dimensional case in \cite{Liu-2022}. \cite{Jin-2021} applied PINNs to solve the Navier-Stokes equations in both velocity-pressure and vorticity-velocity formulations. To accelerate training for the steady Navier-Stokes equations, \cite{Liu-2023} proposed four linearized iteration strategies. Margenberg et al. developed a hybrid solver by augmenting coarse-grid finite element simulations with fine-scale information provided by deep neural networks \cite{Margenberg-2d, Margenberg-3d}. In two dimensions, they further employed stream functions to construct structure-preserving algorithms \cite{Margenberg-2022}. In \cite{Zhang-2022}, the authors propose deep random vortex networks (DRVN), which represents the velocity field using a probabilistic formulation based on the Feynman-Kac formula. Cherepanov and Ertel extended this idea to an implicit DRVN that strictly enforces physical constraints in two dimensions without requiring the Biot-Savart kernel \cite{Cherepanov-2025}. Additionally, \cite{Kashefi-2022} combined PINNs with PointNet to solve Navier-Stokes problems on irregular geometries. As for the error analysis, \cite{Minakowski-2023} derived a priori and a posteriori error estimates for the Stokes problem solved by the deep Ritz method, while \cite{DeRyck-2024} establishes rigorous error bounds for PINN approximations of the incompressible Navier-Stokes equations.

It is worth mentioning that the divergence-free condition is a fundamental conservation law in incompressible fluid flow, which means the mass conservation of the underlying fluid. Preserving the divergence-free condition has been extensively studied in classical numerical methods. Common approaches include the potential-based method \cite{Girault-Raviart,Jardin2010}, the Lagrange multiplier method \cite{John-2017}, and the divergence-free bases method \cite{Wang-2009,Cai-2013}. We refer to \cite{Hu-2017} and the references therein for details along this line. In particular, exactly divergence-free discretizations can be constructed using Stokes complexes which are a class of smoothed de Rham complexes. However, the construction of discrete analogues of Stokes complexes is rather involved. In two dimensions, Falk and Neilan \cite{FalkNeilan2013} developed stable, exactly divergence-free finite element methods based on Stokes complex
\[
    \mathbb{R}
    \xrightarrow{} H^2(\Omega)
    \xrightarrow{\Curl} [H^1(\Omega)]^2
    \xrightarrow{\dive} L^2(\Omega)
    \longrightarrow 0.
\]
In three dimensions, Neilan \cite{Neilan2015} constructed conforming finite element discretizations based on the complex
\[
    \mathbb{R}
    \xrightarrow{} H^2(\Omega)
    \xrightarrow{{\rm grad}} H^1(\Curl;\Omega)
    \xrightarrow{\Curl} [H^1(\Omega)]^3
    \xrightarrow{\dive} L^2(\Omega)
    \longrightarrow 0.
\]
The review of the above results can be found in \cite{Neilan2020} and  more recent complex-based constructions have been developed in \cite{Chen-2024}. In order to make the presentation clear, the function spaces appearing in complexes will be introduced in next section.

On the other hand, most existing neural networks-based methods enforce this constraint through adding penalty terms in the loss function, but it is a critical issue to choose the penalty parameter feasibly. The other way is to use the stream function based formulations in two dimensions  (cf. \cite{Margenberg-2022, Liu-2022}), which can ensure the incompressibility exactly. However, developing an efficient neural network-based formulation
in three dimensions remains challenging. Moreover, even regardless of whether a velocity-pressure or velocity-vorticity formulation is adopted (cf. \cite{Jin-2021}), the resulting system consists of coupled PDEs in which the unknown variables must be solved simultaneously, leading to high computational complexity.

In this paper, we propose the decoupled divergence-free neural networks basis (Decoupled-DFNN) method for numerically solving the steady Stokes equation and Navier-Stokes equation in two-dimensions and three-dimensions. We treat the incompressibility requirement as a linear constraint. Following the general framework for exact enforcement of linear constraints $\mathcal{A}(u) =0$ introduced in \cite{Du-2021}, one needs to construct a linear operator $\mathcal{G}$ such that
\begin{equation}
\label{key-identity}
    \operatorname{im}(\mathcal{G}) = \operatorname{ker}(\mathcal{A}),
\end{equation}
By substituting the solution ansatz $u = \mathcal{G}(v)$, the constraint is satisfied exactly along the solution trajectory, i.e. $\mathcal{A}\circ \mathcal{G}(v)=0$. In our context, $\mathcal{A} = \dive,$ and we choose $\mathcal{G}=\Curl$,
which is acted on a scalar potential in two dimensions and a vector potential in
three dimensions. Then, due to the exactness  of the Stokes complex $\operatorname{im}(\Curl)=\operatorname{ker}(\dive)$, the condition \eqref{key-identity} holds true.

Specifically speaking, in the two-dimensional case, we build upon the well-established stream function formulation \cite{Quartapelle}. While introducing the stream function to ensure incompressibility, the pressure remains in the momentum equations; we therefore introduce a further refinement. By applying the curl operator to both sides of the governing equations, we eliminate the pressure gradient and derive a fully decoupled subproblem for the velocity field (represented by the stream function), as discussed in \cite{Guo-1997, Lequeurre-2020}. This subproblem is then solved efficiently using the TransNet framework. Once the velocity field is obtained, the remaining governing equation reduces to a pressure-only problem, for which a classical solution can be recovered efficiently by TransNet as well. In the three-dimensional case, our method draws inspiration from the vorticity-vector potential formulations \cite{E-1997, Liu-2001}. Whereas standard formulations of this type often involve a coupling between the potential and vorticity, our approach further transforms the governing equations to derive a decoupled subproblem formulated solely for the vector potential by utilizing the properties of the curl operator. This extension allows us to resolve the velocity field (represented by the vector potential) without simultaneous calculation of the vorticity or pressure. Subsequently, the pressure field is recovered from the resolved velocity. Similarly, the decoupled subproblems are resolved by the TransNet.

Though the proposed method retains the essential potential-to-velocity mapping for generating divergence-free fields, it does not require the construction of an exact discrete Stokes complex owing to the decoupled formulation. In this sense, this method can be also viewed as a potential-based approach to constructing exactly divergence-free approximation.
Compared with existing methods in our experiments, the Decoupled-DFNN reduces computational cost by breaking these traditional couplings while strictly preserving the incompressibility constraint to machine precision. Our contributions are summarized as follows:
\begin{itemize}
\item We derive decoupled divergence-free formulations for incompressible flows in both two and three dimensions, leading to single-field subproblems for the velocity representation (stream function in 2D and vector potential in 3D) that are free of pressure and vorticity coupling.

\item We develop a neural network-based framework to efficiently solve the resulting decoupled subproblems for both velocity and pressure. The key advantage lies in its mesh-free feature.

\item We demonstrate that the proposed Decoupled-DFNN enforces the divergence-free constraint to near machine precision while significantly reducing computational cost.
\end{itemize}

The rest of this paper is organized as follows. Section \ref{sec2} introduces the notation and presents some preliminaries. Section \ref{sec3} introduces decoupled divergence free formulations for the Stokes and Navier-Stokes equations in both two and three dimensions. Based on these formulations, Decoupled-DFNN methods are proposed in Section \ref{sec4}. Section \ref{sec5} presents several numerical experiments to demonstrate the performance of the proposed methods. Conclusions are drawn in the last section.

\section{Preliminaries}\label{sec2}
In this section, we first recall the definitions of some function spaces and differential operators, together with the potential representations of divergence-free vector fields. We then collect several vector identities and derive identities for the nonlinear advection term that will be used later on.

Throughout this paper, we use standard notation for Sobolev spaces. We briefly introduce the spaces used below and refer to \cite{Sobolev}
for further details. Let $\Omega$ be a domain in $\mathbb{R}^n$. For any nonnegative integer $m$, the Sobolev space $H^m(\Omega)$ is defined by
\[
    H^m(\Omega) = \{u\in L^2(\Omega): D^{\boldsymbol{\alpha}}u\in L^2(\Omega),\ 0\leq |\boldsymbol{\alpha}|\leq m\},
\]
where $\boldsymbol{\alpha}=(\alpha_1,\ldots,\alpha_n)\in\mathbb{N}_0^n$ is a multi-index, $|\boldsymbol{\alpha}|=\sum_{i=1}^n\alpha_i$, and $D^{\boldsymbol{\alpha}}u$ denotes the corresponding weak derivative. In particular, $H^0(\Omega)=L^2(\Omega)$. For $d$-dimensional vector-value functions, we use the product space
\[
    [H^m(\Omega)]^d =
    \{ \bm{v}=(v_1,\cdots,v_d)^\intercal: v_i\in H^m(\Omega),\ i=1,\cdots,d \}.
\]

Let $\Omega \subset \mathbb{R}^2$ be a bounded contractible domain with Lipschitz boundary $\Gamma$. For $\psi \in \mathcal{D}'(\Omega)$ and $\bm{v} = [v_1, v_2]^{\rm T} \in [\mathcal{D}'(\Omega)]^2$, define
\[
    \Curl \psi:= \left[\frac{\partial \psi}{\partial y}, -\frac{\partial \psi}{\partial x}\right]^{\rm T},
    \quad   
    \curl \bm{v}:= \frac{\partial v_2}{\partial x} - \frac{\partial v_1}{\partial y}.
\]
Let $\Omega \subset\mathbb{R}^3$ be a bounded contractible region with Lipschitz boundary $\Gamma$. Define the curl operator for distribution $\bm{v} = [v_1, v_2, v_3]^{\rm T} \in [\mathcal{D}'(\Omega)]^3$ by
\[
    \Curl \bm{v}:= \nabla \times \bm{v} = \left[\frac{\partial v_3}{\partial y} - \frac{\partial v_2}{\partial z}, \frac{\partial v_1}{\partial z} - \frac{\partial v_3}{\partial x}, \frac{\partial v_2}{\partial x} - \frac{\partial v_1}{\partial y} \right]^{\rm T}.
\]

In Stokes complex, we define
\[
    H^1(\Curl;\Omega) = \{\bm{\psi}\in[H^1(\Omega)]^3: \Curl\bm{\psi}\in[H^1(\Omega)]^3 \}.
\]

The following classical result provides a stream-function representation and a vector-potential representation for divergence-free fields in $[L^2(\Omega)]^2$ and $[L^2(\Omega)]^3$, respectively. The more advanced results can be found in \cite{Costabel-2010}.

\begin{lemma}[Theorem 3.1 in \cite{Girault-Raviart}]\label{lem1}
A function $\bm{v} \in [L^2(\Omega)]^2$ satisfies
\[
    \dive \bm{v} = 0 \ \mbox{and} \ \int_\Gamma \bm{v}\cdot\bm{n} \ds =0,
\]
if and only if there exists a stream function $\psi \in H^1(\Omega)$ such that
\[
    \bm{v} = \Curl \psi.
\]
Here and below, $\bm{n}$ indicates the unit outward normal vector on the boundary.
\end{lemma}

\begin{lemma}[Theorems 3.4 and 3.5 in \cite{Girault-Raviart}]\label{lem2}
 A vector field $\bm{v} \in [L^2(\Omega)]^3$ satisfies
\[
    \dive \bm{v} = 0 \ \mbox{and} \ \int_\Gamma \bm{v}\cdot\bm{n} \ds =0, 
\]
if and only if there exists a potential function $\bm{\psi} \in [H^1(\Omega)]^3$ such that
\[
    \bm{v} = \Curl \bm{\psi}, \quad \dive \bm{\psi} = 0.
\]
Furthermore, the potential can be chosen in $\bm{\psi}\in H(\Curl; \Omega)$ to satisfy
\[
    \bm{\psi} \cdot \bm{n} = 0 \quad \text{on}\ \Gamma.
\]
\end{lemma}

We finally recall and derive some identities that will be used in the subsequent deduction.
\begin{lemma}[cf. \cite{Quartapelle,Girault-Raviart}]
For any vector function $\bm{v}, \bm{\psi} \in [C^1(\Omega)]^3$, there hold
\begin{align}
&(\bm{v}\cdot\nabla)\bm{v} = \frac{1}{2}\nabla |\bm{v}|^2-\bm {v}\times(\nabla\times\bm{v}), \label{idt-1}
\\
& \nabla\times( \bm{\psi} \times \bm{v})    
=(\nabla \cdot \bm{v})\bm{\psi} - (\nabla \cdot \bm{\psi})\bm{v} + (\bm{v} \cdot \nabla) \bm{\psi}   - (\bm{\psi} \cdot \nabla) \bm{v}. \label{idt-2}
\end{align}
Moreover, if $\bm{v}\in [C^2(\Omega)]^3$, then
\begin{equation}\label{idt-3}
\nabla \times (\nabla \times \bm{v}) = \nabla(\nabla\cdot\bm{v}) - \Delta \bm{v}.
\end{equation}
\end{lemma}

With the above results in mind, we can obtain the following two identities.

\begin{theorem}\label{thm1}
Let $\Omega \subset \mathbb{R}^2$ be a bounded domain with Lipschitz boundary and $\phi \in C^3(\Omega)$ a stream function with $\bm{u}=\Curl\phi$. Then there holds
\begin{equation}\label{NS2d-temp}
\curl\bigl((\bm{u}\cdot \nabla)\bm{u}\bigr)
=-(\Curl\phi\cdot\nabla)\Delta\phi.
\end{equation}
\end{theorem}

\begin{proof}
By the definition of the curl operator and the chain rule, we have
\begin{align*}
\curl\big((\bm{u}\cdot \nabla) \bm{u}\big)
&= \frac{\partial}{\partial x}\left( u_1\frac{\partial u_2}{\partial x} + u_2\frac{\partial u_2}{\partial y} \right) - \frac{\partial}{\partial y}\left(u_1\frac{\partial u_1}{\partial x} + u_2\frac{\partial u_1}{\partial y} \right)
\\
&= u_1 \left( \frac{\partial^2 u_2}{\partial x^2} - \frac{\partial^2 u_1}{\partial y\partial x} \right) + u_2 \left( \frac{\partial^2 u_2}{\partial x\partial y} - \frac{\partial^2 u_1}{\partial y^2} \right) + \left( \frac{\partial u_1}{\partial x} + \frac{\partial u_2}{\partial y} \right)\left( \frac{\partial u_2}{\partial x} - \frac{\partial u_1}{\partial y} \right)
\\
&=u_1 \left( \frac{\partial^2 u_2}{\partial x^2} - \frac{\partial^2 u_1}{\partial x\partial y} \right) + u_2 \left( \frac{\partial^2 u_2}{\partial y\partial x} - \frac{\partial^2 u_1}{\partial y^2} \right) + \dive \bm{u}\cdot \Curl \bm{u},
\end{align*}
which combined with $\dive \bm{u}=0$ implies
\[
\curl\big((\bm{u}\cdot \nabla) \bm{u}\big)
= u_1 \frac{\partial}{\partial x}\left( \frac{\partial u_2}{\partial x} - \frac{\partial u_1}{\partial y} \right) + u_2\frac{\partial}{\partial y} \left( \frac{\partial u_2}{\partial x} - \frac{\partial u_1}{\partial y} \right).
\]
Finally, substituting $\bm{u} = \Curl \phi$ into above expression yields
\begin{align*}
\curl\big((\bm{u}\cdot \nabla) \bm{u}\big)
&=  u_1 \frac{\partial}{\partial x}\left( -\frac{\partial^2 \phi}{\partial x^2} - \frac{\partial^2 \phi}{\partial y^2} \right) + u_2\frac{\partial}{\partial y} \left( -\frac{\partial^2 \phi}{\partial x^2} - \frac{\partial^2 \phi}{\partial y^2} \right) \\
&= - (\bm{u}\cdot \nabla) \Delta \phi
 = -(\Curl \phi \cdot \nabla)\Delta \phi.
\end{align*}
\end{proof}

\begin{theorem}\label{thm2}
Let $\Omega\subset\mathbb{R}^3$ be a bounded domain with Lipschitz boundary and $\bm{\phi}\in [C^3(\Omega)]^3$ a potential vector such that $\bm{u} = \Curl \bm{\phi}$ and $\dive \bm{\phi} = 0$.  Then there holds
\begin{equation}\label{NS3d-temp}
\Curl\big((\bm{u}\cdot \nabla) \bm{u} \big) = (\Delta \bm{\phi} \cdot \nabla) \Curl \bm{\phi} - (\Curl \bm{\phi} \cdot \nabla)\Delta \bm{\phi}.
\end{equation}
\end{theorem}

\begin{proof}
Write $\bm{\omega} = \Curl \bm{u}$, meaning the vorticity of $ \bm{u}$. By the identity \eqref{idt-1},
\[
    \Curl\big((\bm u\cdot\nabla)\bm u \big) = \frac{1}{2}\nabla\times \big(\nabla |\bm u|^2\big) -\nabla \times \big(\bm u\times(\nabla\times\bm u)\big ).
\]
Noting that $\nabla \times \nabla f = 0$ for any scalar function $f$, we see that
\begin{equation}\label{expension}
\Curl\big((\bm{u}\cdot \nabla) \bm{u} \big)
= -\nabla \times (\bm{u}\times\bm{\omega})
= \nabla\times( \bm{\omega} \times \bm{u}),
\end{equation}
which combined with \eqref{idt-2} yields
\[
\nabla\times( \bm{\omega} \times \bm{u})    
=
(\bm{u} \cdot \nabla) \bm{\omega} - (\bm{\omega} \cdot \nabla) \bm{u} + (\nabla \cdot \bm{u})\bm{\omega} - (\nabla \cdot \bm{\omega})\bm{u} .
\]
Since $\dive \bm{u} = 0$ and $\nabla\cdot\nabla\times\bm{v} = 0$ for any vector field $\bm v$, we can rewrite \eqref{expension}  as
\begin{equation}\label{NS3d-nonlinear}
    \Curl\big((\bm{u}\cdot \nabla) \bm{u} \big) = (\bm{u} \cdot \nabla) \bm{\omega} - (\bm{\omega} \cdot \nabla) \bm{u}.
\end{equation}
Using \eqref{idt-3}, we obtain
\begin{equation}\label{vorticity}
\bm{\omega} = \Curl \bm{u} = \Curl (\Curl \bm{\phi}) = \nabla(\nabla \cdot \bm{\phi}) - \Delta\bm{\phi} = - \Delta\bm{\phi}.
\end{equation}
Substituting \eqref{vorticity} into \eqref{NS3d-nonlinear} gives
\begin{equation*}
    \Curl\big((\bm{u}\cdot \nabla) \bm{u} \big) = -(\bm{u} \cdot \nabla)\Delta \bm{\phi} + (\Delta \bm{\phi} \cdot \nabla) \bm{u} =  (\Delta \bm{\phi} \cdot \nabla) \Curl \bm{\phi} - (\Curl \bm{\phi} \cdot \nabla)\Delta \bm{\phi}.
\end{equation*}
\end{proof}

\section{Decoupled divergence-free formulations for incompressible flows}\label{sec3}
In this section, we derive the decoupled divergence-free formulas for the steady Stokes and the Navier-Stokes equations in $\mathbb{R}^2$ and $\mathbb{R}^3$. The results in three dimensions differ from the classical vorticity-velocity (or vorticity-vector potential) formulations usually used in the literature. Throughout the remainder of this paper, we assume that $\Omega$ is a bounded contractible domain with Lipschitz boundary.

\subsection{The 2D case: stream function representation}
We first consider the two-dimensional Stokes equations \eqref{SE-a}--\eqref{SE-c}. By the divergence theorem and the incompressibility constraint, it follows that
\[
    \int_{\Gamma} \bm{u} \cdot \bm{n} \ds = \int_{\Omega} \dive \bm{u} \dx = 0.
\]
This identity implies that for any divergence-free velocity field $\bm{u} \in [H^1(\Omega)]^2$, there exists a stream function $\phi \in H^2(\Omega)$ such that $\bm{u}$ can be represented via $\phi$ (cf. Lemma \ref{lem1}).

To derive a decoupled, divergence-free formulation, we first reformulate the governing equations in terms of $\phi$. Suppose $\bm{u} \in [H^3(\Omega)]^2$ and $\phi \in H^4(\Omega)$. Then applying the curl operator to the left-hand side of \eqref{SE-a} yields
\[
\curl (-\nu \Delta \Curl \phi + \nabla p)
= -\nu \Delta \curl \Curl \phi + \curl\nabla p
=\nu\Delta^2 \phi.
\]
Moreover, a direct computation confirms that $\dive \Curl \phi = 0$. Consequently, \eqref{SE-a}--\eqref{SE-b} can be transformed into the following decoupled biharmonic equation
\begin{equation*}
\nu\Delta^2 \phi = \curl \bm{f} \quad \text{in}\ \Omega.
\end{equation*}
Concerning the boundary condition \eqref{SE-c}, the no-slip condition $\bm{u} = 0$ implies
\[\begin{cases}
\bm{u} \cdot \bm{n}_0 = \curl \phi \cdot \bm{n}_0 = \nabla\phi \cdot \bm{\tau}_0 = 0,
\\
\bm{u} \cdot \bm{\tau}_0 = \curl \phi \cdot \bm{\tau}_0 = -\nabla\phi \cdot \bm{n}_0 =0,
\end{cases}\]
where $\bm{n}_0$ and $\bm{\tau}_0$ denote the unit outward normal and unit tangent vector on $\Gamma$, respectively. In other words, we have
\begin{equation}\label{SE2d-temp}
\nabla \phi = \bm{0} \quad \text{on}\ \Gamma.
\end{equation}
From \eqref{SE2d-temp}, we know $\phi$ is constant along the boundary. Since the velocity field depends only on the derivatives of $\phi$, we impose $\phi = 0$ on $\Gamma$ without loss of generality to ensure the uniqueness of the solution. So the decoupled subproblem for the velocity in terms of $\phi \in H_0^2(\Omega)$ reads
\begin{equation}\label{SE2d-vel}
\begin{cases}
\nu\Delta^2 \phi = \curl \bm{f} \quad &\text{in}\ \Omega,
\\
\phi = \frac{\partial \phi}{\partial n} = 0 \quad &\text{on}\ \Gamma.
\end{cases}
\end{equation}

\begin{remark}
If we consider the inhomogeneous Dirichlet boundary condition $\bm{u} = \bm{g}$ on the boundary $\Gamma$, a similar derivation gives
\[\begin{cases}
\bm{u} \cdot \bm{n}_0 = \nabla\phi \cdot \bm{\tau}_0 = \bm{g} \cdot \bm{n}_0,
\\
\bm{u} \cdot \bm{\tau}_0 = -\nabla\phi \cdot \bm{n}_0 =  \bm{g} \cdot \bm{\tau}_0,
\end{cases}\]
Thus, we can impose the constraint of $\phi$ on the boundary.
\end{remark}

Once the velocity $\bm{u} = \Curl \phi$ is determined, the pressure $p$ can be computed from
\begin{equation*}
\nabla p = \bm{f} + \nu \Delta \bm{u}.
\end{equation*}
To ensure a unique solution for the pressure, we impose an additional reference constraint
\begin{equation*}
    p(\bm{x}_0) = 0,
\end{equation*}
where $\bm{x}_0$ is any given point in the interior of $\Omega$. Thus, the decoupled subproblem for the pressure reads
\begin{equation}\label{SE-pre}
\begin{cases}
\nabla p = \bm{f} + \nu \Delta \bm{u} \quad \text{in}\ \Omega,
\\
p(\bm{x}_0) = 0.
\end{cases}
\end{equation}

Next, we extend our framework to the steady Navier-Stokes equations \eqref{NS-a}--\eqref{NS-c}. As in the Stokes case, applying the curl operator to \eqref{NS-a} eliminates the pressure term. The only additional term is the nonlinear advection term. By Theorem~\ref{thm1}, its curl admits the representation \eqref{NS2d-temp}. Combining \eqref{SE2d-vel} and \eqref{NS2d-temp} yields the decoupled subproblem for the velocity
\begin{equation}\label{NS2d-vel}
\begin{cases}
\nu\Delta^2 \phi - (\Curl \phi \cdot \nabla)\Delta \phi = \curl \bm{f} \quad &\text{in}\ \Omega,
\\
\phi = \frac{\partial \phi}{\partial  n} = 0 \quad &\text{on}\ \Gamma.
\end{cases}
\end{equation}
Similarly, the decoupled subproblem for the pressure reads
\begin{equation}\label{NS-pre}
\begin{cases}
\nabla p = \bm{f} + \nu \Delta \bm{u} - (\bm{u}\cdot \nabla) \bm{u} \quad \text{in}\ \Omega,
\\
p(\bm{x}_0) = 0.
\end{cases}
\end{equation}

Although formulation \eqref{NS2d-vel} has been studied in \cite{Guo-1997, Lequeurre-2020}, and Legendre spectral methods were further developed for two-dimensional Navier-Stokes problems in \cite{Guo-1998}, extending these techniques to the three-dimensional setting involves additional technical considerations and is not entirely straightforward. Moreover, \eqref{NS-pre} allows the recovery of a classical solution, in contrast to the weak solution frameworks commonly adopted in earlier works.

\subsection{The 3D case: vector potential representation}
It can be verified that the solution to \eqref{SE-a}--\eqref{SE-c} satisfies the assumptions of Lemma \ref{lem2} by the divergence theorem and the incompressibility constraint. We choose a vector potential $\bm{\phi}\in H(\Curl;\Omega)$ such that
\[
    \bm{u}=\Curl\bm{\phi}, \qquad
    \dive\bm{\phi}=0 \quad \text{in } \Omega, \qquad
    \bm{\phi}\cdot\bm{n}=0 \quad \text{on } \Gamma.
\]
For the following formal derivation, we additionally assume that $\bm{u}\in[H^3(\Omega)]^3$ and $\bm{\phi}\in[H^4(\Omega)]^3$.

Taking the {\bf curl} of the left side of \eqref{SE-a} yields
\[
\Curl  (-\nu \Delta \Curl \bm{\phi} + \nabla p)
= -\nu \Delta \Curl \Curl \bm{\phi} + \Curl \nabla p
=-\nu \Delta (\nabla (\nabla \cdot \bm{\phi}) - \Delta \bm{\phi}).
\]
Noticing that $\bm{\phi}$ is divergence-free, the above expression further simplifies to
\begin{equation}\label{SE3d-temp}
    \nu\Delta^2\bm{\phi} = \Curl \bm{f}.
\end{equation}
A direct computation confirms that $\dive \Curl \bm{\phi} = 0$, so the incompressibility constraint is satisfied identically. Combining \eqref{SE3d-temp}, Lemma \ref{lem2} and the boundary condition \eqref{SE-c}, we obtain the following decoupled subproblem for the velocity, expressed in terms of $\bm{\phi}$
\begin{equation}\label{SE3d-vel}
\begin{cases}
\nu\Delta^2\bm{\phi} = \Curl \bm{f} \quad &\text{in}\ \Omega,
\\
\dive \bm{\phi} = 0 \quad &\text{in}\ \Omega,
\\
\bm{\phi} \cdot \bm{n} = 0 \quad &\text{on}\ \Gamma,
\\
\Curl \bm{\phi} = \bm{0} \quad &\text{on}\ \Gamma.
\end{cases}
\end{equation}

After the determination of the velocity $\bm{u} = \Curl \bm{\phi}$, the pressure $p$ can be recovered by following the same procedure as in the two-dimensional case, namely by solving subproblem \eqref{SE-pre}. For brevity, the details are omitted here.

We now extend our framework to the steady three-dimensional Navier-Stokes problem. As discussed above, applying the curl operator to both sides of \eqref{NS-a} effectively decouples the velocity field from the pressure. Since the only difference between the Stokes and Navier-Stokes equations is the presence of the nonlinear advection term, however, Theorem \ref{thm2} gives the result of $\Curl((\bm{u}\cdot \nabla) \bm{u})$. Thus, the combination of \eqref{SE3d-vel} and \eqref{NS3d-temp} results in the decoupled subproblem for the velocity:
\begin{equation}\label{NS3d-vel}
\begin{cases}
\nu\Delta^2 \phi + (\Delta \bm{\phi} \cdot \nabla) \Curl \bm{\phi} - (\Curl \bm{\phi} \cdot \nabla)\Delta \bm{\phi} = \curl \bm{f} \quad &\text{in}\ \Omega,
\\
\dive \bm{\phi} = 0 \quad &\text{in}\ \Omega,
\\
\bm{\phi} \cdot \bm{n} = 0 \quad &\text{on}\ \Gamma,
\\
\Curl \bm{\phi} = \bm{0} \quad &\text{on}\ \Gamma.
\end{cases}
\end{equation}
Similarly, the decoupled pressure subproblem is formulated as \eqref{NS-pre}.

\begin{remark}
In the earlier works \cite{E-1997,Liu-2001}, E and Liu developed the finite difference method and finite element method to solve the Navier-Stokes based on the vorticity-vector potential formulations, which typically require the simultaneous approximation of multiple fields (potential and vorticity). Building upon these studies, our approach departs from the classical formulations by further exploiting properties of the curl operator on \eqref{NS3d-nonlinear}. This leads to the decoupled formulation \eqref{NS3d-vel}, which is characterized by a single unknown $\bm{\phi}$. On the other hand, this can be viewed as the extension of ``pure'' stream function formulation (Eqs. (1.5a)--(1.5b) in \cite{Lequeurre-2020} or \eqref{NS2d-vel} in this manuscript) in three dimensions.
\end{remark}

\begin{remark}
Our derivations are primarily formal and rely on strong regularity assumptions. Justifying these assumptions and establishing equivalence in appropriate function spaces require substantial additional analysis of regularity and boundary compatibility, even for scalar functions in two dimensions; see \cite{Grisvard1985} for more details. Such an analysis is beyond the scope of this work.
\end{remark}

It is noted that the decoupled subproblems for the Stokes and Navier--Stokes equations require higher regularity, but the challenges associated with high-order derivatives can be effectively addressed by machine learning methods through automatic differentiation and neural network parameterizations. The specific machine learning algorithm used to solve these decoupled subproblems is presented in the next section. Moreover, we mention that all the above derivations can also be applied to attack unsteady Stokes and Navier-Stokes problems in a straightforward way.

\section{The Neural networks basis method for incompressible flows}\label{sec4}
In this section, we first briefly review the TransNet, a special neural networks basis method for solving PDEs, as introduced in \cite{TransNet}. We emphasize that the deduction given below can be extended to any other neural networks basis method.

Mathematically, a single hidden layer fully-connected neural network can be formulated as
\begin{equation}\label{NN}
    u_{\rm NN}(\bm{x}) = \alpha_0 + \sum_{m=1}^M \alpha_m \sigma(\bm{w}^{\rm T}_m\bm{x} + b_m),
\end{equation}
where $\sigma(\cdot)$ denotes the nonlinear activation function. $\bm{w}_m \in \mathbb{R}^d$ and $b_m\in\mathbb{R}$ are the weights and bias of the $m$-th hidden neuron, respectively. The vector $\bm{\alpha} = [\alpha_0, \alpha_1, \cdots, \alpha_M]^{\rm T}$ are the parameters including the weights and bias of the output layer. Here, $M$ is the width of $u_{\rm NN}(\bm{x})$.

In neural networks basis method, the parameters $\bm{w}_m$ and $b_m$ are fixed for all $m$. Consequently, each hidden neuron acts as a fixed nonlinear basis function over the domain. We define the neural feature space as
\[
    V_M = \text{span}\{\psi_0, \psi_1, \cdots, \psi_M \},
\]
where $\psi_0 = 1$ and $\psi_m = \sigma(\bm{w}^{\rm T}_m\bm{x} + b_m)$ for $m=1,\cdots,M$. Since $u_{\rm NN}(\bm{x}) \in V_M$, the training process reduces to determining the optimal output coefficients $\bm{\alpha}$.

The initialization of hidden layer weights and biases remains a critical issue for neural network basis methods. TransNet addresses this challenge by adopting a special strategy that ensures hidden layer neurons are uniformly distributed within the unit ball. In the present work, we employ the same initialization strategy, which is summarized in the following steps.
\begin{enumerate}
\item Reparameterize each neural feature into the form
\[
    \sigma(\bm{w}^{\rm T}_m\bm{x} + b_m) = \sigma(\gamma_m(\bm{s}^{\rm T}_m\bm{x} + r_m)),
\]
where $\gamma_m = \|\bm{w}_m \|_2$, $\bm{s}_m = \bm{w}_m/\gamma_m$, $r_m = b_m/\gamma_m$. Then, one turn to initialize the parameters $\bm{s}_m$, $r_m$ and $\gamma_m$, $m=1, \cdots, M$.
\item Pre-set the parameters $\bm{s}_m$ and $r_m$ as follows.
\[
    \bm{s}_m = \frac{\bm{X}_m}{\| X_m\|_2}, \ r_m = U_m,  \quad 1 \leq m \leq M,
\]
where $\bm{X}_m$ are i.i.d. $d$-dimensional standard Gaussian random variables and $U_m$ are i.i.d. uniform random variables on $[0,1]$. This approach yields a set of uniformly distributed partition hyperplanes within the unit ball $B(0) = \{\bm{x}: \|\bm{x}\|_2\leq 1\}\subset \mathbb{R}^d$.
\item  Assume $\gamma = \gamma_m$ for $m=1,\cdots, M$ for simplicity. One will use realizations of Gaussian random fields as auxiliary functions to tune the shape parameter $\gamma$ without utilizing any information about PDEs. We refer to Algorithm 1 in \cite{TransNet} for details.
\end{enumerate}

While the preceding formulation describes a scalar-valued neural network, it can be naturally extended to vector-valued functions. For a function $\bm{\Phi}_{\rm NN}(\bm{x}): \mathbb{R}^d \to \mathbb{R}^k$, each component $\Phi_i$ is represented by a scalar neural network of the form \eqref{NN}. To enhance computational efficiency, all components share the same set of hidden-layer basis functions, that is, identical weights $\bm{w}_m$ and biases $b_m$, while maintaining independent output coefficients $\bm{\alpha}^{(i)}\in\mathbb{R}^{M+1}$ for each component.

In the remainder of this section, we detail the algorithmic implementation of the proposed framework for solving the steady Stokes and Navier-Stokes equations. In particular, we leverage the TransNet architecture to solve the decoupled, structure-preserving formulations derived in Section \ref{sec3}.

\subsection{The steady Stokes problem}
Generally speaking, we parameterize the stream function $\phi$ in two-dimensions or the vector potential $\bm{\phi}$ in three-dimensions, using the neural network $\phi_{\rm NN}$ or $\bm{\Phi}_{\rm NN}$, respectively. By imposing the subproblems \eqref{SE2d-vel} or \eqref{SE3d-vel} at collocation points, we obtain a least squares problem with respect to the output coefficients $\bm{\alpha}$ or $[\bm{\alpha}_1, \bm{\alpha}_2,\bm{\alpha}_3]^{\rm T}$. Once the velocity field is resolved, the corresponding pressure can be efficiently recovered through the decoupled subproblem \eqref{SE-pre} within the TransNet framework. In the following, we denote $\| \cdot \|_2$ as the vector $l^2$-norm in Euclidean space.

We first consider the two-dimensional Stokes problem. Since the single hidden layer fully-connected neural networks have sufficient expressive power for approximating low-dimensional functions \cite{Hornik-1989,Barron1993}, we approximate the stream function $\phi$ in \eqref{SE2d-vel} by the neural network $\phi_{\rm NN}$ as defined by \eqref{NN}. To discretize the problem, we choose the collocation points $\bm{X}^{in} =\{ \bm{x}^{in}_i \}_{i=1}^{I} \subset \Omega$ and $\bm{X}^{bd} =\{ \bm{x}^{bd}_j \}_{j=1}^J \subset \Gamma$. We then proceed to solve the resulting system
\begin{equation}\label{SE2d-dis}
\begin{cases}
\nu\Delta^2\sum_{m=0}^M \alpha_m \psi_m(\bm{x}^{in}_i) = \curl \bm{f}(\bm{x}^{in}_i) \quad i= 1,\cdots, I,
\\
\sum_{m=0}^M \alpha_m\psi_m(\bm{x}^{bd}_j) =0 \quad j= 1,\cdots, J,
\\
\frac{\partial}{\partial n}\sum_{m=0}^M \alpha_m \psi_m(\bm{x}^{bd}_j) = 0, \quad j= 1,\cdots, J.
\end{cases} 
\end{equation}
Following the linearity of the differential operator, we can reformulate \eqref{SE2d-dis} as
\begin{equation}\label{SE2d-matv}
\begin{bmatrix}{}
\nu \Delta^2 \bm{\Psi}_{in} \\
\bm{\Psi}_{bd} \\
\partial_n \bm{\Psi}_{bd}
\end{bmatrix}
\bm{\alpha}
=
\begin{bmatrix}{}
\curl \bm{F} \\
\bm{0} \\
\bm{0}
\end{bmatrix},
\end{equation}
where
\begin{align*}
& \Delta^2 \bm{\Psi}_{in} = [\Delta^2\psi_m(\bm{x}^{in}_i)] \in \mathbb{R}^{I\times(M+1)},
\quad
\curl \bm{F} = [\curl \bm{f}(\bm{x}^{in}_i)] \in \mathbb{R}^I
\\
& \bm{\Psi}_{bd} = [\psi_m(\bm{x}^{bd}_j)] \in \mathbb{R}^{J\times(M+1)},
\quad
\partial_n \bm{\Psi}_{bd} = [\partial_n \psi_m(\bm{x}^{bd}_j)] \in \mathbb{R}^{J\times(M+1)}.
\end{align*}
Typically, \eqref{SE2d-matv} is an over-determined system. Therefore, we solve the corresponding least-squares problem
\begin{equation}\label{SE2d-min_v}
\min_{\bm{\alpha}}
\left\Vert
\begin{bmatrix}{}
\nu \Delta^2 \bm{\Psi}_{in} \\
\bm{\Psi}_{bd} \\
\partial_n \bm{\Psi}_{bd}
\end{bmatrix}
\bm{\alpha}
-
\begin{bmatrix}{}
\curl \bm{F} \\
\bm{0} \\
\bm{0}
\end{bmatrix}
\right\Vert_2^2,
\end{equation}
which can be solved using the linear solvers provided by NumPy. If we denote $\bm{\alpha}^*$ as the minimizer of \eqref{SE2d-min_v}, then the approximate velocity $\bm{u}_{\rm NN}$ is given by
\[
\bm{u}_{\rm NN} = [u_{1}^*, u_2^*]^{\rm T}
=
\begin{bmatrix}{}
\sum_{m=0}^M \alpha_m^* \partial_y\psi_m \\
-\sum_{m=0}^M \alpha_m^* \partial_x\psi_m
\end{bmatrix}.
\]
Next, we find the approximate pressure within the TransNet framework. Suppose $p \approx \sum_{m=0}^M \beta_m\psi_m$. Substituting this into \eqref{SE-pre} yields
\[\begin{cases}
\nabla \sum_{m=0}^M \beta_m\psi_m(\bm{x}^{in}_i) = \bm{f}(\bm{x}^{in}_i) + \nu \Delta \bm{u}^*(\bm{x}^{in}_i) \quad i= 1,\cdots, I,
\\
\sum_{m=0}^M \beta_m\psi_m(\bm{x}_0) = 0,
\end{cases}\]
which can be reformulated in the form of matrix
\begin{equation}\label{SE2d-matp}
\begin{bmatrix}{}
\partial_x \bm{\Psi}_{in} \\
\partial_y \bm{\Psi}_{in} \\
\bm{\Psi}_{0}^{\rm T}
\end{bmatrix}
\bm{\beta}
=
\begin{bmatrix}{}
\tilde{\bm{F}}_1  \\
\tilde{\bm{F}}_2  \\
{0}
\end{bmatrix},
\end{equation}
where
\begin{align*}
& \partial_x \bm{\Psi}_{in} = [\partial_x\psi_m(\bm{x}^{in}_i)] \in \mathbb{R}^{I\times(M+1)},
\quad
\tilde{\bm{F}}_1 = [f_1(\bm{x}^{in}_i) + \nu\Delta u_1^*(\bm{x}^{in}_i)] \in \mathbb{R}^{I},
\\
& \partial_y \bm{\Psi}_{in} = [\partial_y\psi_m(\bm{x}^{in}_i)] \in \mathbb{R}^{I\times(M+1)},
\quad
\tilde{\bm{F}}_2 = [f_2(\bm{x}^{in}_i) + \nu\Delta u_2^*(\bm{x}^{in}_i)] \in \mathbb{R}^{I},
\\
&\bm{\Psi}_{0} = [\psi_m(\bm{x}_0)] \in \mathbb{R}^{M+1}.
\end{align*}
Since \eqref{SE2d-matp} is also an over-determined system, we turn to solve another least-squares problem
\begin{equation}\label{SE2d-min_p}
\min_{\bm{\beta}}
\left\Vert
\begin{bmatrix}{}
\partial_x \bm{\Psi}_{in} \\
\partial_y \bm{\Psi}_{in} \\
\bm{\Psi}_{0}^{\rm T}
\end{bmatrix}
\bm{\beta}
-
\begin{bmatrix}{}
\tilde{\bm{F}}_1  \\
\tilde{\bm{F}}_2  \\
{0}
\end{bmatrix}
\right\Vert_2^2.
\end{equation}

We now turn to the three-dimensional Stokes problem. The potential function $\bm{\phi}$ in \eqref{SE3d-vel} is approximated by the neural network
\[
\bm{\phi}_{\rm NN}=[\phi_{1,\rm NN}, \phi_{2,\rm NN}, \phi_{3,\rm NN}]^{\rm T} = \left[\sum_{m=0}^M\alpha_{1,m}\psi_m, \sum_{m=0}^M\alpha_{2,m}\psi_m, \sum_{m=0}^M\alpha_{3,m}\psi_m \right]^{\rm T} ,
\]
and discretize \eqref{SE3d-vel} on collocation points $\bm{X}^{in} =\{ \bm{x}^{in}_i \}_{i=1}^{I} \subset \Omega$ and $\bm{X}^{bd} =\{ \bm{x}^{bd}_j \}_{j=1}^J \subset \Gamma$. This leads to the following system
\[\begin{cases}
\nu\Delta^2\bm{\phi}_{\rm NN}(\bm{x}^{in}_i) = \Curl \bm{f}(\bm{x}^{in}_i)  \quad i= 1,\cdots, I,
\\
\dive \bm{\phi}_{\rm NN}(\bm{x}^{in}_i) = 0 \quad i= 1,\cdots, I,
\\
\bm{\phi}_{\rm NN}(\bm{x}^{bd}_j) \cdot \bm{n}(\bm{x}^{bd}_j) = 0 \quad j= 1,\cdots, J,
\\
\Curl \bm{\phi}_{\rm NN}(\bm{x}^{bd}_j) = \bm{0} \quad j= 1,\cdots, J,
\end{cases}\]
reformulating it into the form of block matrix form
\begin{equation}\label{SE3d-matv}
\begin{bmatrix}{}
\nu \Delta^2 \bm{\Psi}_{in} &  & \\
& \nu \Delta^2 \bm{\Psi}_{in} & \\
& & \nu \Delta^2 \bm{\Psi}_{in} \\
\partial_x\bm{\Psi}_{in} & \partial_y\bm{\Psi}_{in} & \partial_z\bm{\Psi}_{in} \\
\bm{\Psi}_{bd}\bm{n}_1 & \bm{\Psi}_{bd}\bm{n}_2 & \bm{\Psi}_{bd}\bm{n}_3 \\
  &-\partial_z \bm{\Psi}_{bd}  &\partial_y \bm{\Psi}_{bd} \\
\partial_z \bm{\Psi}_{bd} &  &-\partial_x \bm{\Psi}_{bd}  \\
-\partial_y \bm{\Psi}_{bd}  &\partial_x \bm{\Psi}_{bd}  &
\end{bmatrix}
\begin{bmatrix}{}
\bm{\alpha}_{1} \\
\bm{\alpha}_{2} \\
\bm{\alpha}_{3}
\end{bmatrix}
=
\begin{bmatrix}{}
\partial_y \bm{F}_3-\partial_z \bm{F}_2 \\
\partial_z \bm{F}_1-\partial_x \bm{F}_3 \\
\partial_x \bm{F}_2-\partial_y \bm{F}_1 \\
\bm{0} \\
\bm{0} \\
\bm{0} \\
\bm{0} \\
\bm{0} \\
\end{bmatrix},
\end{equation}
where
\begin{align*}
&\Delta^2 \bm{\Psi}_{in} = [\Delta^2\psi_m(\bm{x}^{in}_i)] \in \mathbb{R}^{I\times(M+1)},
\quad
\partial_x\bm{\Psi}_{in} = [\partial_x\psi_m(\bm{x}^{in}_i)] \in \mathbb{R}^{I\times(M+1)},
\\
&\partial_y\bm{\Psi}_{in} = [\partial_y\psi_m(\bm{x}^{in}_i)] \in \mathbb{R}^{I\times(M+1)},
\quad
\partial_z\bm{\Psi}_{in} = [\partial_z\psi_m(\bm{x}^{in}_i)] \in \mathbb{R}^{I\times(M+1)},
\\
&\bm{\Psi}_{bd}\bm{n}_1 = [\psi_m(\bm{x}^{bd}_j)n_1(\bm{x}^{bd}_j)] \in \mathbb{R}^{J\times(M+1)},
\quad
\bm{\Psi}_{bd}\bm{n}_2 = [\psi_m(\bm{x}^{bd}_j)n_2(\bm{x}^{bd}_j)] \in \mathbb{R}^{J\times(M+1)},
\\
&\bm{\Psi}_{bd}\bm{n}_3 = [\psi_m(\bm{x}^{bd}_j)n_3(\bm{x}^{bd}_j)] \in \mathbb{R}^{J\times(M+1)},
\quad
\partial_x\bm{\Psi}_{bd} = [\partial_x\psi_m(\bm{x}^{bd}_j)] \in \mathbb{R}^{J\times(M+1)},
\\
&\partial_y\bm{\Psi}_{bd} = [\partial_y\psi_m(\bm{x}^{bd}_j)] \in \mathbb{R}^{J\times(M+1)},
\quad
\partial_z\bm{\Psi}_{bd} = [\partial_z\psi_m(\bm{x}^{bd}_j)] \in \mathbb{R}^{J\times(M+1)},
\\
&\partial_y \bm{F}_3-\partial_z \bm{F}_2 =[\partial_yf_3(\bm{x}^{in}_i)-\partial_zf_2(\bm{x}^{in}_i)] \in \mathbb{R}^{I},\\
&\partial_z \bm{F}_1-\partial_x \bm{F}_3 =[\partial_zf_1(\bm{x}^{in}_i)-\partial_xf_3(\bm{x}^{in}_i)] \in \mathbb{R}^{I},\\
&\partial_x \bm{F}_2-\partial_y \bm{F}_1 =[\partial_xf_2(\bm{x}^{in}_i)-\partial_yf_1(\bm{x}^{in}_i)] \in \mathbb{R}^{I}.
\end{align*}
Obviously, \eqref{SE3d-matv} is an over-determined system, so we turn to solve the least squares problem as follows
\begin{equation}\label{SE3d-min_v}
\min_{\bm{\alpha}}
\left\Vert
\begin{bmatrix}{}
\nu \Delta^2 \bm{\Psi}_{in} &  & \\
& \nu \Delta^2 \bm{\Psi}_{in} & \\
& & \nu \Delta^2 \bm{\Psi}_{in} \\
\partial_x\bm{\Psi}_{in} & \partial_y\bm{\Psi}_{in} & \partial_z\bm{\Psi}_{in} \\
\bm{\Psi}_{bd}\bm{n}_1 & \bm{\Psi}_{bd}\bm{n}_2 & \bm{\Psi}_{bd}\bm{n}_3 \\
  &-\partial_z \bm{\Psi}_{bd}  &\partial_y \bm{\Psi}_{bd} \\
\partial_z \bm{\Psi}_{bd} &  &-\partial_x \bm{\Psi}_{bd}  \\
-\partial_y \bm{\Psi}_{bd}  &\partial_x \bm{\Psi}_{bd}  &
\end{bmatrix}
\begin{bmatrix}{}
\bm{\alpha}_{1} \\
\bm{\alpha}_{2} \\
\bm{\alpha}_{3}
\end{bmatrix}
-
\begin{bmatrix}{}
\partial_y \bm{F}_3-\partial_z \bm{F}_2 \\
\partial_z \bm{F}_1-\partial_x \bm{F}_3 \\
\partial_x \bm{F}_2-\partial_y \bm{F}_1 \\
\bm{0} \\
\bm{0} \\
\bm{0} \\
\bm{0} \\
\bm{0} \\
\end{bmatrix}
\right\Vert_2^2.
\end{equation}
Once the optimal parameters $[\bm{\alpha}_1^*, \bm{\alpha}_2^*, \bm{\alpha}_3^*]$ are determined, we can obtain the velocity field as
\[
    \bm{u}_{\rm NN} = [u_{1}^*, u_2^*, u_3^*]^{\rm T}
=
\begin{bmatrix}{}
\sum_{m=0}^M (\alpha_{3,m}^* \partial_y\psi_m - \alpha_{2,m}^* \partial_z\psi_m) \\
\sum_{m=0}^M (\alpha_{1,m}^* \partial_z\psi_m - \alpha_{3,m}^* \partial_x\psi_m) \\
\sum_{m=0}^M (\alpha_{2,m}^* \partial_x\psi_m - \alpha_{1,m}^* \partial_y\psi_m) \\
\end{bmatrix}.
\]
Similarly, the subproblem for the pressure \eqref{SE-pre} can be reformulated as
\begin{equation}\label{SE3d-matp}
\begin{bmatrix}{}
\partial_x \bm{\Psi}_{in} \\
\partial_y \bm{\Psi}_{in} \\
\partial_z \bm{\Psi}_{in} \\
\bm{\Psi}_{0}^{\rm T}
\end{bmatrix}
\bm{\beta}
=
\begin{bmatrix}{}
\tilde{\bm{F}}_1  \\
\tilde{\bm{F}}_2  \\
\tilde{\bm{F}}_3  \\
{0}
\end{bmatrix},
\end{equation}
where $\partial_x \bm{\Psi}_{in}$, $\partial_y \bm{\Psi}_{in}$, $\partial_z \bm{\Psi}_{in}$ are defined in the same way as in \eqref{SE3d-matv}, and
\begin{align*}
&\tilde{\bm{F}}_1 = [f_1(\bm{x}^{in}_i) + \nu\Delta u_1^*(\bm{x}^{in}_i)] \in \mathbb{R}^{I},
\\
&\tilde{\bm{F}}_2 = [f_2(\bm{x}^{in}_i) + \nu\Delta u_2^*(\bm{x}^{in}_i)] \in \mathbb{R}^{I},
\\
& \tilde{\bm{F}}_3 = [f_3(\bm{x}^{in}_i) + \nu\Delta u_3^*(\bm{x}^{in}_i)] \in \mathbb{R}^{I},\\
&\bm{\Psi}_{0} = [\psi_m(\bm{x}_0)] \in \mathbb{R}^{M+1}.
\end{align*}

\subsection{The steady Navier-Stokes problem}
For the steady Navier-Stokes problem, the presence of the nonlinear advection term prevents a direct formulation as a linear least-squares problem. To address this difficulty, we first linearize the governing equations at the continuous level using the Gauss-Newton method, which yields a sequence of linear subproblems. Additional linearization strategies are provided in \cite{Liu-2023}. Subsequently, the neural network basis method is employed to solve these linearized equations. This strategy is referred to as the Newton-LLSQ framework, as proposed in \cite{Dong-2021a}. By approximating the stream function in two dimensions or the vector potential in three dimensions and enforcing the interior and boundary conditions at collocation points, each iteration reduces to solving a linear least-squares system for the coefficients $\bm{\alpha}$. This approach is equivalent to first deriving a discretized nonlinear system and then applying the Gauss-Newton method to solve it \cite{Huang-2025}.

To address the nonlinearity of the Navier-Stokes equations, we follow the idea in \cite{Huang-2025} and linearize the convection term via variational derivatives. For the nonlinear advection operator $f(\bm{u}) = (\bm{u}\cdot \nabla) \bm{u}$, the G{\^a}teaux derivative (see \cite{Chang2005}) of $f$ at $\bm{u}$ in any direction $\bm{v}$ is defined as
\[
    \mathcal{D}f(\bm{u})[\bm{v}]
    = \lim_{\varepsilon \to 0}\frac{f(\bm{u} + \varepsilon \bm{v}) - f(\bm{u})}{\varepsilon}.
\]
A straightforward expansion yields
\[
    \mathcal{D}f(\bm{u})[\bm{v}]
    = \lim_{\varepsilon \to 0}\frac{(\bm{u}\cdot \nabla) \varepsilon\bm{v} + (\varepsilon\bm{v} \cdot \nabla) \bm{u}+ (\varepsilon\bm{v} \cdot \nabla) \varepsilon\bm{v}}{\varepsilon}
    = (\bm{u} \cdot \nabla) \bm{v} + (\bm{v} \cdot \nabla) \bm{u}.
\]{}
It is worth emphasizing that the mapping
$
    \bm{v} \;\mapsto\; \mathcal{D}f(\bm{u})[\bm{v}]
$
is linear and continuous with respect to $\bm{v}$ in Sobolev spaces. Therefore, the G{\^a}teaux derivative coincides with the Fr{\'e}chet derivative of $f$ at $\bm{u}$ \cite{Chang2005}. Based on this result, the standard Newton linearization of the convection term at the $(k+1)$-th iteration reads
\begin{align}
    f(\bm{u}^{(k+1)})
    &\approx f(\bm{u}^{(k)}) + \mathcal{D}f(\bm{u}^{(k)})[\delta\bm{u}] \nonumber \\
    &= (\bm{u}^{(k)} \cdot \nabla)\bm{u}^{(k)} + (\bm{u}^{(k)} \cdot \nabla)\delta\bm{u} + (\delta\bm{u} \cdot \nabla) \bm{u}^{(k)} \nonumber
    \\
    &=(\bm{u}^{(k)} \cdot \nabla) \bm{u}^{(k+1)} + (\bm{u}^{(k+1)} \cdot \nabla) \bm{u}^{(k)} - (\bm{u}^{(k)} \cdot \nabla) \bm{u}^{(k)}, \label{k-nonlinear}
\end{align}
where $\delta \bm{u} = \bm{u}^{(k+1)} - \bm{u}^{(k)}$.

Specifically, in the two-dimensional case, we employ the stream function representation $\bm{u} = \Curl \phi$. Substituting the linearized term \eqref{k-nonlinear} into the decoupled velocity subproblem \eqref{NS2d-vel}, we obtain the linearized problem at each iteration step as follows:
\begin{equation}\label{2d-linear}
\nu \Delta ^2 \phi^{(k+1)} - (\Curl \phi^{(k)} \cdot \nabla) \Delta \phi^{(k+1)} - (\Curl \phi^{(k+1)} \cdot \nabla) \Delta \phi^{(k)} = \curl\bm{f} - (\Curl \phi^{(k)} \cdot \nabla) \Delta \phi^{(k)}.
\end{equation}
Similarly, in the three-dimensional case, where $\bm{u} = \Curl \bm{\phi}$, substituting the linearized term \eqref{k-nonlinear} into the corresponding subproblem \eqref{NS3d-vel} yields
\begin{align}\label{3d-linear}
&\nu \Delta ^2 \bm{\phi}^{(k+1)} - (\Curl \bm{\phi}^{(k)} \cdot \nabla) \Delta \bm{\phi}^{(k+1)} +(\Delta \bm{\phi}^{(k)} \cdot \nabla)\Curl \bm{\phi}^{(k+1)} \nonumber \\
&\qquad \qquad \quad  - (\Curl \bm{\phi}^{(k+1)} \cdot \nabla) \Delta \bm{\phi}^{(k)} + (\Delta \bm{\phi}^{(k+1)} \cdot \nabla)\Curl \bm{\phi}^{(k)}  \nonumber \\
= & \Curl\bm{f} - (\Curl \bm{\phi}^{(k)} \cdot \nabla) \Delta \bm{\phi}^{(k)} + (\Delta \bm{\phi}^{(k)} \cdot \nabla)\Curl \bm{\phi}^{(k)}.
\end{align}

Next, we implement the Decoupled-DFNN framework for the Navier-Stokes equations. In the two-dimensional case, we first approximate the stream function $\phi$ using the neural network $\phi_{\rm NN}$. Substituting the \eqref{2d-linear} into the \eqref{NS2d-vel} and selecting collocation points $\bm{X}^{in} =\{ \bm{x}^{in}_i \}_{i=1}^{I} \subset \Omega$ and $\bm{X}^{bd} =\{ \bm{x}^{bd}_j \}_{j=1}^J \subset \Gamma$, the problem is reduced to determining the optimal coefficients $\bm{\alpha}^{(k+1)}$ that satisfy the following system
\[
\begin{cases}
&\quad \nu \Delta ^2 \sum_{m=0}^M\alpha_m^{(k+1)}\psi_m(\bm{x}_i^{in}) - (\Curl \sum_{m=0}^M\alpha_m^{(k)}\psi_m(\bm{x}_i^{in}) \cdot \nabla) \Delta \sum_{m=0}^M\alpha_m^{(k+1)}\psi_m(\bm{x}_i^{in}) \\
& \qquad \qquad \qquad \qquad \qquad \qquad \ -(\Curl \sum_{m=0}^M\alpha_m^{(k+1)}\psi_m(\bm{x}_i^{in}) \cdot \nabla) \Delta \sum_{m=0}^M\alpha_m^{(k)}\psi_m(\bm{x}_i^{in}) \\
&= \curl\bm{f} - (\Curl \sum_{m=0}^M\alpha_m^{(k)}\psi_m(\bm{x}_i^{in}) \cdot \nabla) \Delta \sum_{m=0}^M\alpha_m^{(k)}\psi_m(\bm{x}_i^{in}), \quad i=1\cdots,I,
\\
&\sum_{m=0}^M\alpha_m^{(k+1)}\psi_m(\bm{x}_j^{bd}) = 0, \quad j=1\cdots,J,
\\
&\frac{\partial}{\partial n} \sum_{m=0}^M\alpha_m^{(k+1)}\psi_m(\bm{x}_j^{bd}) = 0 , \quad j=1\cdots,J.
\end{cases}
\]
To represent this system in matrix form, denote $\odot$ the standard Hadamard product for matrices of the same dimensions, and extend this operator to the case of a vector $\bm{a} \in \mathbb{R}^m$ and a matrix $\bm{B} \in \mathbb{R}^{m \times n}$. Specifically, for $\bm{C} = \bm{a} \odot \bm{B}$, the entries are defined by $c_{ij} = a_i b_{ij}$. This operation can be expressed in terms of standard matrix multiplication as
\[
    \bm{C} = \mbox{diag}(\bm{a})\bm{B},
\]
where $\text{diag}(\bm{a})$ is the $m \times m$ diagonal matrix with the elements of $\bm{a}$ on its main diagonal. Using this notation, we reformulate the system into the following block matrix form
\begin{equation}\label{NS2d-matv}
\begin{bmatrix}{}
\nu \Delta^2 \bm{\Psi}_{in} + \bm{NL} \\
\bm{\Psi}_{bd} \\
\partial_n \bm{\Psi}_{bd}
\end{bmatrix}
\bm{\alpha}^{(k+1)}
=
\begin{bmatrix}{}
\tilde{\bm{F}} \\
\bm{0} \\
\bm{0}
\end{bmatrix},
\end{equation}
where
\begin{align*}
& \bm{NL} = -\bm{U}_1^{(k)}\odot\partial_x\Delta\bm{\Psi}_{in} - \bm{U}_2^{(k)}\odot\partial_y\Delta\bm{\Psi}_{in} + \Delta \bm{U}_2^{(k)}\odot\partial_y\bm{\Psi}_{in} + \Delta \bm{U}_1^{(k)}\odot\partial_x\bm{\Psi}_{in},
\\
&\tilde{\bm{F}} = \curl \bm{F} + \bm{U}_1^{(k)}\odot\Delta \bm{U}_2^{(k)} - \bm{U}_2^{(k)}\odot\Delta \bm{U}_1^{(k)},
\quad
\curl \bm{F} = [\curl \bm{f}(\bm{x}^{in}_i)] \in \mathbb{R}^I,
\\
& \bm{U}_1^{k} = \partial_y \bm{\Psi}_{in} \cdot \bm{\alpha}^{(k)},
\qquad
\Delta \bm{U}_1^{k} = \partial_y \Delta \bm{\Psi}_{in} \cdot \bm{\alpha}^{(k)},
\\
& \bm{U}_2^{k} = -\partial_x \bm{\Psi}_{in}\cdot \bm{\alpha}^{(k)},
\quad
\Delta \bm{U}_2^{k} = -\partial_x \Delta \bm{\Psi}_{in}\cdot \bm{\alpha}^{(k)},
\\
& \Delta^2 \bm{\Psi}_{in} = [\Delta^2\psi_m(\bm{x}^{in}_i)] \in \mathbb{R}^{I\times(M+1)},
\\
&\partial_x \Delta \bm{\Psi}_{in} = [\partial_x \Delta \psi_m(\bm{x}^{in}_i)] \in \mathbb{R}^{I\times(M+1)},
\quad
\partial_y \Delta \bm{\Psi}_{in} = [\partial_y \Delta \psi_m(\bm{x}^{in}_i)] \in \mathbb{R}^{I\times(M+1)},
\\
&\partial_x \bm{\Psi}_{in} = [\partial_x \psi_m(\bm{x}^{in}_i)] \in \mathbb{R}^{I\times(M+1)},
\quad
\partial_y \bm{\Psi}_{in} = [\partial_y \psi_m(\bm{x}^{in}_i)] \in \mathbb{R}^{I\times(M+1)},
\\
& \bm{\Psi}_{bd} = [\psi_m(\bm{x}^{bd}_j)] \in \mathbb{R}^{J\times(M+1)},
\quad
\partial_n \bm{\Psi}_{bd} = [\partial_n \psi_m(\bm{x}^{bd}_j)] \in \mathbb{R}^{J\times(M+1)}.
\end{align*}
Since \eqref{NS2d-matv} is an over-determined system, we solve the corresponding least-squares problem
\begin{equation}\label{NS2d-min_v}
\min_{\bm{\alpha}^{(k+1)}}
\left\Vert
\begin{bmatrix}{}
\nu \Delta^2 \bm{\Psi}_{in} +\bm{NL} \\
\bm{\Psi}_{bd} \\
\partial_n \bm{\Psi}_{bd}
\end{bmatrix}
\bm{\alpha} ^{(k+1)}
-
\begin{bmatrix}{}
\tilde{\bm{F}} \\
\bm{0} \\
\bm{0}
\end{bmatrix}
\right\Vert_2^2.
\end{equation}
Once the minimizer $\bm{\alpha}^*$ is obtained, the solution to \eqref{NS-pre} can be approximated as
\begin{equation}\label{NS2d-matp}
\begin{bmatrix}{}
\partial_x \bm{\Psi}_{in} \\
\partial_y \bm{\Psi}_{in} \\
\bm{\Psi}_{0}^{\rm T}
\end{bmatrix}
\bm{\beta}
=
\begin{bmatrix}{}
\tilde{\bm{F}}_1  \\
\tilde{\bm{F}}_2  \\
{0}
\end{bmatrix},
\end{equation}
where
\begin{align*}
&\tilde{\bm{F}}_1 = [f_1(\bm{x}^{in}_i) + \nu\Delta u_1^*(\bm{x}^{in}_i)-(u_1\partial_x u_1 + u_2\partial_y u_1)] \in \mathbb{R}^{I},
\\
&\tilde{\bm{F}}_2 = [f_2(\bm{x}^{in}_i) + \nu\Delta u_2^*(\bm{x}^{in}_i)-(u_1\partial_x u_2 + u_2\partial_y u_2)] \in \mathbb{R}^{I},
\\
&\bm{\Psi}_{0} = [\psi_m(\bm{x}_0)] \in \mathbb{R}^{M+1}.
\end{align*}
Similarly, the system \eqref{NS2d-matp} is reformulated as a least-squares problem, from which we compute the minimizer $\bm{\beta}^*$. The pressure is then approximated by $p_{\rm NN} = \sum_{m=0}^M \beta^*_m\psi_m(\bm{x})$.

For the three-dimensional case, we approximate the potential function $\bm{\phi}$ using the neural network
\[
\bm{\phi}_{\rm NN}=[\phi_{1,\rm NN}, \phi_{2,\rm NN}, \phi_{3,\rm NN}]^{\rm T} = \left[\sum_{m=0}^M\alpha_{1,m}\psi_m, \sum_{m=0}^M\alpha_{2,m}\psi_m, \sum_{m=0}^M\alpha_{3,m}\psi_m \right]^{\rm T}.
\]
Substituting the \eqref{3d-linear} into the \eqref{NS3d-vel} and discretizing the resulting equations at the collocation points $\bm{X}^{in} =\{ \bm{x}^{in}_i \}_{i=1}^{I} \subset \Omega$ and $\bm{X}^{bd} =\{ \bm{x}^{bd}_j \}_{j=1}^J \subset \Gamma$, we obtain
\[\begin{cases}
&\nu \Delta ^2 \bm{\phi}^{(k+1)}_{\rm NN}(\bm{x}^{in}_i) - (\Curl \bm{\phi}^{(k)}_{\rm NN}(\bm{x}^{in}_i) \cdot \nabla) \Delta \bm{\phi}^{(k+1)}_{\rm NN}(\bm{x}^{in}_i) +(\Delta \bm{\phi}^{(k)}_{\rm NN}(\bm{x}^{in}_i) \cdot \nabla)\Curl \bm{\phi}^{(k+1)}_{\rm NN}(\bm{x}^{in}_i) \nonumber \\
&\qquad \qquad \qquad  - (\Curl \bm{\phi}^{(k+1)}_{\rm NN}(\bm{x}^{in}_i) \cdot \nabla) \Delta \bm{\phi}^{(k)}_{\rm NN}(\bm{x}^{in}_i) + (\Delta \bm{\phi}^{(k+1)}_{\rm NN}(\bm{x}^{in}_i) \cdot \nabla)\Curl \bm{\phi}^{(k)}_{\rm NN}(\bm{x}^{in}_i) \nonumber \\
= & \Curl\bm{f} - (\Curl \bm{\phi}^{(k)}_{\rm NN}(\bm{x}^{in}_i) \cdot \nabla) \Delta \bm{\phi}^{(k)}_{\rm NN}(\bm{x}^{in}_i) + (\Delta \bm{\phi}^{(k)}_{\rm NN}(\bm{x}^{in}_i) \cdot \nabla)\Curl \bm{\phi}^{(k)}_{\rm NN}(\bm{x}^{in}_i),  \quad i= 1,\cdots, I,
\\
&\dive \bm{\phi}^{(k+1)}_{\rm NN}(\bm{x}^{in}_i) = 0 \quad i= 1,\cdots, I,
\\
&\bm{\phi}_{\rm NN}^{(k+1)}(\bm{x}^{bd}_j) \cdot \bm{n}(\bm{x}^{bd}_j) = 0 \quad j= 1,\cdots, J,
\\
&\Curl \bm{\phi}^{(k+1)}_{\rm NN}(\bm{x}^{bd}_j) = \bm{0} \quad j= 1,\cdots, J,
\end{cases}\]
which can be further reformulated in block matrix form
\begin{equation}\label{NS3d-matv}
\begin{bmatrix}{}
\nu \Delta^2 \bm{\Psi}_{in} + \bm{N}_{11}   & \bm{N}_{12}  &\bm{N}_{13}  \\
\bm{N}_{21}   & \nu \Delta^2\bm{\Psi}_{in}+\bm{N}_{22}    & \bm{N}_{23} \\
\bm{N}_{31}   & \bm{N}_{32}   & \nu \Delta^2 \bm{\Psi}_{in} +\bm{N}_{33} \\
\partial_x\bm{\Psi}_{in} & \partial_y\bm{\Psi}_{in} & \partial_z\bm{\Psi}_{in} \\
\bm{\Psi}_{bd}\bm{n}_1 & \bm{\Psi}_{bd}\bm{n}_2 & \bm{\Psi}_{bd}\bm{n}_3 \\
  &-\partial_z \bm{\Psi}_{bd}  &\partial_y \bm{\Psi}_{bd} \\
\partial_z \bm{\Psi}_{bd} &  &-\partial_x \bm{\Psi}_{bd}  \\
-\partial_y \bm{\Psi}_{bd}  &\partial_x \bm{\Psi}_{bd}  &
\end{bmatrix}
\begin{bmatrix}{}
\bm{\alpha}_{1}^{(k+1)} \\
\bm{\alpha}_{2}^{(k+1)} \\
\bm{\alpha}_{3}^{(k+1)}
\end{bmatrix}
=
\begin{bmatrix}{}
\tilde{\bm{F}}_1 \\
\tilde{\bm{F}}_2 \\
\tilde{\bm{F}}_3 \\
\bm{0} \\
\bm{0} \\
\bm{0} \\
\bm{0} \\
\bm{0} \\
\end{bmatrix},
\end{equation}
where
\begin{align*}
\bm{N}_{11} =
&-(\bm{U}_1^{(k)}\odot\partial_x\Delta \bm{\Psi}_{in} + \bm{U}_2^{(k)}\odot\partial_y\Delta \bm{\Psi}_{in} + \bm{U}_3^{(k)}\odot\partial_z\Delta \bm{\Psi}_{in})\\
&- (\partial_y\Delta\bm{\Psi}_{in}\cdot\bm{\alpha}_1^{(k)}\odot\partial_z\bm{\Psi}_{in} - \partial_z\Delta\bm{\Psi}_{in}\cdot \bm{\alpha}_1^{(k)}\odot\partial_y\bm{\Psi}_{in}) + \partial_x\bm{U}_1^{(k)}\odot \Delta\bm{\Psi}_{in},
\\
\bm{N}_{12} =
& -\Delta\bm{\Psi}_{in}\cdot \bm{\alpha}_1^{(k)} \odot\partial_{xz}\bm{\Psi}_{in} - \Delta\bm{\Psi}_{in}\cdot \bm{\alpha}_2^{(k)} \odot\partial_{yz}\bm{\Psi}_{in} -\Delta\bm{\Psi}_{in}\cdot \bm{\alpha}_3^{(k)} \odot\partial_{zz}\bm{\Psi}_{in} \\
& - (\partial_z\Delta\bm{\Psi}_{in}\cdot \bm{\alpha}_1^{(k)}\odot\partial_x\bm{\Psi}_{in} - \partial_x\Delta\bm{\Psi}_{in}\cdot\bm{\alpha}_1^{(k)}\odot\partial_z\bm{\Psi}_{in} ) + \partial_y\bm{U}_1^{(k)}\odot \Delta\bm{\Psi}_{in},
\\
\bm{N}_{13} =
& \Delta\bm{\Psi}_{in}\cdot \bm{\alpha}_1^{(k)} \odot\partial_{xy}\bm{\Psi}_{in} + \Delta\bm{\Psi}_{in}\cdot \bm{\alpha}_2^{(k)} \odot\partial_{yy}\bm{\Psi}_{in}  +\Delta\bm{\Psi}_{in}\cdot \bm{\alpha}_3^{(k)} \odot\partial_{zy}\bm{\Psi}_{in} \\
& - (\partial_x\Delta\bm{\Psi}_{in}\cdot \bm{\alpha}_1^{(k)}\odot\partial_y\bm{\Psi}_{in} - \partial_y\Delta\bm{\Psi}_{in}\cdot\bm{\alpha}_1^{(k)}\odot\partial_x\bm{\Psi}_{in} ) + \partial_z\bm{U}_1^{(k)}\odot \Delta\bm{\Psi}_{in},
\\
\bm{N}_{21} =
& \Delta\bm{\Psi}_{in}\cdot \bm{\alpha}_1^{(k)} \odot\partial_{xz}\bm{\Psi}_{in} + \Delta\bm{\Psi}_{in}\cdot \bm{\alpha}_2^{(k)} \odot\partial_{yz}\bm{\Psi}_{in} +\Delta\bm{\Psi}_{in}\cdot \bm{\alpha}_3^{(k)} \odot\partial_{zz}\bm{\Psi}_{in} \\
& - (\partial_y\Delta\bm{\Psi}_{in}\cdot \bm{\alpha}_2^{(k)}\odot\partial_z\bm{\Psi}_{in} - \partial_z\Delta\bm{\Psi}_{in}\cdot\bm{\alpha}_2^{(k)}\odot\partial_y\bm{\Psi}_{in} ) + \partial_x\bm{U}_2^{(k)}\odot \Delta\bm{\Psi}_{in},
\\
\bm{N}_{22} =
&-(\bm{U}_1^{(k)}\odot\partial_x\Delta \bm{\Psi}_{in} + \bm{U}_2^{(k)}\odot\partial_y\Delta \bm{\Psi}_{in} + \bm{U}_3^{(k)}\odot\partial_z\Delta \bm{\Psi}_{in})\\
&- (\partial_z\Delta\bm{\Psi}_{in}\cdot\bm{\alpha}_2^{(k)}\odot\partial_x\bm{\Psi}_{in} - \partial_x\Delta\bm{\Psi}_{in}\cdot \bm{\alpha}_2^{(k)}\odot\partial_z\bm{\Psi}_{in}) + \partial_y\bm{U}_2^{(k)}\odot \Delta\bm{\Psi}_{in}
\\
\bm{N}_{23} =
& -\Delta\bm{\Psi}_{in}\cdot \bm{\alpha}_1^{(k)} \odot\partial_{xx}\bm{\Psi}_{in} - \Delta\bm{\Psi}_{in}\cdot \bm{\alpha}_2^{(k)} \odot\partial_{yx}\bm{\Psi}_{in} - \Delta\bm{\Psi}_{in}\cdot \bm{\alpha}_3^{(k)} \odot\partial_{zx}\bm{\Psi}_{in} \\
& - (\partial_x\Delta\bm{\Psi}_{in}\cdot \bm{\alpha}_2^{(k)}\odot\partial_y\bm{\Psi}_{in} - \partial_y\Delta\bm{\Psi}_{in}\cdot\bm{\alpha}_2^{(k)}\odot\partial_x\bm{\Psi}_{in} ) + \partial_z\bm{U}_2^{(k)}\odot \Delta\bm{\Psi}_{in},
\\
\bm{N}_{31} =
& -(\Delta\bm{\Psi}_{in}\cdot \bm{\alpha}_1^{(k)} \odot\partial_{xy}\bm{\Psi}_{in} + \Delta\bm{\Psi}_{in}\cdot \bm{\alpha}_2^{(k)} \odot\partial_{yy}\bm{\Psi}_{in}  +\Delta\bm{\Psi}_{in}\cdot \bm{\alpha}_3^{(k)} \odot\partial_{zy}\bm{\Psi}_{in}) \\
& - (\partial_y\Delta\bm{\Psi}_{in}\cdot \bm{\alpha}_3^{(k)}\odot\partial_z\bm{\Psi}_{in} - \partial_z\Delta\bm{\Psi}_{in}\cdot\bm{\alpha}_3^{(k)}\odot\partial_y\bm{\Psi}_{in} ) + \partial_x\bm{U}_3^{(k)}\odot \Delta\bm{\Psi}_{in},
\\
\bm{N}_{32} =
& \Delta\bm{\Psi}_{in}\cdot \bm{\alpha}_1^{(k)} \odot\partial_{xx}\bm{\Psi}_{in} + \Delta\bm{\Psi}_{in}\cdot \bm{\alpha}_2^{(k)} \odot\partial_{yx}\bm{\Psi}_{in} + \Delta\bm{\Psi}_{in}\cdot \bm{\alpha}_3^{(k)} \odot\partial_{zx}\bm{\Psi}_{in} \\
& - (\partial_z\Delta\bm{\Psi}_{in}\cdot \bm{\alpha}_3^{(k)}\odot\partial_x\bm{\Psi}_{in} - \partial_x\Delta\bm{\Psi}_{in}\cdot\bm{\alpha}_3^{(k)}\odot\partial_z\bm{\Psi}_{in} ) + \partial_y\bm{U}_3^{(k)}\odot \Delta\bm{\Psi}_{in},
\\
\bm{N}_{33} =
&-(\bm{U}_1^{(k)}\odot\partial_x\Delta \bm{\Psi}_{in} + \bm{U}_2^{(k)}\odot\partial_y\Delta \bm{\Psi}_{in} + \bm{U}_3^{(k)}\odot\partial_z\Delta \bm{\Psi}_{in})\\
&- (\partial_x\Delta\bm{\Psi}_{in}\cdot\bm{\alpha}_3^{(k)}\odot\partial_y\bm{\Psi}_{in} - \partial_y\Delta\bm{\Psi}_{in}\cdot \bm{\alpha}_3^{(k)}\odot\partial_x\bm{\Psi}_{in}) + \partial_z\bm{U}_3^{(k)}\odot \Delta\bm{\Psi}_{in},
\end{align*}
\begin{align*}
\tilde{\bm{F}}_1 =
& \partial_y \bm{F}_3 -\partial_z \bm{F}_2 - (\bm{U}_1^{(k)}\odot\partial_x\Delta\bm{\Psi}_{in} \cdot \bm{\alpha}_1^{(k)} + \bm{U}_2^{(k)}\odot\partial_y\Delta\bm{\Psi}_{in} \cdot \bm{\alpha}_1^{(k)}+\bm{U}_3^{(k)}\odot\partial_{z}\Delta\bm{\Psi}_{in} \cdot \bm{\alpha}_1^{(k)} )\\
& \qquad \qquad \qquad \quad + (\partial_x\bm{U}_1^{(k)}\odot\Delta\bm{\Psi}_{in} \cdot \bm{\alpha}_1^{(k)} + \partial_y\bm{U}_1^{(k)}\odot\Delta\bm{\Psi}_{in} \cdot \bm{\alpha}_2^{(k)}+ \partial_z\bm{U}_1^{(k)}\odot\Delta\bm{\Psi}_{in} \cdot \bm{\alpha}_3^{(k)}),
\\
\tilde{\bm{F}}_2 =
& \partial_z \bm{F}_1-\partial_x \bm{F}_3  - (\bm{U}_1^{(k)}\odot\partial_x\Delta\bm{\Psi}_{in} \cdot \bm{\alpha}_2^{(k)} + \bm{U}_2^{(k)}\odot\partial_y\Delta\bm{\Psi}_{in} \cdot \bm{\alpha}_2^{(k)}+\bm{U}_3^{(k)}\odot\partial_{z}\Delta\bm{\Psi}_{in} \cdot \bm{\alpha}_2^{(k)} )\\
& \qquad \qquad \qquad \quad + (\partial_x\bm{U}_2^{(k)}\odot\Delta\bm{\Psi}_{in} \cdot \bm{\alpha}_1^{(k)} + \partial_y\bm{U}_2^{(k)}\odot\Delta\bm{\Psi}_{in} \cdot \bm{\alpha}_2^{(k)}+ \partial_z\bm{U}_2^{(k)}\odot\Delta\bm{\Psi}_{in} \cdot \bm{\alpha}_3^{(k)}),
\\
\tilde{\bm{F}}_3 =
&\partial_x \bm{F}_2-\partial_y \bm{F}_1  - (\bm{U}_1^{(k)}\odot\partial_x\Delta\bm{\Psi}_{in} \cdot \bm{\alpha}_3^{(k)} + \bm{U}_2^{(k)}\odot\partial_y\Delta\bm{\Psi}_{in} \cdot \bm{\alpha}_3^{(k)}+\bm{U}_3^{(k)}\odot\partial_{z}\Delta\bm{\Psi}_{in} \cdot \bm{\alpha}_3^{(k)} )\\
& \qquad \qquad \qquad \quad + (\partial_x\bm{U}_3^{(k)}\odot\Delta\bm{\Psi}_{in} \cdot \bm{\alpha}_1^{(k)} + \partial_y\bm{U}_3^{(k)}\odot\Delta\bm{\Psi}_{in} \cdot \bm{\alpha}_2^{(k)}+ \partial_z\bm{U}_3^{(k)}\odot\Delta\bm{\Psi}_{in} \cdot \bm{\alpha}_3^{(k)}).
\end{align*}
Here,
\begin{align*}
&\bm{U}_1^{(k)} = \partial_y \bm{\Psi}_{in} \cdot \bm{\alpha}_3^{(k)} - \partial_z \bm{\Psi}_{in} \cdot \bm{\alpha}_2^{(k)},
\quad
\partial_r \bm{U}_1^{(k)} = \partial_{ry} \bm{\Psi}_{in} \cdot \bm{\alpha}_3^{(k)} - \partial_{rz} \bm{\Psi}_{in} \cdot \bm{\alpha}_2^{(k)},  \ r= x,y,z,
\\
&\bm{U}_2^{(k)} = \partial_z \bm{\Psi}_{in} \cdot \bm{\alpha}_1^{(k)} - \partial_x \bm{\Psi}_{in} \cdot \bm{\alpha}_3^{(k)},
\quad
\partial_r \bm{U}_2^{(k)} = \partial_{rz} \bm{\Psi}_{in} \cdot \bm{\alpha}_1^{(k)} - \partial_{rx} \bm{\Psi}_{in} \cdot \bm{\alpha}_3^{(k)}, \ r =x,y,z,
\\
&\bm{U}_3^{(k)} = \partial_x \bm{\Psi}_{in} \cdot \bm{\alpha}_2^{(k)} - \partial_y \bm{\Psi}_{in} \cdot \bm{\alpha}_1^{(k)},
\quad
\partial_r \bm{U}_3^{(k)} = \partial_{rx} \bm{\Psi}_{in} \cdot \bm{\alpha}_2^{(k)} - \partial_{ry} \bm{\Psi}_{in} \cdot \bm{\alpha}_1^{(k)}, \ r =x,y,z,
\end{align*}
\begin{align*}
& \Delta \bm{\Psi}_{in} = [\Delta\psi_m(\bm{x}^{in}_i)] \in \mathbb{R}^{I\times(M+1)},
\quad \quad
\partial_r\bm{\Psi}_{in} = [\partial_r\psi_m(\bm{x}^{in}_i)] \in \mathbb{R}^{I\times(M+1)},  \ r = x,y,z,
\\
&\Delta^2 \bm{\Psi}_{in} = [\Delta^2\psi_m(\bm{x}^{in}_i)] \in \mathbb{R}^{I\times(M+1)},
\quad
\partial_{\iota r}\bm{\Psi}_{in} = [\partial_\iota\partial_r\psi_m(\bm{x}^{in}_i)] \in \mathbb{R}^{I\times(M+1)}, \ \iota, r = x,y,z,
\\
&\partial_r\Delta\bm{\Psi}_{in} = [\partial_r\Delta\psi_m(\bm{x}^{in}_i)] \in \mathbb{R}^{I\times(M+1)},  \ r = x,y,z,
\\
&\bm{\Psi}_{bd}\bm{n}_1 = [\psi_m(\bm{x}^{bd}_j)n_1(\bm{x}^{bd}_j)] \in \mathbb{R}^{J\times(M+1)},
\quad
\bm{\Psi}_{bd}\bm{n}_2 = [\psi_m(\bm{x}^{bd}_j)n_2(\bm{x}^{bd}_j)] \in \mathbb{R}^{J\times(M+1)},
\\
&\bm{\Psi}_{bd}\bm{n}_3 = [\psi_m(\bm{x}^{bd}_j)n_3(\bm{x}^{bd}_j)] \in \mathbb{R}^{J\times(M+1)},
\quad
\partial_s\bm{\Psi}_{bd} = [\partial_s\psi_m(\bm{x}^{bd}_j)] \in \mathbb{R}^{J\times(M+1)}, \ s = x,y,z,
\\
&\partial_y \bm{F}_3-\partial_z \bm{F}_2 =[\partial_yf_3(\bm{x}^{in}_i)-\partial_zf_2(\bm{x}^{in}_i)] \in \mathbb{R}^{I},\\
&\partial_z \bm{F}_1-\partial_x \bm{F}_3 =[\partial_zf_1(\bm{x}^{in}_i)-\partial_xf_3(\bm{x}^{in}_i)] \in \mathbb{R}^{I},\\
&\partial_x \bm{F}_2-\partial_y \bm{F}_1 =[\partial_xf_2(\bm{x}^{in}_i)-\partial_yf_1(\bm{x}^{in}_i)] \in \mathbb{R}^{I}.
\end{align*}
To solve \eqref{NS3d-matv}, we turn to solve least-squares problem
\begin{equation}\label{NS3d-min_v}
\min_{\bm{\alpha}^{(k+1)}}
\left\Vert
\begin{bmatrix}{}
\nu \Delta^2 \bm{\Psi}_{in} + \bm{N}_{11}   & \bm{N}_{12}  &\bm{N}_{13}  \\
\bm{N}_{21}   & \nu \Delta^2\bm{\Psi}_{in}+\bm{N}_{22}    & \bm{N}_{23} \\
\bm{N}_{31}   & \bm{N}_{32}   & \nu \Delta^2 \bm{\Psi}_{in} +\bm{N}_{33} \\
\partial_x\bm{\Psi}_{in} & \partial_y\bm{\Psi}_{in} & \partial_z\bm{\Psi}_{in} \\
\bm{\Psi}_{bd}\bm{n}_1 & \bm{\Psi}_{bd}\bm{n}_2 & \bm{\Psi}_{bd}\bm{n}_3 \\
  &-\partial_z \bm{\Psi}_{bd}  &\partial_y \bm{\Psi}_{bd} \\
\partial_z \bm{\Psi}_{bd} &  &-\partial_x \bm{\Psi}_{bd}  \\
-\partial_y \bm{\Psi}_{bd}  &\partial_x \bm{\Psi}_{bd}  &
\end{bmatrix}
\begin{bmatrix}{}
\bm{\alpha}_{1}^{(k+1)} \\
\bm{\alpha}_{2}^{(k+1)} \\
\bm{\alpha}_{3}^{(k+1)}
\end{bmatrix}
-
\begin{bmatrix}{}
\tilde{\bm{F}}_1 \\
\tilde{\bm{F}}_2 \\
\tilde{\bm{F}}_3 \\
\bm{0} \\
\bm{0} \\
\bm{0} \\
\bm{0} \\
\bm{0} \\
\end{bmatrix}
\right\Vert_2^2.
\end{equation}
Once obtaining the minimizer $\bm{\alpha}^*$, we can reformulate \eqref{NS-pre} to the least-squares problem
\begin{equation}\label{NS3d-matp}
\min_{\bm{\beta}}\left\Vert
\begin{bmatrix}{}
\partial_x \bm{\Psi}_{in} \\
\partial_y \bm{\Psi}_{in} \\
\partial_z \bm{\Psi}_{in} \\
\bm{\Psi}_{0}^{\rm T}
\end{bmatrix}
\bm{\beta}
-
\begin{bmatrix}{}
\tilde{\bm{F}}_1  \\
\tilde{\bm{F}}_2  \\
\tilde{\bm{F}}_3  \\
{0}
\end{bmatrix}
\right\Vert_2^2,
\end{equation}
where
\begin{align*}
&\tilde{\bm{F}}_1 = [f_1(\bm{x}^{in}_i) + \nu\Delta u_1^*(\bm{x}^{in}_i)-(u_1^*(\bm{x}^{in}_i)\partial_x u_1^*(\bm{x}^{in}_i) + u_2^*(\bm{x}^{in}_i)\partial_y u_1^*(\bm{x}^{in}_i) + u_3^*(\bm{x}^{in}_i)\partial_z u_1^*(\bm{x}^{in}_i))] \in \mathbb{R}^{I},
\\
&\tilde{\bm{F}}_2 = [f_2(\bm{x}^{in}_i) + \nu\Delta u_2^*(\bm{x}^{in}_i)-(u_1^*(\bm{x}^{in}_i)\partial_x u_2^*(\bm{x}^{in}_i) + u_2^*(\bm{x}^{in}_i)\partial_y u_2^*(\bm{x}^{in}_i) + u_3^*(\bm{x}^{in}_i)\partial_z u_2^*(\bm{x}^{in}_i))] \in \mathbb{R}^{I},
\\
&\tilde{\bm{F}}_3 = [f_3(\bm{x}^{in}_i) + \nu\Delta u_3^*(\bm{x}^{in}_i)-(u_1^*(\bm{x}^{in}_i)\partial_x u_3^*(\bm{x}^{in}_i) + u_2^*(\bm{x}^{in}_i)\partial_y u_3^*(\bm{x}^{in}_i) + u_3^*(\bm{x}^{in}_i)\partial_z u_3^*(\bm{x}^{in}_i))] \in \mathbb{R}^{I},
\\
&\bm{\Psi}_{0} = [\psi_m(\bm{x}_0)] \in \mathbb{R}^{M+1}.
\end{align*}
The resulting pressure approximation is $p_{\rm NN} = \sum_{m=0}^M \beta^*_m\psi_m(\bm{x})$.

\subsection{Computational complexity}
Based on the previous discussion, the final step of the proposed method involves solving an over-determined linear least-squares problem of the form
\begin{equation}\label{LS}
    \min_{\bm{x}} \| \bm{A}\bm{x} - \bm{b} \|_2^2,
\end{equation}
where $\bm{A}\in\mathbb{R}^{m\times n}$ ($m>>n$), $\bm{x} \in \mathbb{R}^n$ and $\bm{b} \in \mathbb{R}^{m}$. The standard computational complexity for resolving such a system is $\mathcal{O}(mn^2)$. Both the current decoupled Decoupled-DFNN and the standard TransNet rely on solving \eqref{LS}, their relative efficiency can be compared by analyzing the dimensions of their respective matrices.

In our proposed framework, the velocity and pressure subproblems are solved sequentially, resulting in smaller and independent linear systems. In contrast, the standard TransNet approach solves all variables simultaneously, which leads to a larger coupled system. The computational complexities for the two-dimensional and three-dimensional cases are summarized in Tab. \ref{tab-complexity}, where $I$ and $J$ denote the number of interior and boundary collocation points, and $M$ represents the number of hidden nodes.

\begin{table}[H]
\small
\centering
  \begin{tabular}{ccc}
    \hline
    Method & Decoupled-DFNN &TransNet\\
    \hline
    2d  & $\mathcal{O}\big((I+2J)(M+1)^2\big) + \mathcal{O}\big((2I+1)(M+1)^2\big)$   & $\mathcal{O}\big((3I+2J)(3M+3)^2\big)$
    \\
    3d  & $\mathcal{O}\big((4I+4J)(3M+3)^2\big) + \mathcal{O}\big((3I+1)(M+1)^2\big)$ & $\mathcal{O}\big((4I+3J)(4M+4)^2\big)$
    \\
    \hline
  \end{tabular}
  \caption{Comparison of the linear least-squares solution costs for 2D and 3D problems.}\label{tab-complexity}
\end{table}

\section{Numerical experiments}\label{sec5}
In this section, we evaluate the performance of the proposed Decoupled-DFNN method for incompressible flow problems. Our experiments include numerical comparisons of Decoupled-DFNN with TransNet and, for nonlinear problems, PINN. For brevity, Decoupled-DFNN is denoted by DFNN in the figures and tables.
In PINN framework, we shall use the ResNet \cite{ResNet} with width $M=30$ and depth $L=4$ with an activation function $\sigma(x) = \tanh(x)$ to approximate the velocity and pressure, and the parameters are determined by the Adam optimizer \cite{Adam} and L-BFGS optimizer \cite{LBFGS}. Specifically, we first train the ResNet using Adam for several epochs and then employ L-BFGS with a small number of iterations to further improve the accuracy of the PINN.

To quantify the accuracy of the numerical results, we define the relative $L^2$ error between the exact solution $g$ and the neural network approximation $g_{\rm NN}$ as
\[
    \mbox{error}\_ g = \frac{\sqrt{\frac{1}{N}\sum_{i=1}^N [g(\bm{x}_i) - g_{\rm NN}(\bm{x}_i)]^2}}{\sqrt{\frac{1}{N}\sum_{i=1}^N [g(\bm{x}_i)]^2}}.
\]
To assess the enforcement of the incompressibility constraint, we further define the absolute divergence error as
\[
    \mbox{error}\_\mbox{div} = \sqrt{\frac{1}{N}\sum_{i=1}^N [\dive \bm{g}_{\rm NN}(\bm{x}_i)]^2}.
\]
For simplicity, the velocity field is denoted by $\bm{u}(\bm{x}) = [u(\bm{x}), v(\bm{x})]^\intercal$ for 2D cases, and $\bm{u}(\bm{x}) = [u(\bm{x}), v(\bm{x}), w(\bm{x})]^\intercal$ for 3D case.

All experiments are performed on an NVIDIA Tesla V100 PCIe GPU with 16 GB of memory. The algorithms are implemented in Python 3.8 using NumPy 1.24, SciPy 1.10, CuPy 12.3, and PyTorch 2.2 with CUDA 11.8.

\subsection{2D Stokes problem}\label{sub:2d_stokes_problem}
Consider the Stokes problem defined on the domain $\Omega = (-0.5, 1) \times (-0.5, 1.5)$, for which the exact solution is given by
\[
u = 1- e^{\zeta x} \cos(2\pi y), \quad v = \frac{\zeta}{2\pi}e^{\zeta x} \sin(2\pi y),
\quad
p = -\frac{1}{2}e^{2\zeta x} + p_0,
\]
where $\zeta = \frac{1}{2\nu} - \sqrt{\frac{1}{4\nu^2} + 4 \pi^2}$.

In this test case, we select $2~500~(50\times50)$ interior and $2~000~(500\times4)$ boundary uniform grid points as the training samples, while an additional $12~321~(111\times111)$ uniform grid points are chosen as test samples.

We first compare the numerical performance of our proposed Decoupled-DFNN method with that of the standard TransNet, using the tanh activation function and setting the shape parameters to $\gamma=1$ and $r=2.5$ for both methods. For a fixed kinematic viscosity $\nu = 10^{-4}$, Fig.~\ref{fig-2d-1} shows the errors in velocity, divergence, and pressure with varying numbers of basis functions. The results indicate that Decoupled-DFNN not only achieves higher accuracy for velocity and pressure but also satisfies the incompressibility constraint more effectively, reaching a precision of at least $\mathcal{O}(10^{-12})$. Moreover, Fig.~\ref{fig-2d-1} demonstrates the superior efficiency of our method over the standard TransNet. While the computational cost of both approaches increases with the number of basis functions, the efficiency advantage of Decoupled-DFNN becomes more pronounced for larger basis sizes. This improvement is mainly due to the fact that the proposed method solves smaller and decoupled subproblems, rather than a large coupled system.

\begin{figure}[H]
\centering
    \begin{minipage}[t]{0.45\linewidth}
    {\includegraphics[scale=0.4]{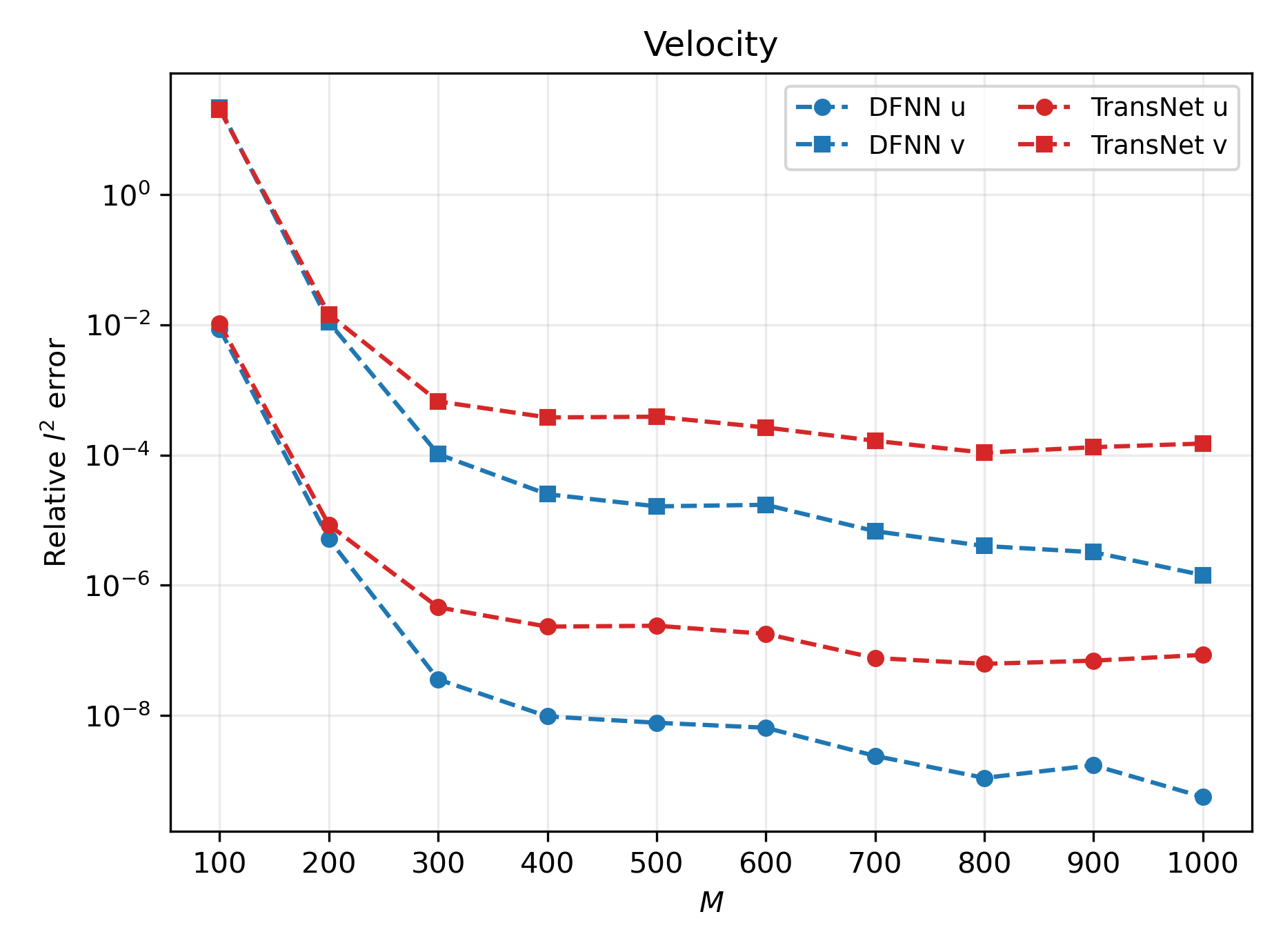}}
  \end{minipage}
  \begin{minipage}[t]{0.45\linewidth}
    {\includegraphics[scale=0.4]{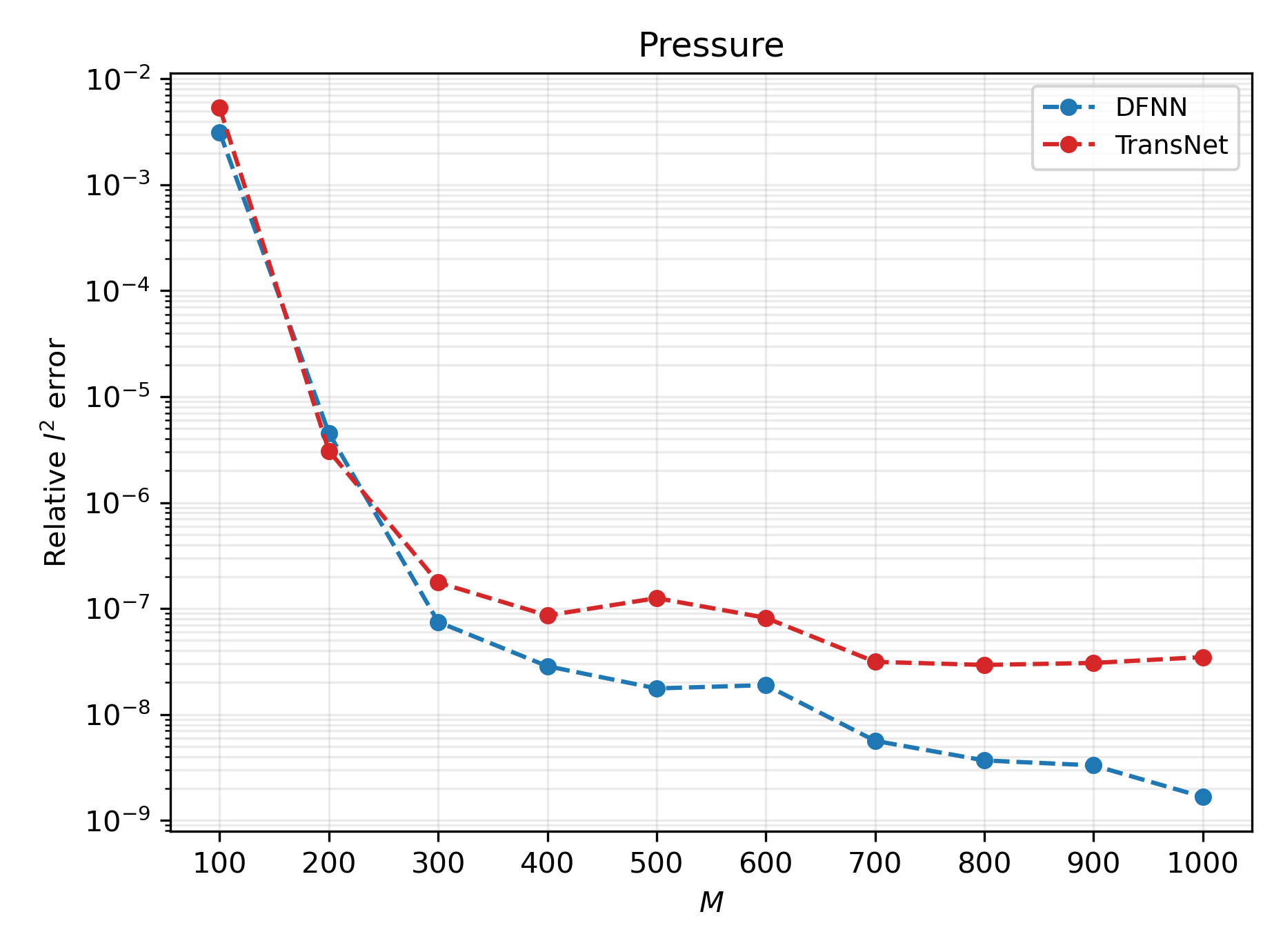}}
  \end{minipage}
  \begin{minipage}[t]{0.45\linewidth}
    {\includegraphics[scale=0.4]{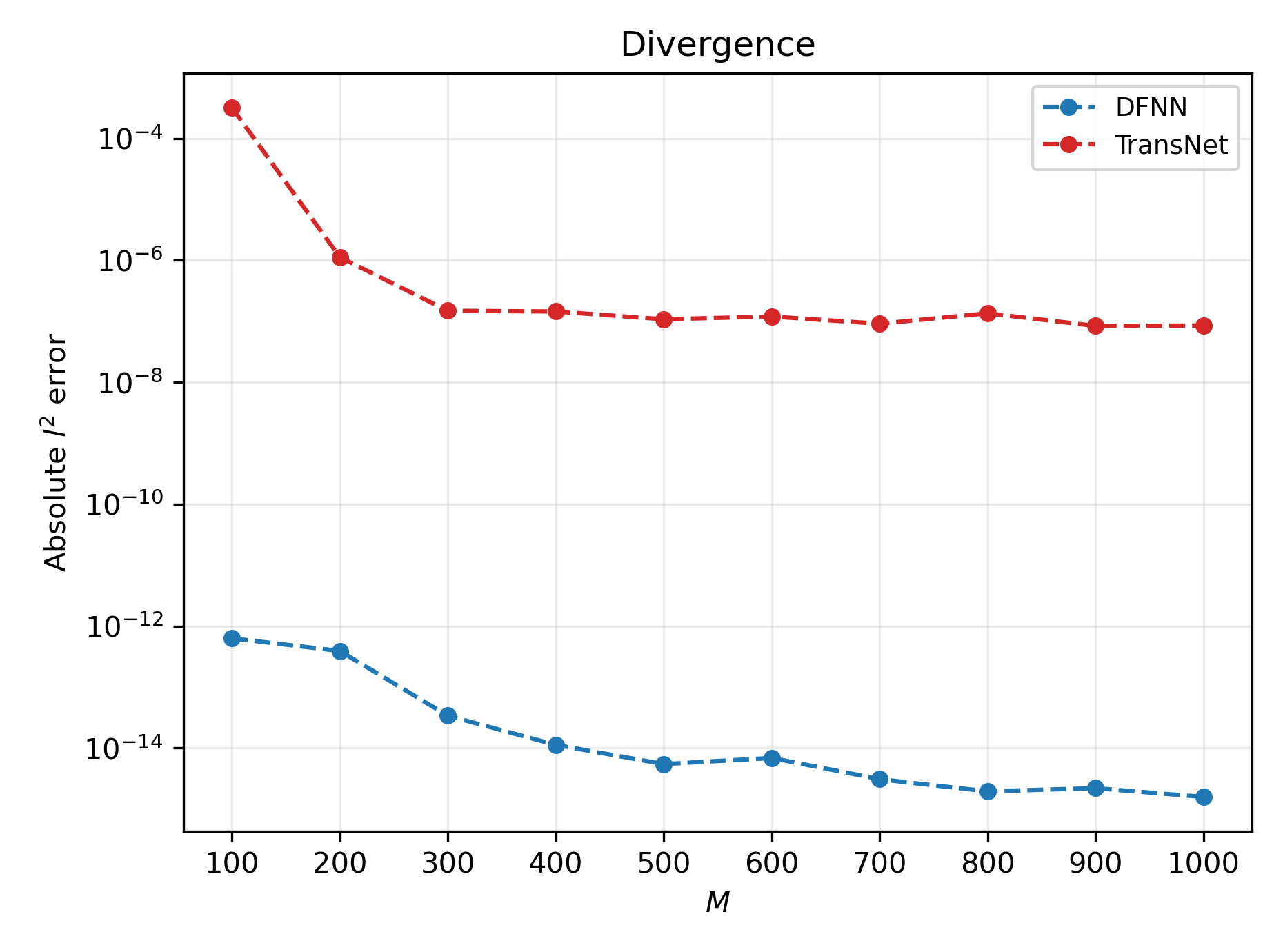}}
  \end{minipage}
    \begin{minipage}[t]{0.45\linewidth}
    {\includegraphics[scale=0.4]{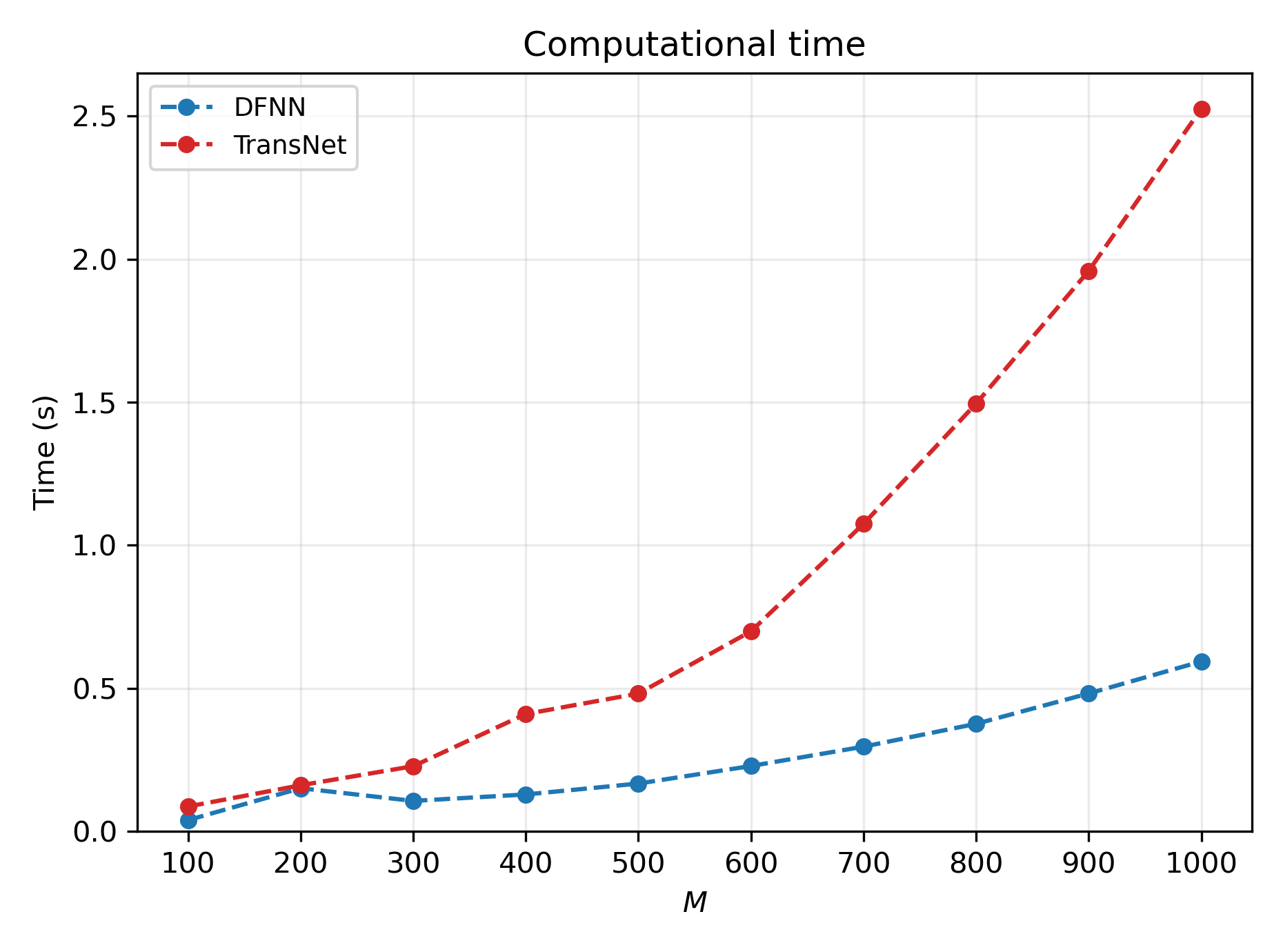}}
  \end{minipage}
  \caption{Numerical performance comparison for the 2D Stokes problem with varying numbers of basis functions $M$.}\label{fig-2d-1}
\end{figure}

Next, we examine numerical performance across varying kinematic viscosity $\nu$ (or equivalently, Reynolds number $\mathrm{Re}$) with $M=1~000$. As illustrated in Fig.~\ref{fig-2d-2}, Decoupled-DFNN demonstrates significantly superior performance in preserving the divergence-free condition. In terms of velocity and pressure accuracy, the two methods remain largely comparable for large viscosities, whereas Decoupled-DFNN performs significantly better for smaller viscosities ($\nu \leq 10^{-2}$, or equivalently, $\mathrm{Re} \geq 10^2$), outperforming TransNet. 
This enhanced precision indicates favorable performance for the present Stokes benchmark at small viscosities. Furthermore, variations in $\nu$ have little effect on the computational cost.

\begin{figure}[H]
\centering
    \begin{minipage}[t]{0.32\linewidth}
    {\includegraphics[scale=0.34]{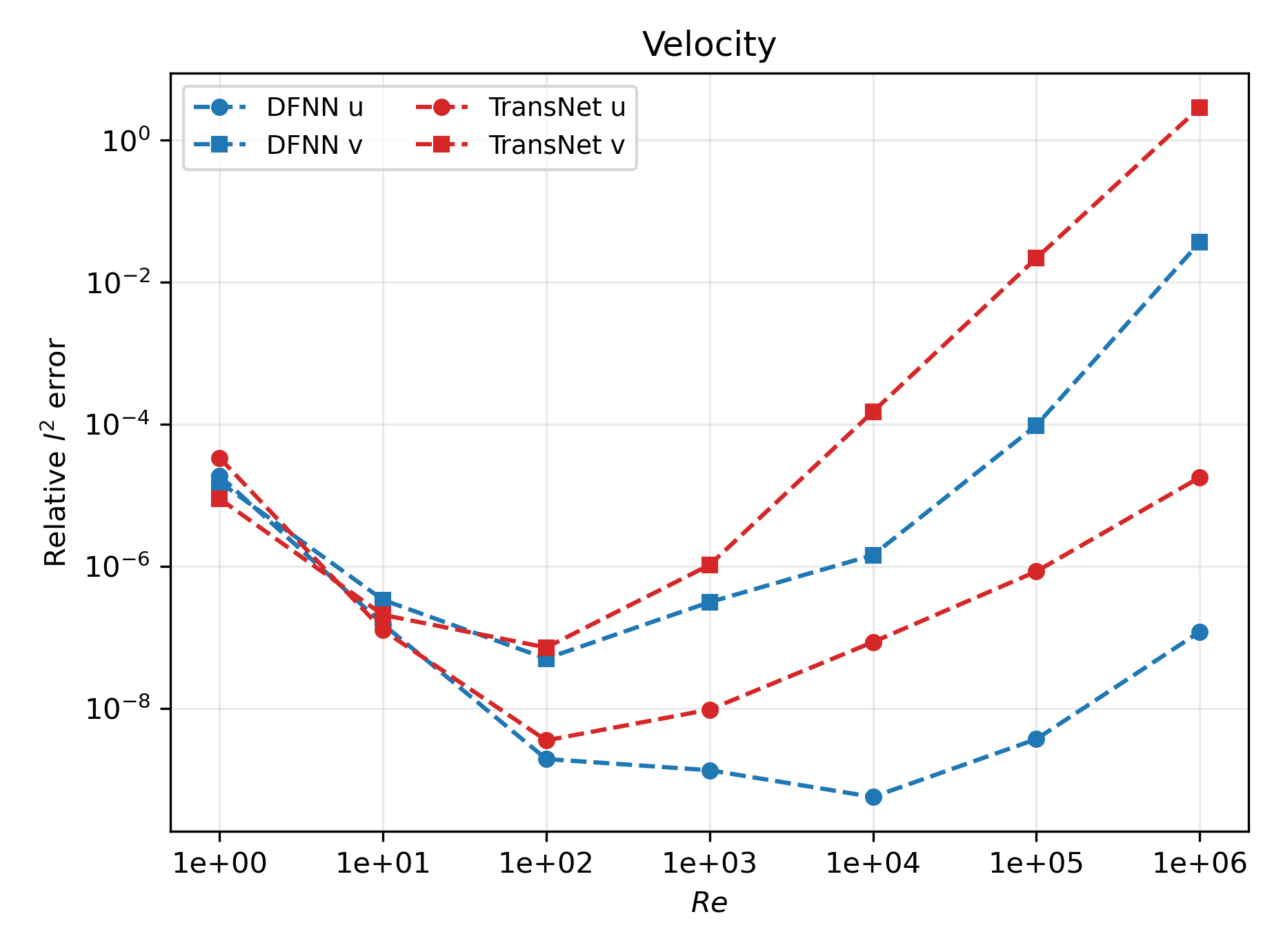}}
  \end{minipage}
    \begin{minipage}[t]{0.32\linewidth}
    {\includegraphics[scale=0.34]{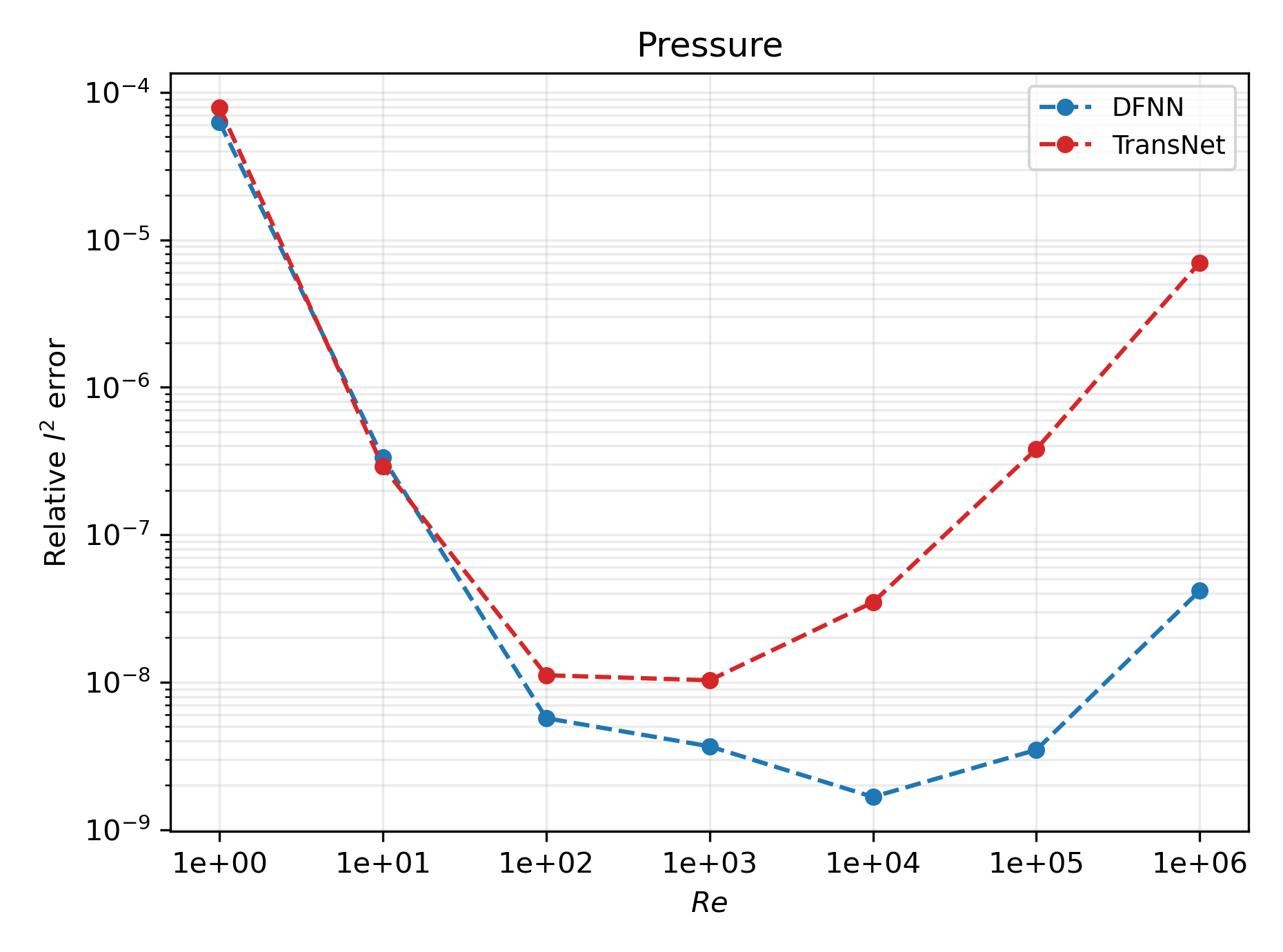}}
  \end{minipage}
  \begin{minipage}[t]{0.32\linewidth}
    {\includegraphics[scale=0.34]{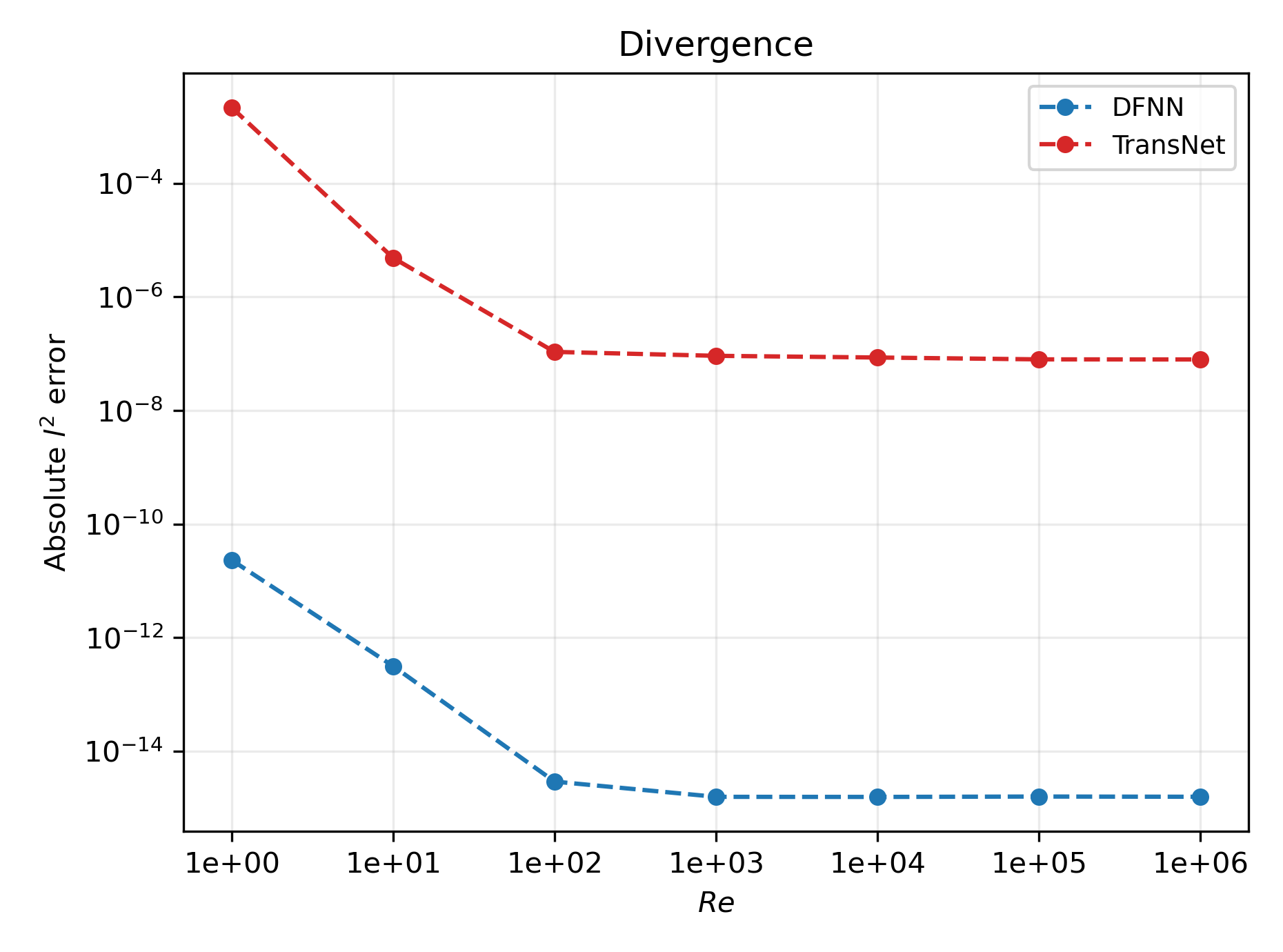}}
  \end{minipage}
  \caption{Error comparison for the 2D Stokes problem with varying Reynolds number $\mathrm{Re}$.}\label{fig-2d-2}
\end{figure}

\subsection{2D Navier-Stokes problem}\label{sub:2d_navier_stokes_problem}
We study the Kovasznay flow problem defined on the domain $\Omega = (-0.5, 1) \times (-0.5, 1.5)$. The exact solution is given by
\[
u = 1- e^{\zeta x} \cos(2\pi y), \quad v = \frac{\zeta}{2\pi}e^{\zeta x} \sin(2\pi y),
\quad
p = -\frac{1}{2}e^{2\zeta x} + p_0,
\]
where $\zeta = \frac{1}{2\nu} - \sqrt{\frac{1}{4\nu^2} + 4 \pi^2}$. This is a classical steady-flow benchmark (see, e.g. \cite{Chen-2017}).

In this case, we employ the same training and test sets as those in subsection \ref{sub:2d_stokes_problem} and use a zero initial guess for the Gauss-Newton solver. We set the maximum number of iterations to $100$ and the tolerance to $5\times10^{-8}$.

\begin{figure}[H]
\centering
    \begin{minipage}[t]{0.34\linewidth}
    {\includegraphics[scale=0.37]{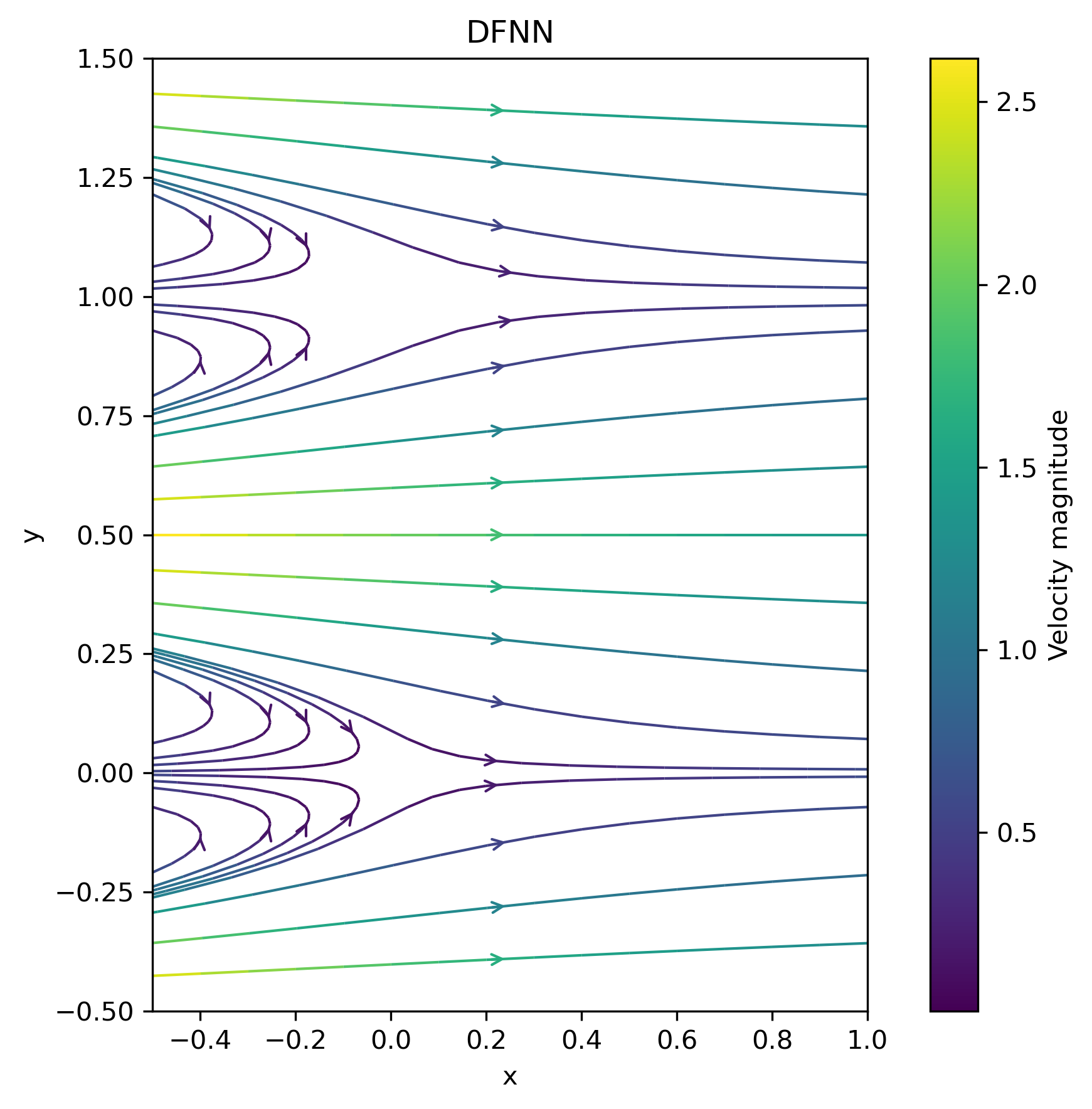}}
  \end{minipage}
  \begin{minipage}[t]{0.34\linewidth}
    {\includegraphics[scale=0.37]{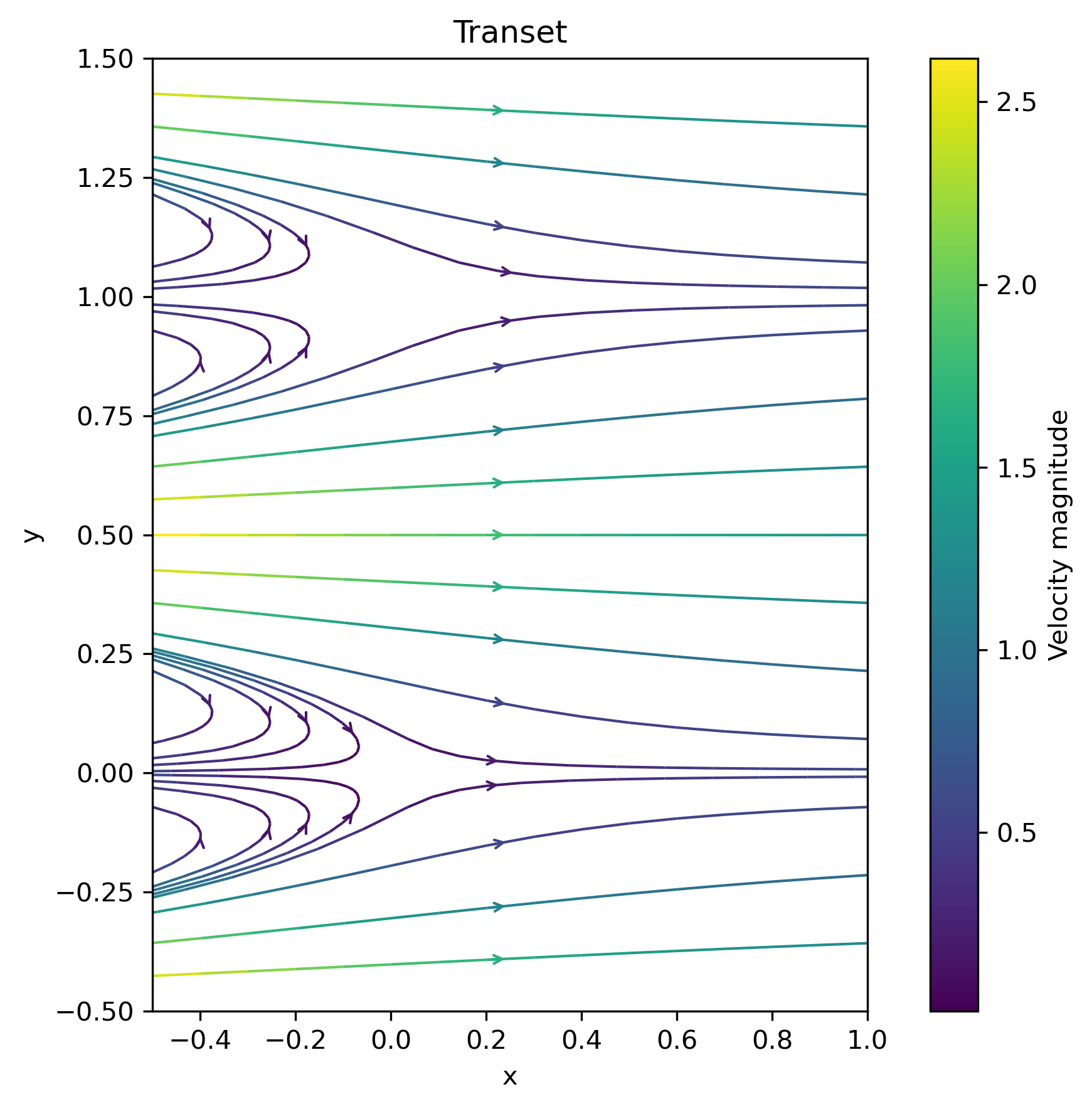}}
  \end{minipage}
  \begin{minipage}[t]{0.3\linewidth}
    {\includegraphics[scale=0.37]{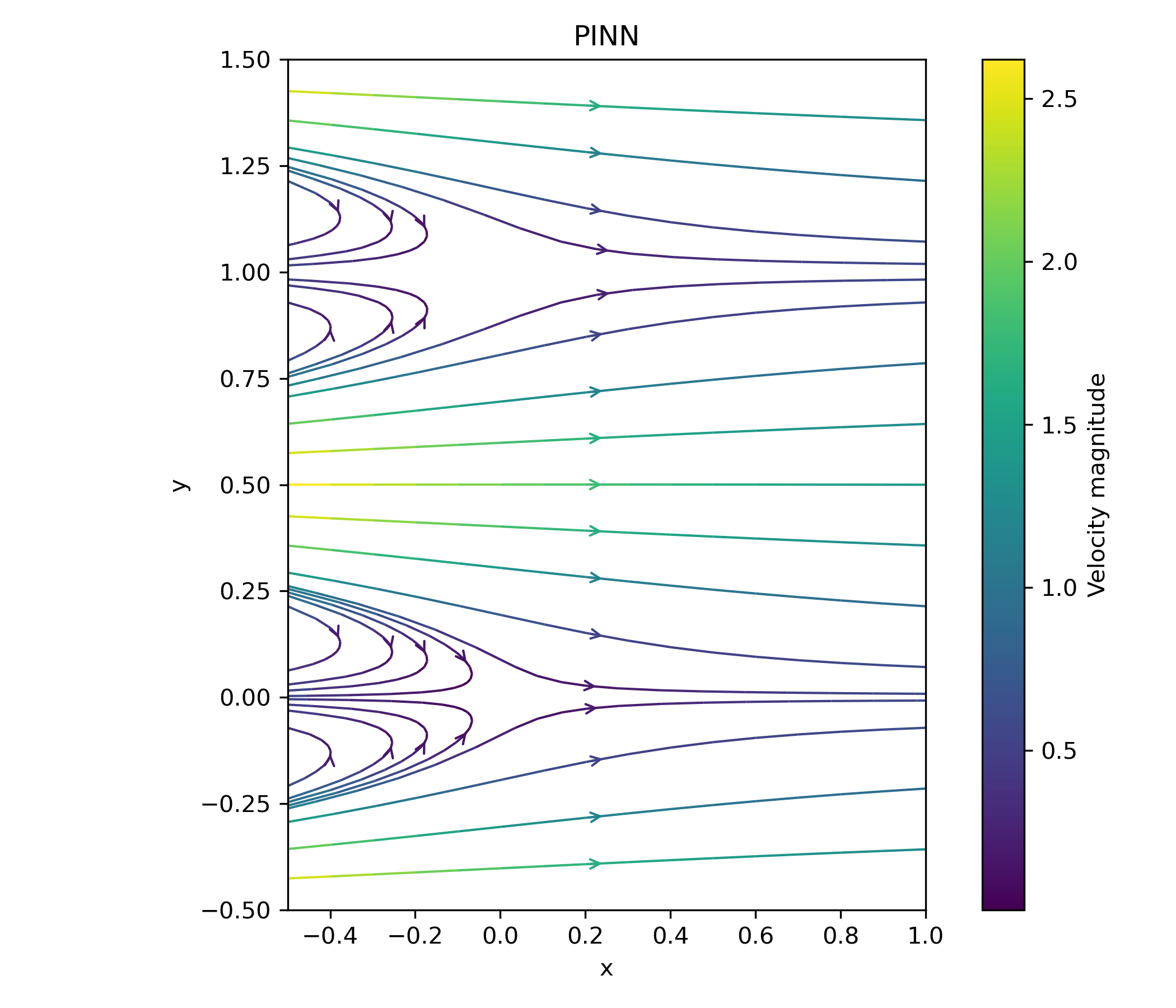}}
  \end{minipage}
  \caption{Comparison of streamlines obtained by Decoupled-DFNN, TransNet, and PINN at $\mathrm{Re}=40$}\label{fig-2d-3}
\end{figure}

We first consider the case $\mathrm{Re}=40$ and compare the numerical results obtained by Decoupled-DFNN, TransNet, and PINN, as shown in Figs.~\ref{fig-2d-3}--\ref{fig-2d-4} and Tab.~\ref{tab2}. The streamlines and field plots indicate that all three methods capture the flow behavior correctly, while Tab.~\ref{tab2} further quantifies their accuracy and computational cost. Decoupled-DFNN and TransNet attain errors of order $\mathcal{O}(10^{-9})$, whereas the error of PINN is of order $\mathcal{O}(10^{-3})$. More importantly, Decoupled-DFNN requires less than 4 seconds, about half the computational time of TransNet and only about $1/50$ that of PINN. These results clearly demonstrate the advantage of neural network basis methods over PINNs for nonlinear problems, particularly in terms of accuracy and computational efficiency. These results are obtained with $M=1~000$, $\gamma=2$, $r=3$, and the Gaussian activation function.

\begin{figure}[H]
\centering
    {\includegraphics[scale=0.45]{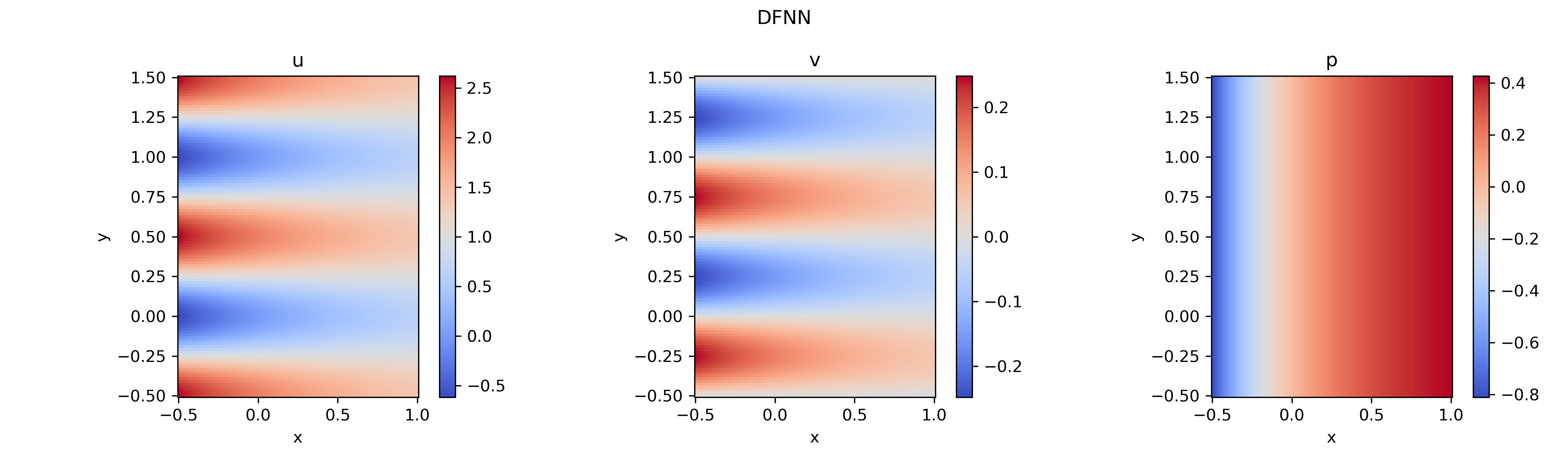}}
    {\includegraphics[scale=0.45]{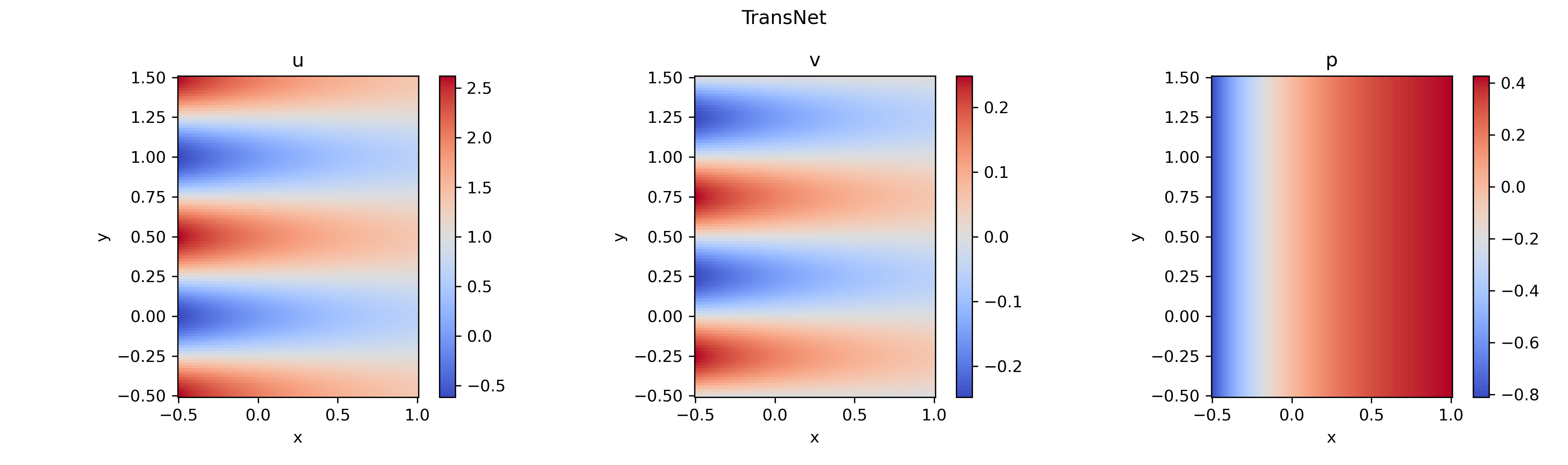}}
    {\includegraphics[scale=0.45]{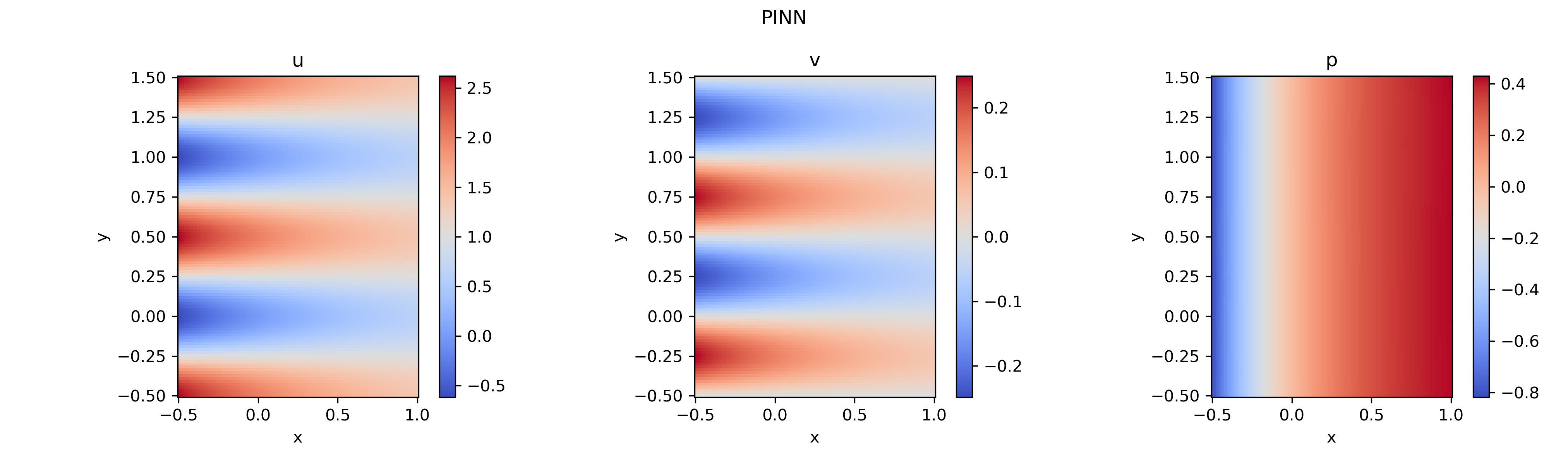}}
  \caption{Comparison of fields obtained by Decoupled-DFNN, TransNet, and PINN at $\mathrm{Re}=40$}\label{fig-2d-4}
\end{figure}

\begin{table}[H]
  \centering
  \caption{The numerical performance of Decoupled-DFNN, TransNet and PINN.}\label{tab2}
    \begin{tabular}{ccccccc}
    \hline
    method  &$\mbox{error}_u$ &$\mbox{error}_v$  &$\mbox{error}_p$  &$\mbox{error}_{div}$  & iterations &GPU time(s) \\
    \hline
    DFNN        & 2.2227e-10    & 2.5811e-09   & 1.2831e-09  & 1.7085e-16  & 6  & 3.6 \\
    TransNet    & 1.8331e-10    & 1.7102e-09   & 1.1373e-09  & 5.7781e-09  & 5  & 8.1 \\
    PINN        & 1.0717e-03    & 9.6570e-03   & 4.5527e-03  & 7.6598e-03  & 2000+100 &229.0 \\
    \hline
    \end{tabular}
\end{table}

We next investigate the effects of the parameters in neural-network-basis methods, beginning with the number of basis functions $M$. As shown in Fig.~\ref{fig-2d-5}, the Gauss--Newton iteration satisfies the prescribed stopping criterion for TransNet when $M\geq400$ and for Decoupled-DFNN when $M\geq500$; otherwise, the iteration terminates upon reaching the maximum number of iterations. When the iterations for both methods satisfy the prescribed stopping criterion, TransNet shows an accuracy advantage over Decoupled-DFNN for relatively small $M$; however, this advantage gradually diminishes and eventually disappears as $M$ increases. In contrast, Decoupled-DFNN consistently requires less computational time than TransNet. Moreover, consistent with the Stokes case, Decoupled-DFNN maintains its superior divergence-free property, with divergence errors on the order of $10^{-15}$.

\begin{figure}[H]
\centering
    \begin{minipage}[t]{0.45\linewidth}
    {\includegraphics[scale=0.4]{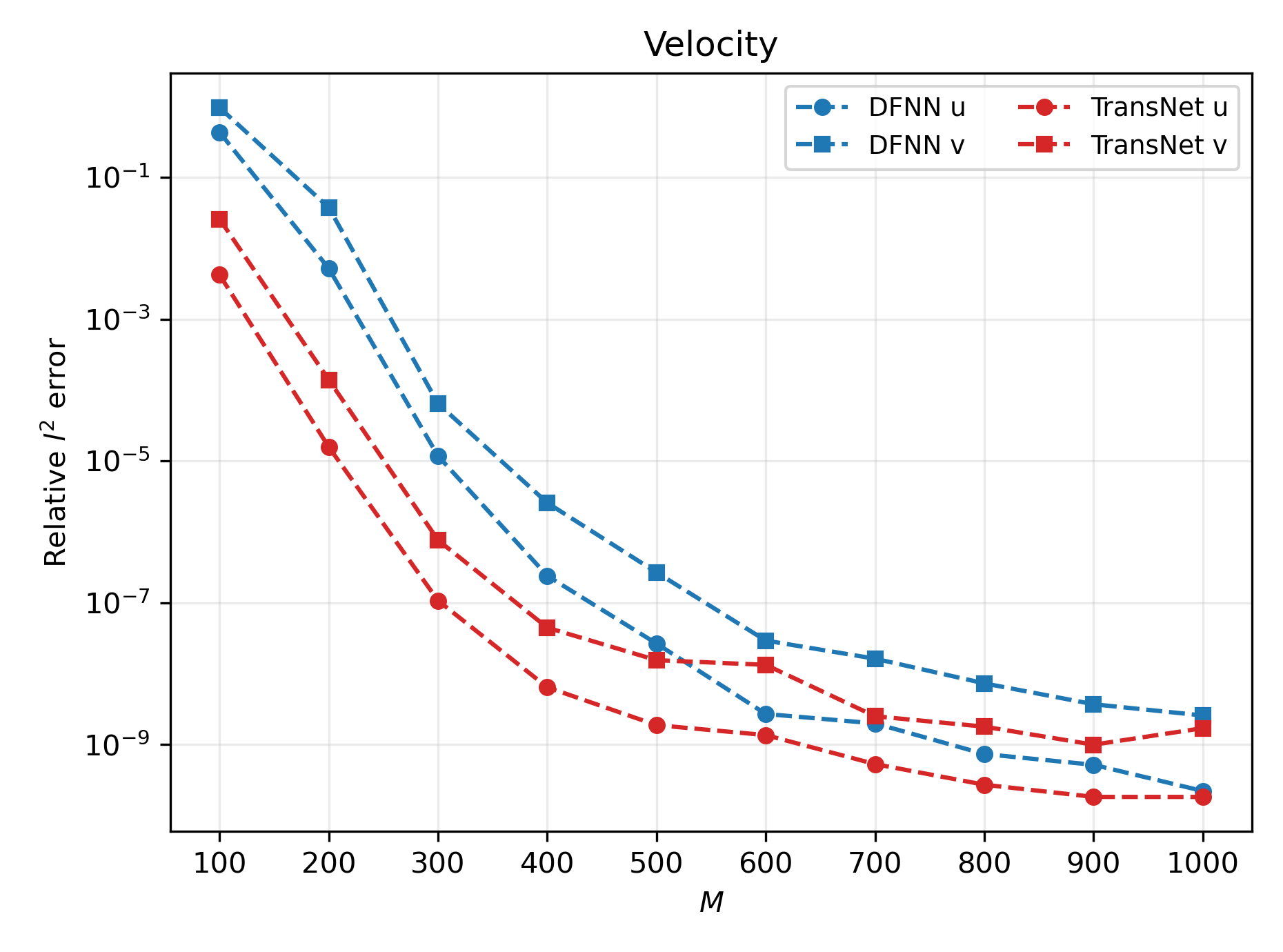}}
  \end{minipage}
  \begin{minipage}[t]{0.45\linewidth}
    {\includegraphics[scale=0.4]{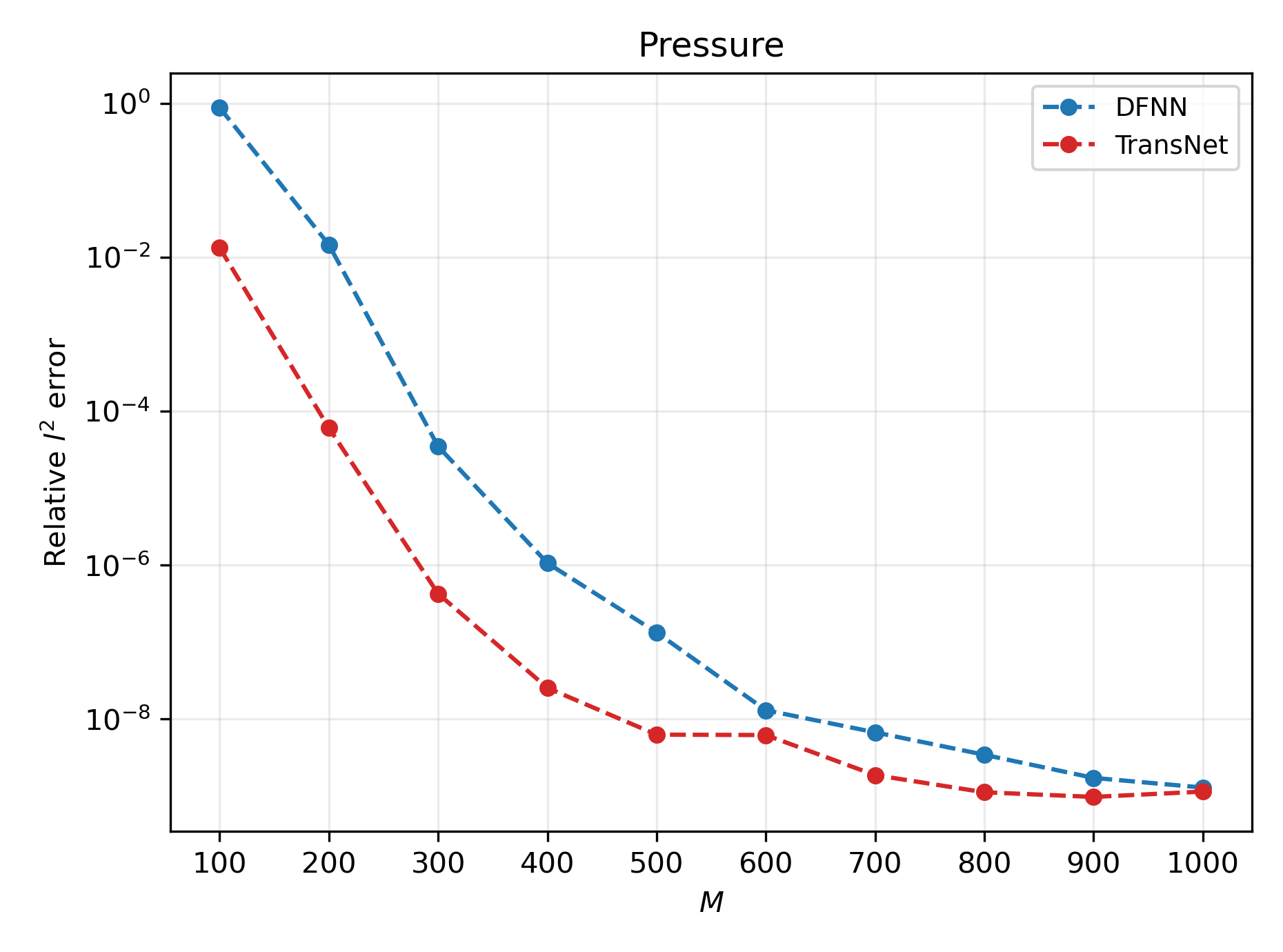}}
  \end{minipage}
  \begin{minipage}[t]{0.45\linewidth}
    {\includegraphics[scale=0.4]{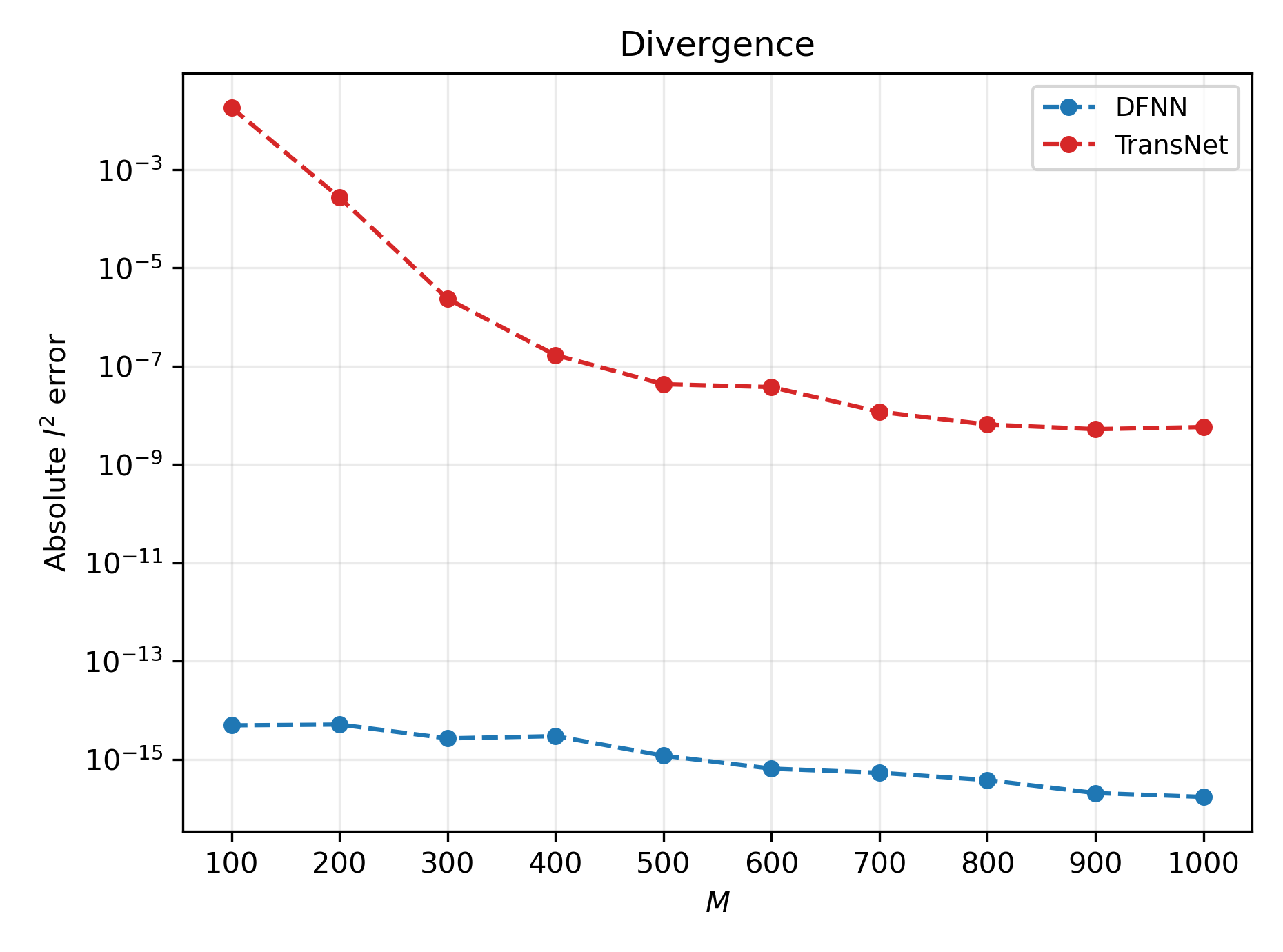}}
  \end{minipage}
  \begin{minipage}[t]{0.45\linewidth}
    {\includegraphics[scale=0.4]{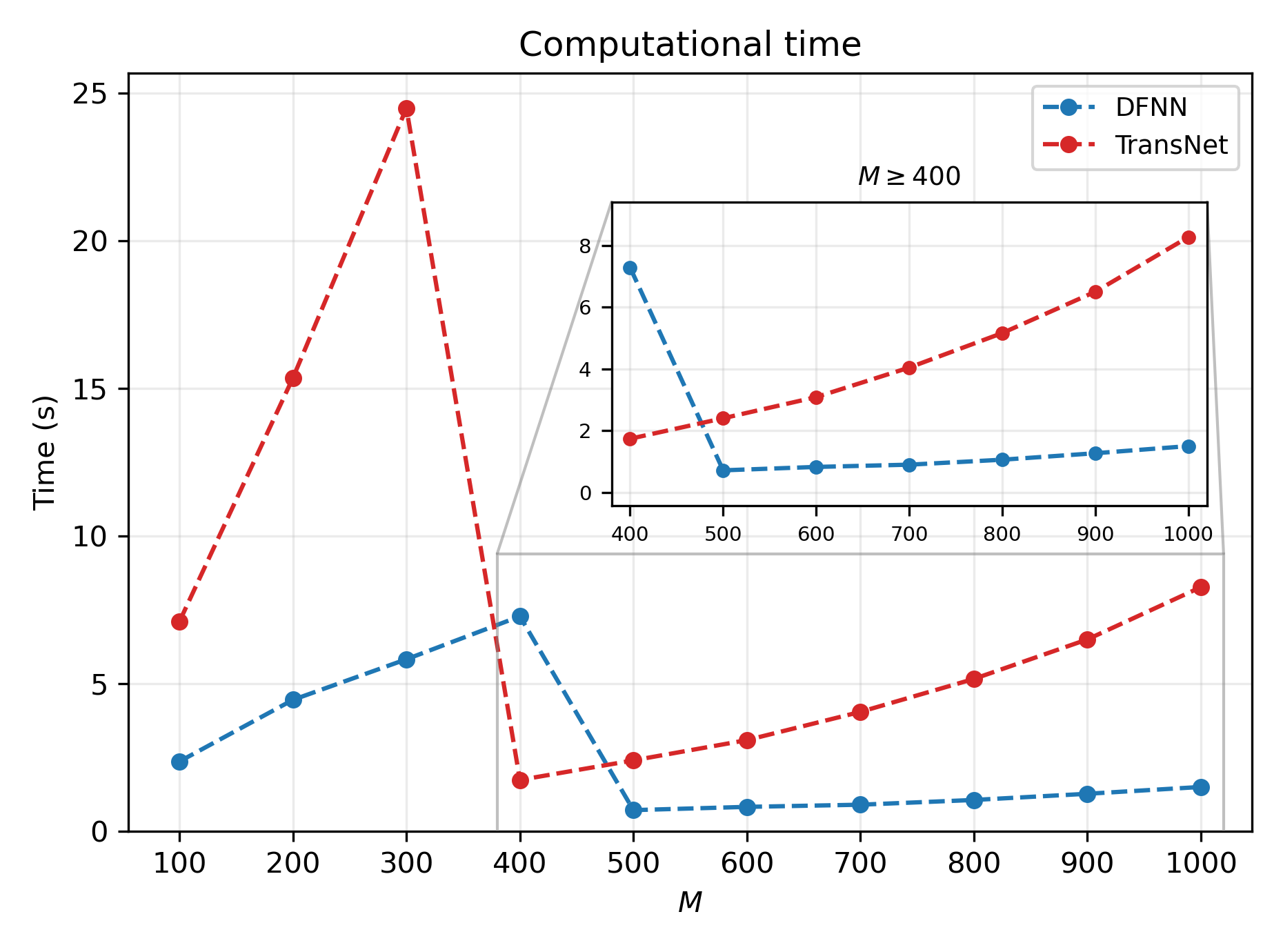}}
  \end{minipage}
  \caption{Numerical results for the Kovasznay flow problem with varying numbers of basis functions $M$.}\label{fig-2d-5}
\end{figure}

We then fix $M$, $\gamma$, and $r$ at the values specified above and compare the tanh and Gaussian activation functions in Figs.~\ref{fig-2d-6}--\ref{fig-2d-7} The results show that the Gaussian activation function not only yields higher accuracy but also improves the iterative performance of both Decoupled-DFNN and TransNet. In particular, among the tested Reynolds numbers, both methods satisfy the prescribed stopping criterion with the Gaussian activation function for $\mathrm{Re}=10,20,40,80,$ and $100$. In contrast, with the tanh activation function, Decoupled-DFNN fails to meet the stopping criterion at these Reynolds numbers, whereas TransNet meets it only for $\mathrm{Re}=40,80,$ and $100$. Therefore, the Gaussian activation function is used throughout the remainder of this paper.

We next investigate the effect of the shape parameter $\gamma$. As shown in Figs.~\ref{fig-2d-8}--\ref{fig-2d-9}, the accuracy generally improves as $\gamma$ increases over the tested range, although the improvement in accuracy becomes less pronounced for moderate-to-large values of $\gamma$. For $\mathrm{Re}=40$ and $100$, both Decoupled-DFNN and TransNet satisfy the prescribed stopping criterion for the tested values with $\gamma\geq1.25$. For $\mathrm{Re}=1$, however, both methods reach the maximum number of iterations before satisfying the prescribed stopping criterion. Nevertheless, increasing $\gamma$ continues to reduce the error. Decoupled-DFNN still attains velocity and pressure errors of approximately $10^{-6}$ and $10^{-5}$, respectively, compared with approximately $10^{-5}$ and $10^{-4}$ for TransNet.

 %%%%%%%%%%%%%%%%%%%%%%%%% activation %%%%%%%%%%%%%%%%%%%%%%%
\begin{figure}[H]
\centering
    \begin{minipage}[t]{0.45\linewidth}
    {\includegraphics[scale=0.4]{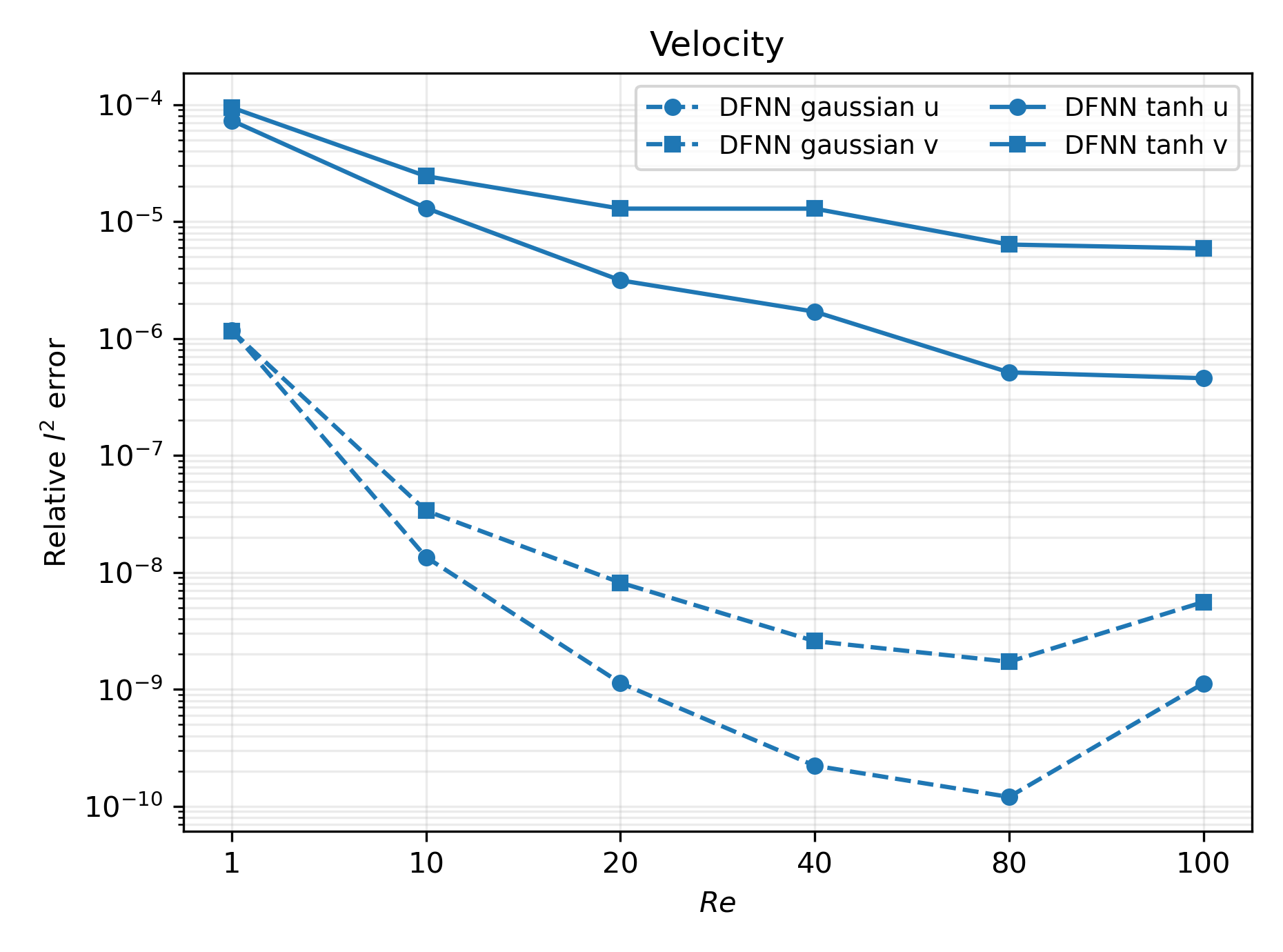}}
  \end{minipage}
  \begin{minipage}[t]{0.45\linewidth}
    {\includegraphics[scale=0.4]{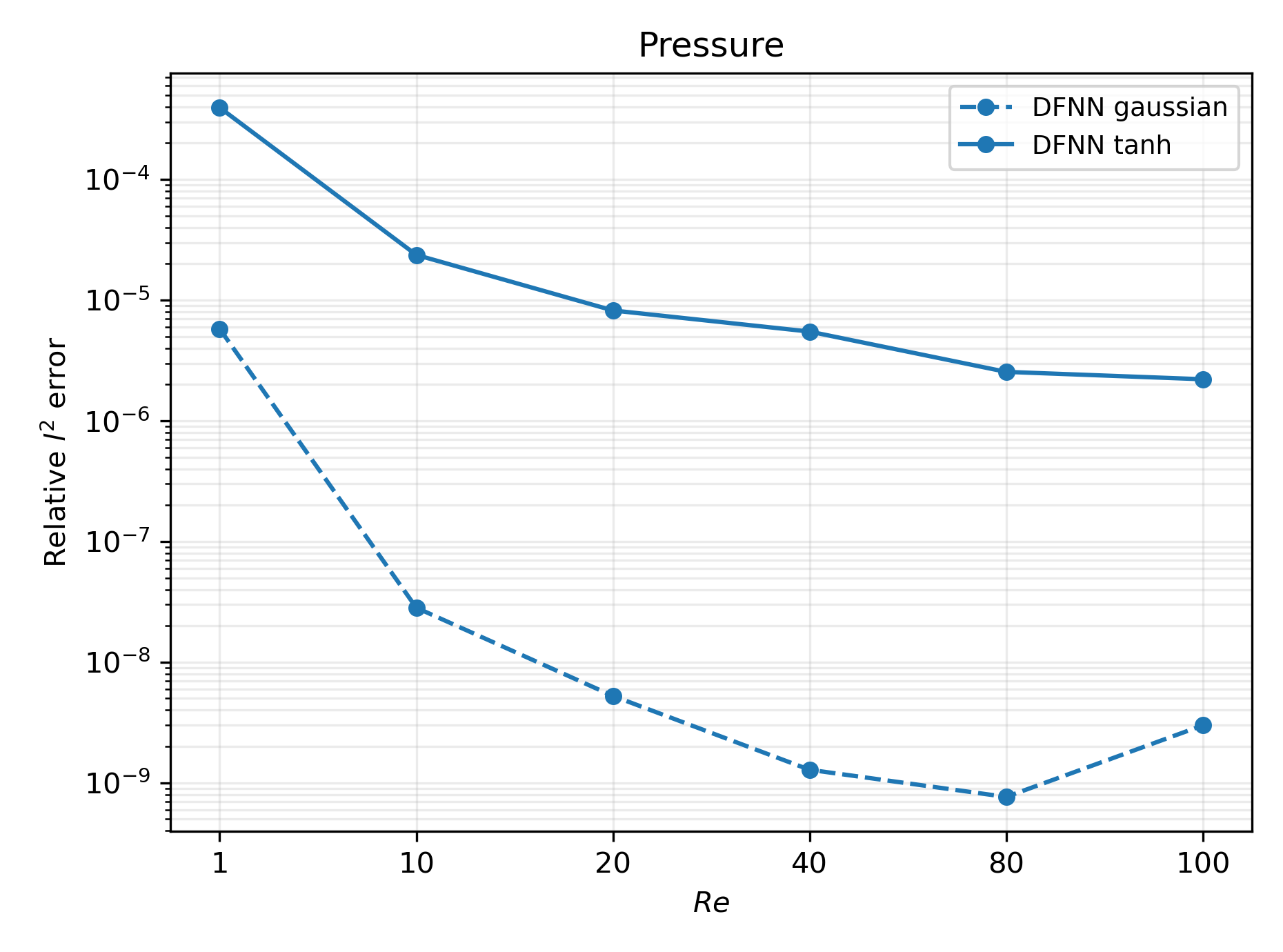}}
  \end{minipage}
  \begin{minipage}[t]{0.45\linewidth}
    {\includegraphics[scale=0.4]{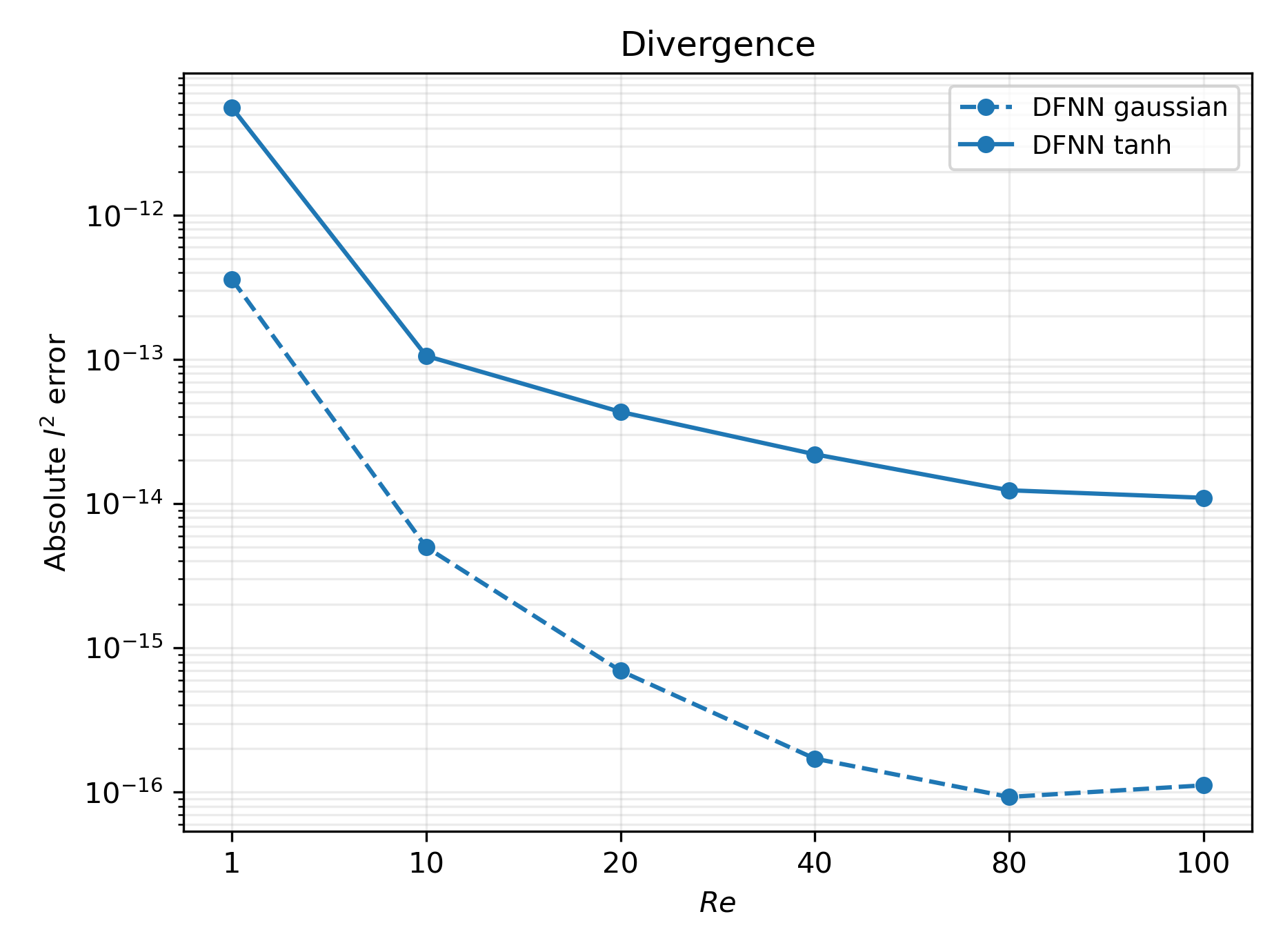}}
  \end{minipage}
  \begin{minipage}[t]{0.45\linewidth}
    {\includegraphics[scale=0.4]{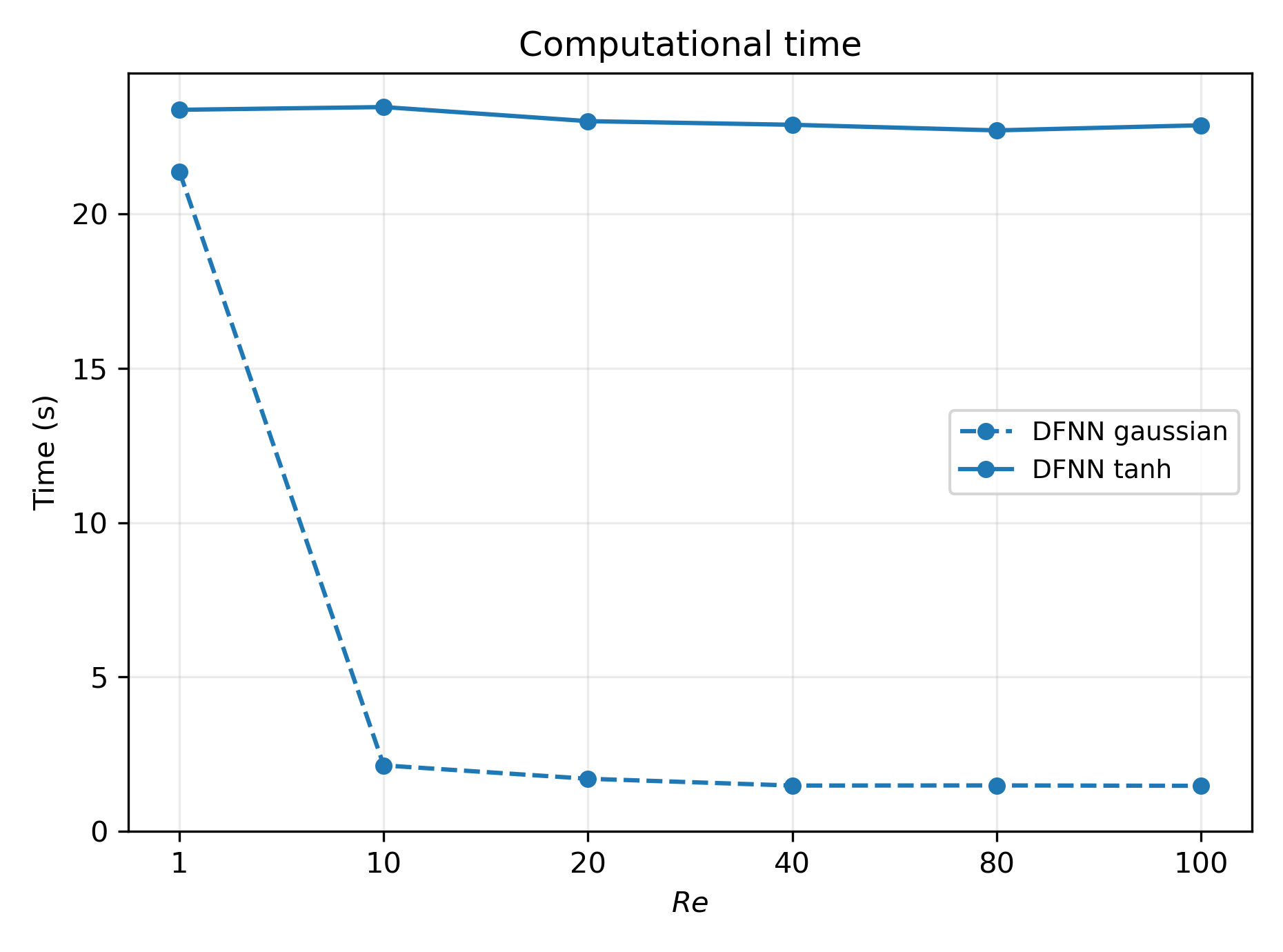}}
  \end{minipage}
    \caption{Numerical results of DFNN for the Kovasznay flow problem with different activation functions.}\label{fig-2d-6}
\end{figure}

\begin{figure}[H]
\centering
    \begin{minipage}[t]{0.45\linewidth}
    {\includegraphics[scale=0.4]{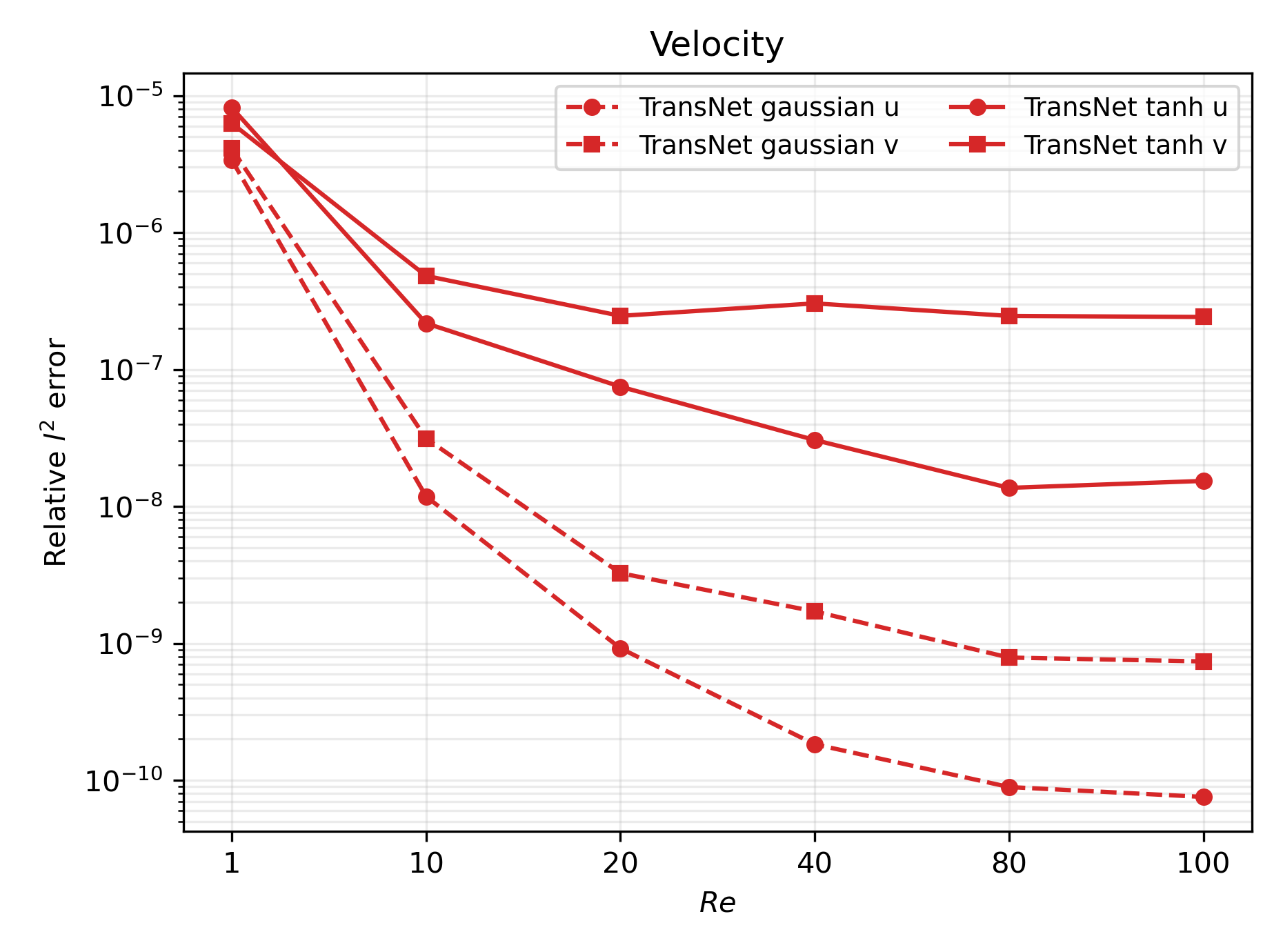}}
  \end{minipage}
  \begin{minipage}[t]{0.45\linewidth}
    {\includegraphics[scale=0.4]{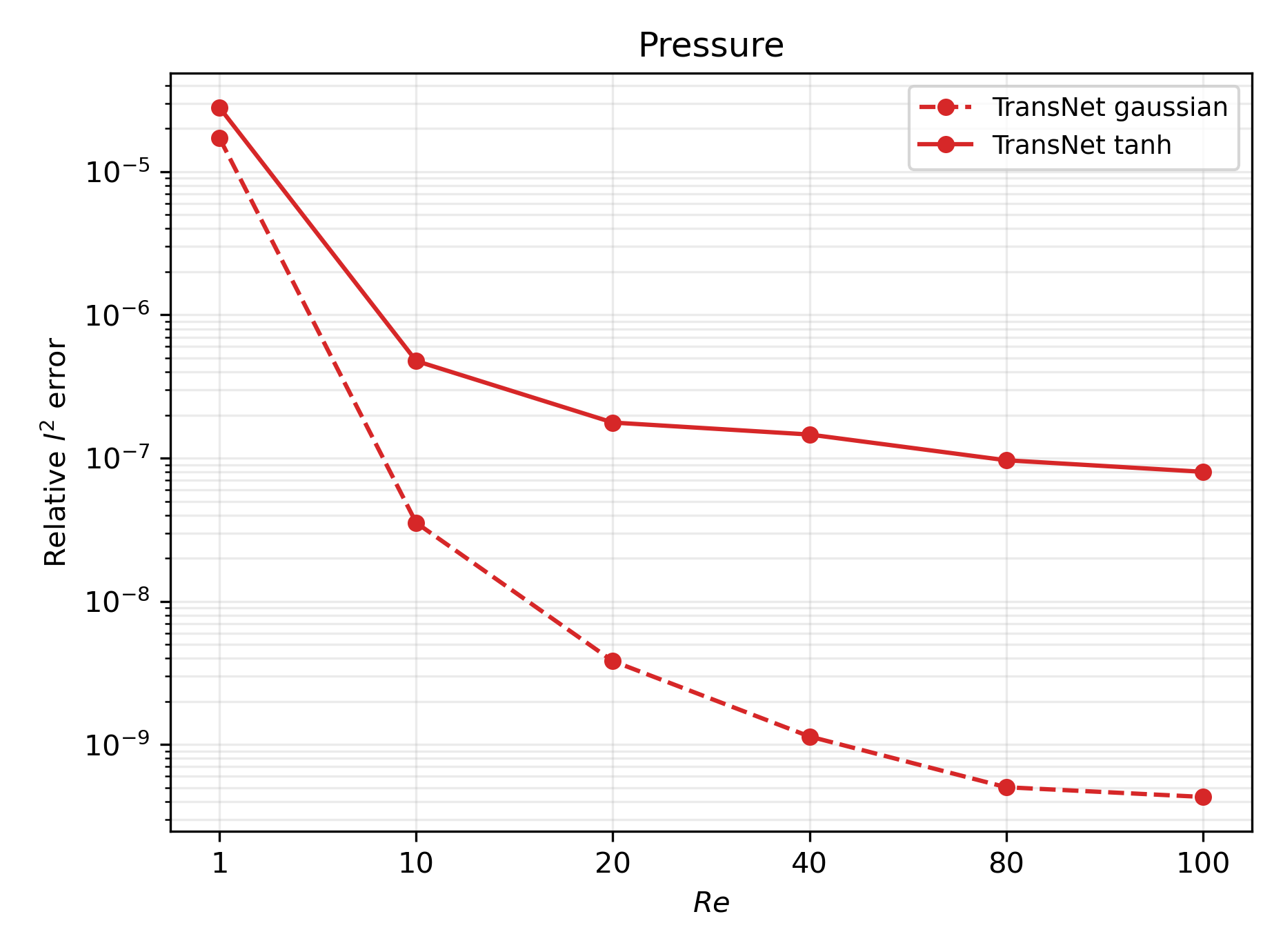}}
  \end{minipage}
  \begin{minipage}[t]{0.45\linewidth}
    {\includegraphics[scale=0.4]{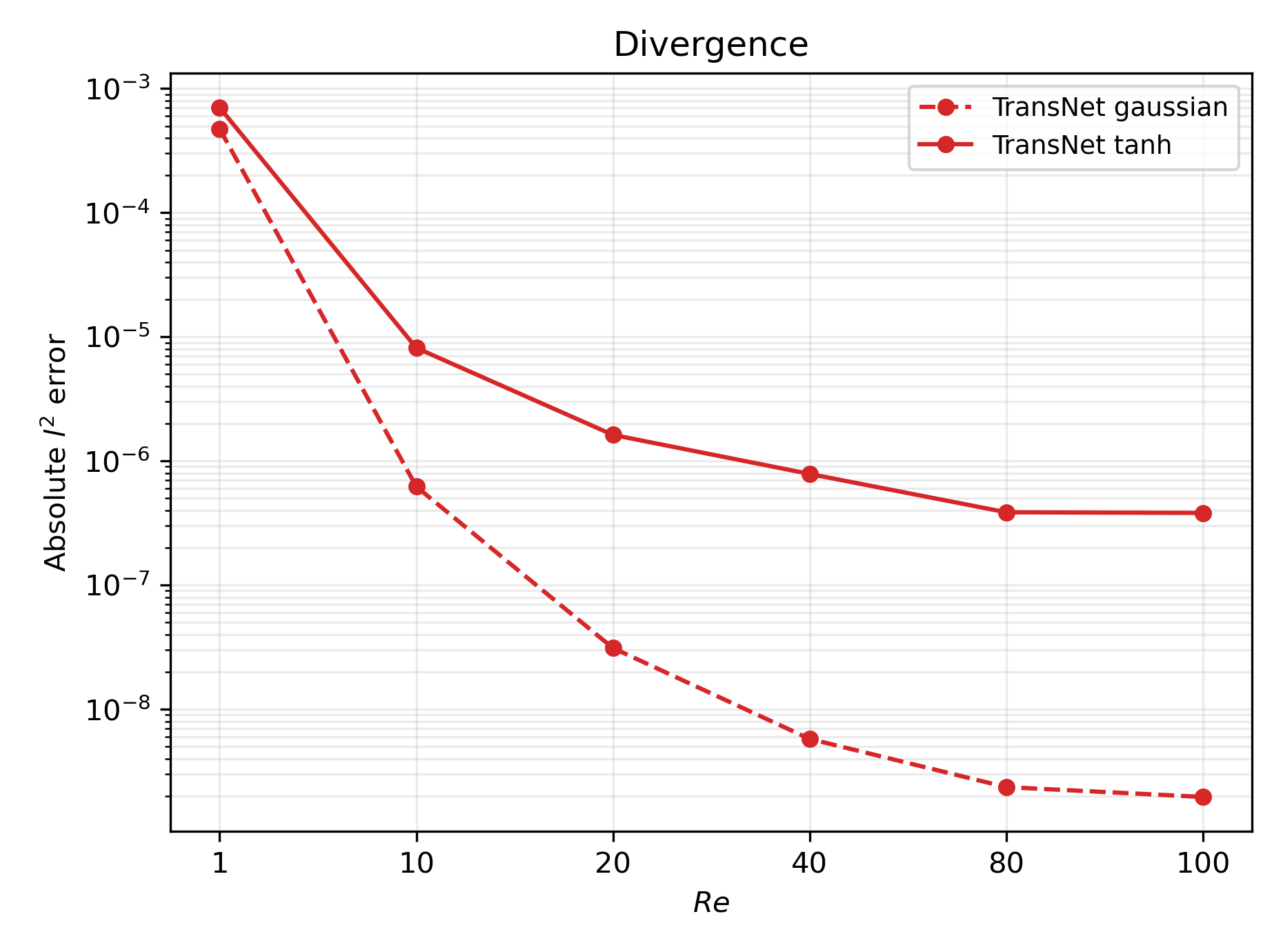}}
  \end{minipage}
  \begin{minipage}[t]{0.45\linewidth}
    {\includegraphics[scale=0.4]{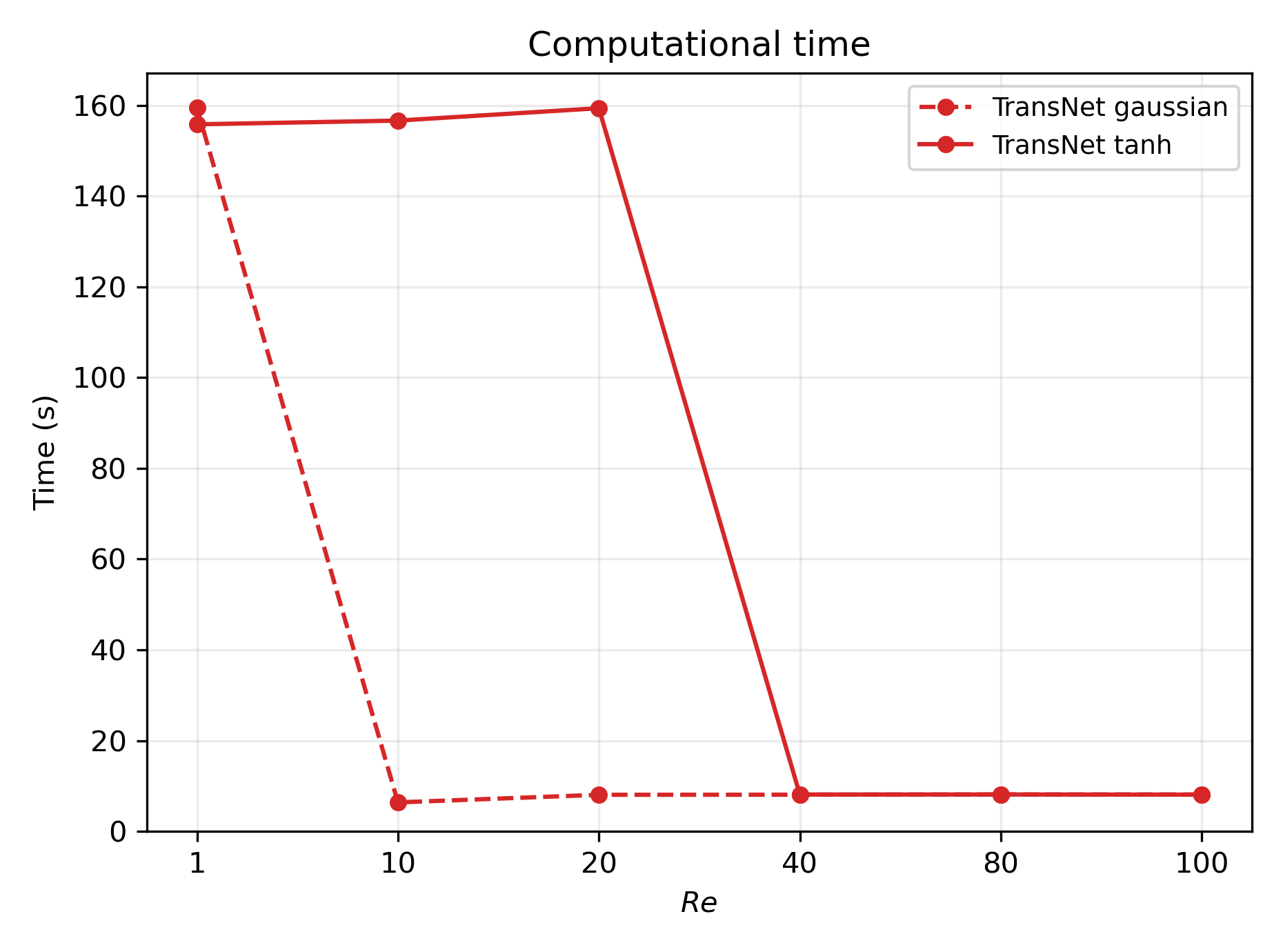}}
  \end{minipage}
    \caption{Numerical results of TransNet for the Kovasznay flow problem with different activation functions.}\label{fig-2d-7}
\end{figure}

 %%%%%%%%%%%%%%%%%%%%%%%%% gam %%%%%%%%%%%%%%%%%%%%%%%
\begin{figure}[H]
\centering
    \begin{minipage}[t]{0.45\linewidth}
    {\includegraphics[scale=0.4]{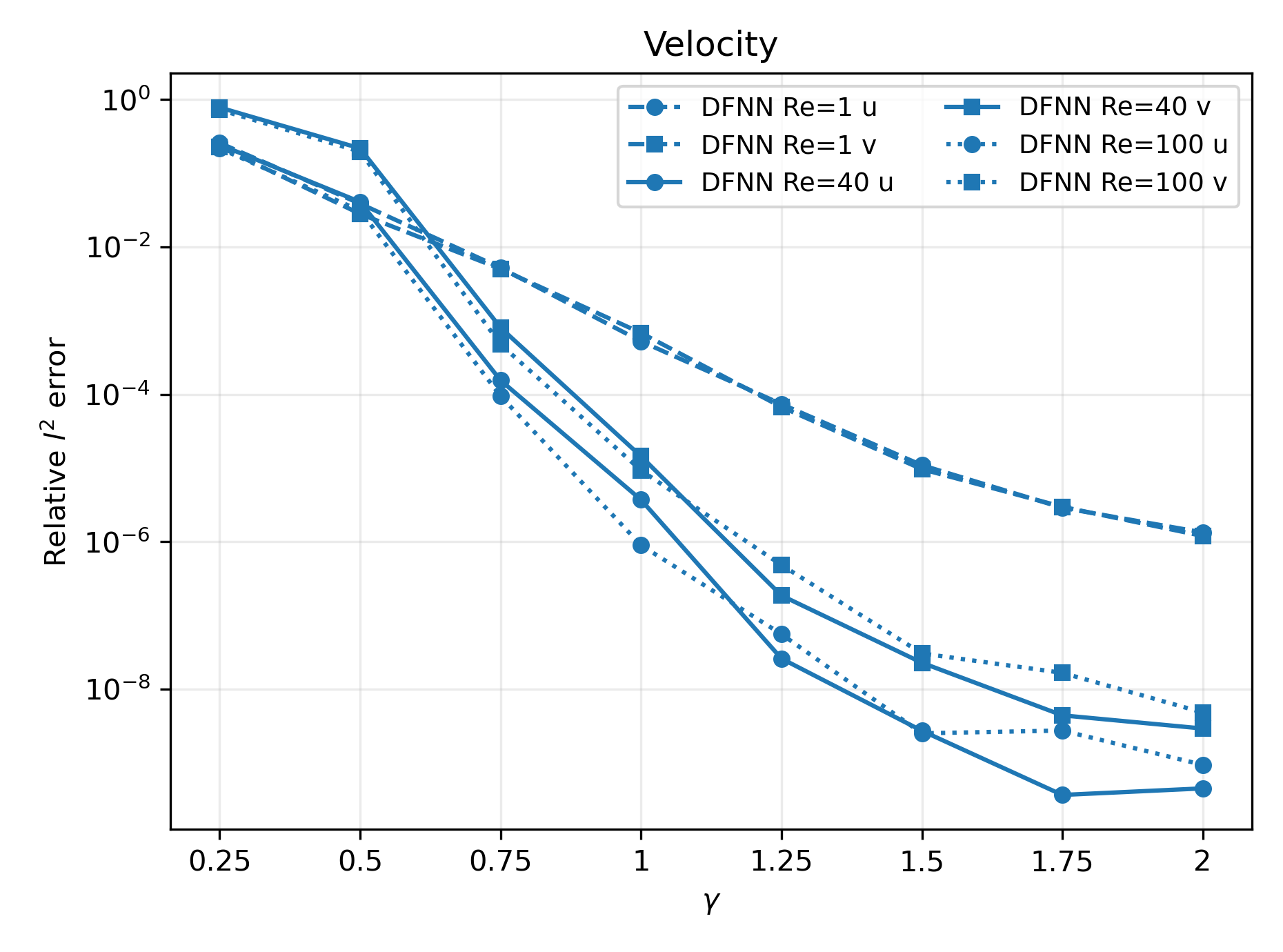}}
  \end{minipage}
  \begin{minipage}[t]{0.45\linewidth}
    {\includegraphics[scale=0.4]{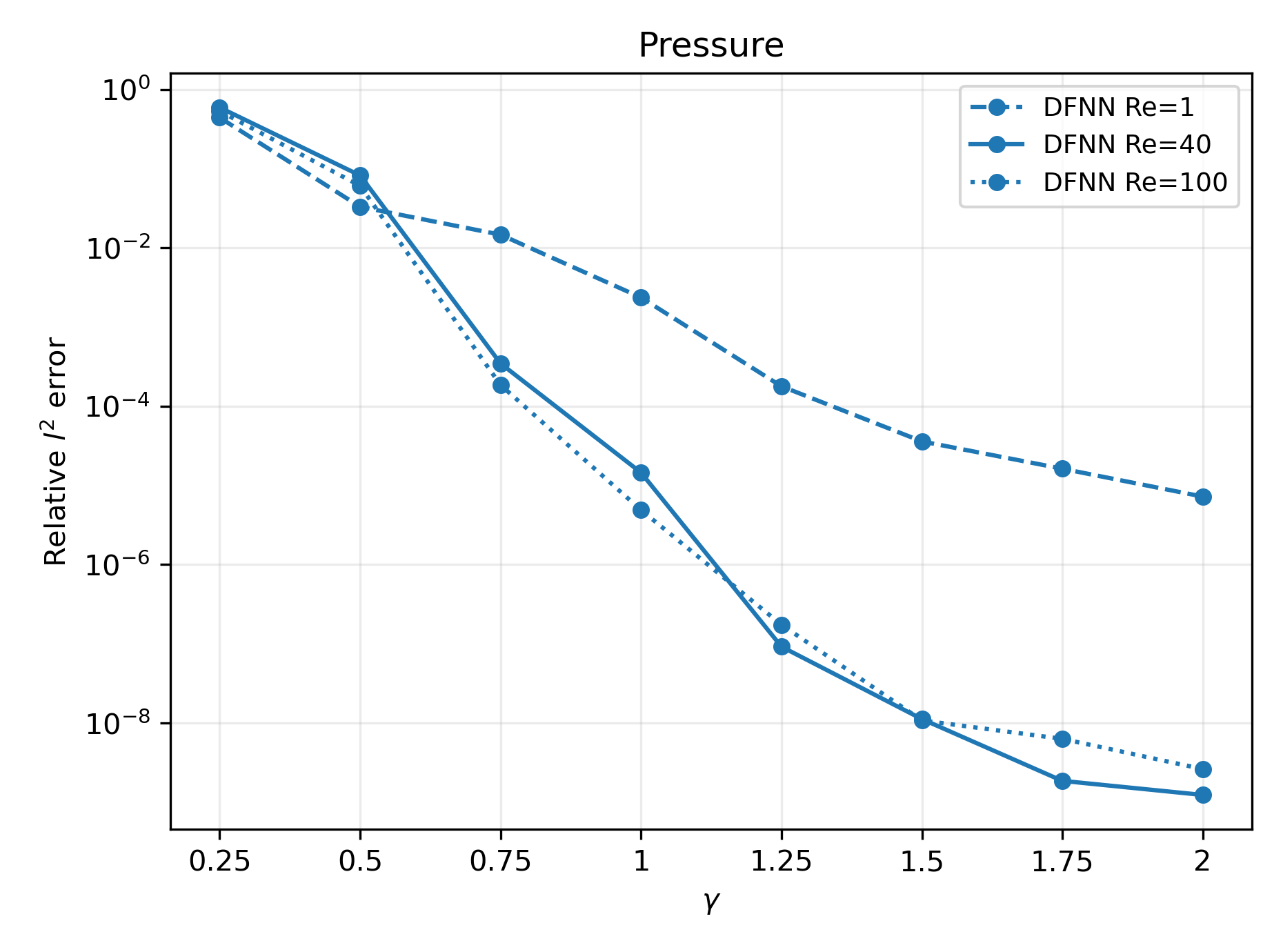}}
  \end{minipage}
  \begin{minipage}[t]{0.45\linewidth}
    {\includegraphics[scale=0.4]{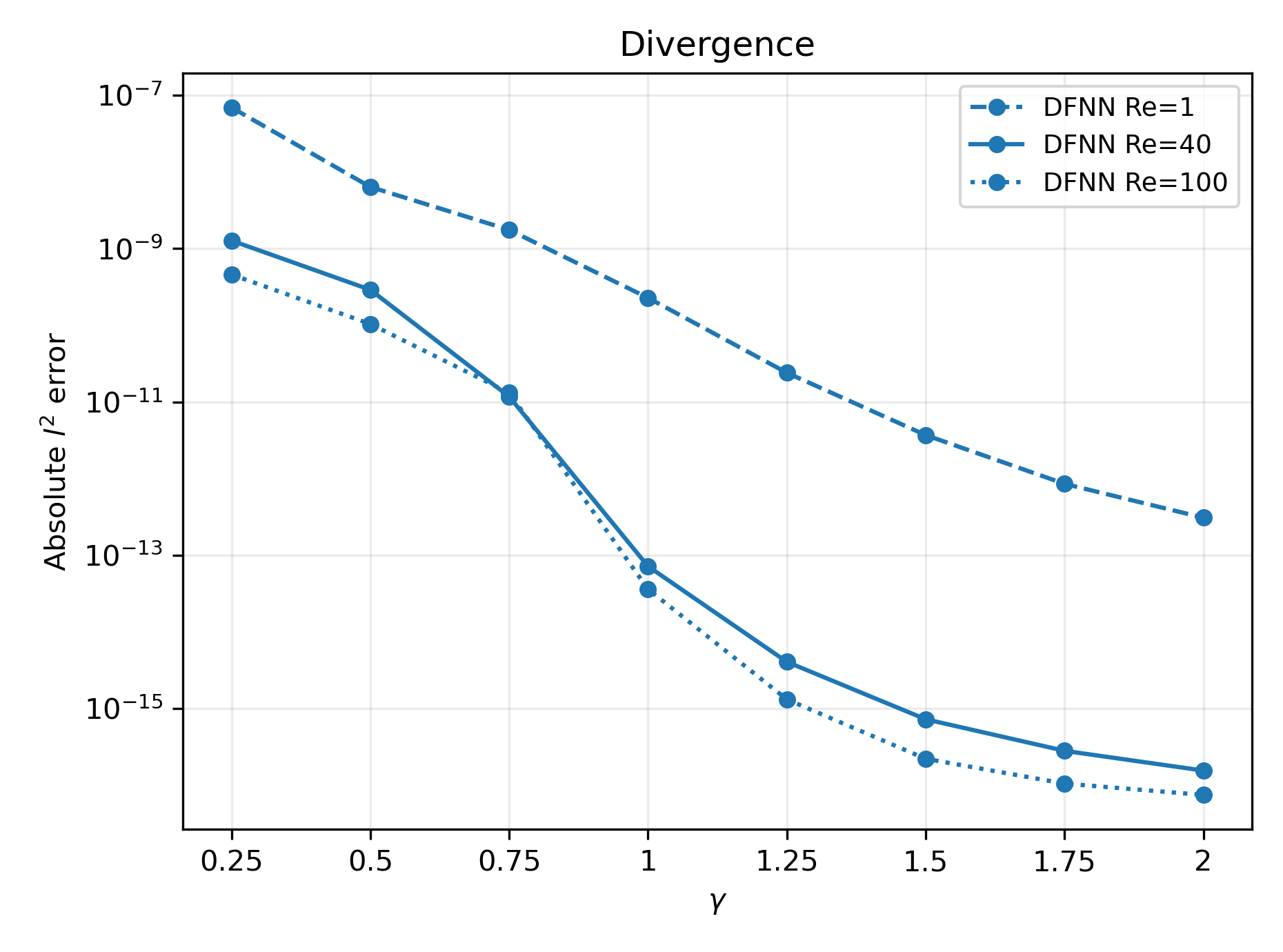}}
  \end{minipage}
  \begin{minipage}[t]{0.45\linewidth}
    {\includegraphics[scale=0.4]{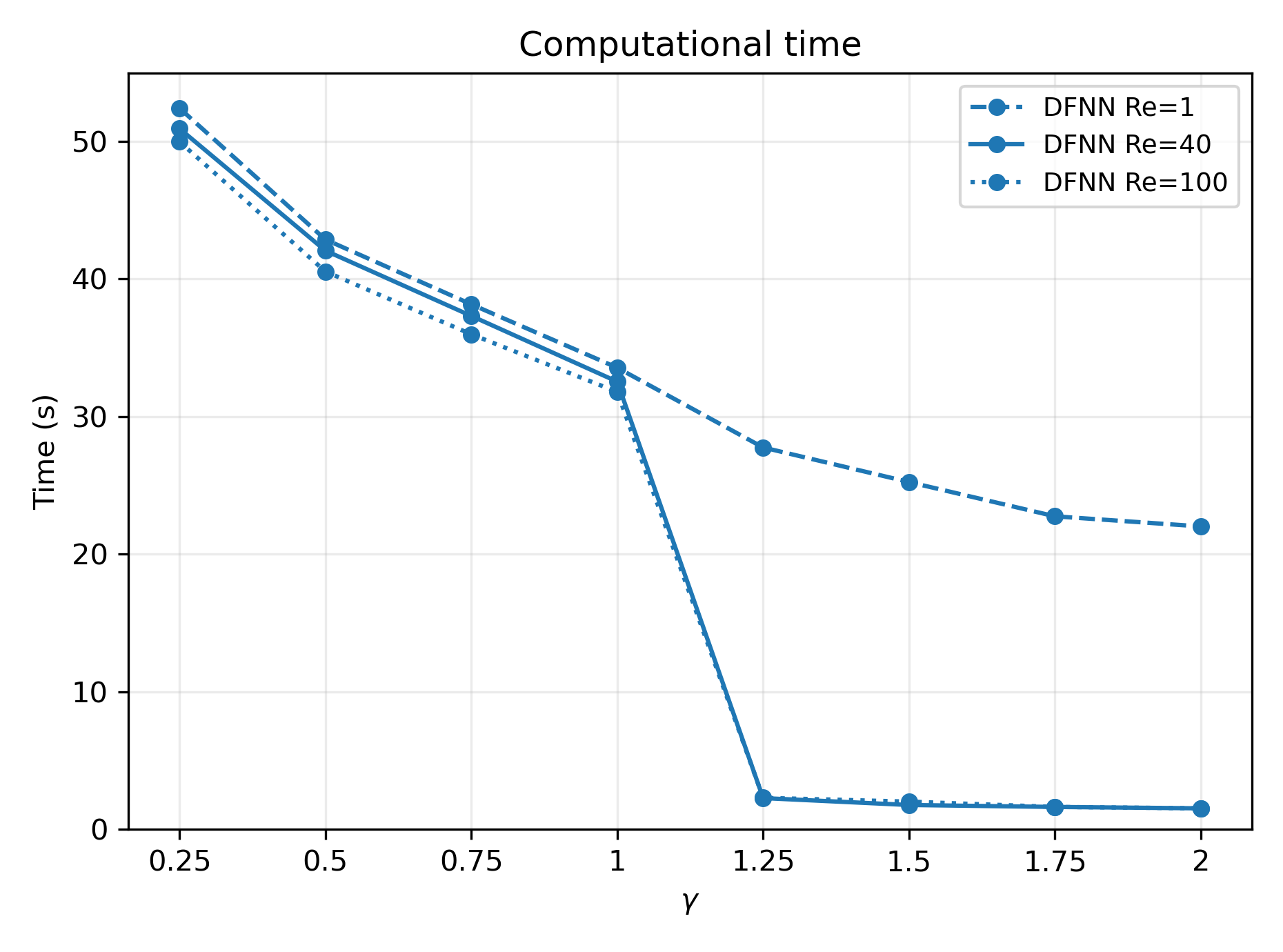}}
  \end{minipage}
    \caption{Numerical results of DFNN for the Kovasznay flow problem with $r=2.5$ varying shape parameter $\gamma$.}\label{fig-2d-8}
\end{figure}

\begin{figure}[H]
\centering
    \begin{minipage}[t]{0.45\linewidth}
    {\includegraphics[scale=0.4]{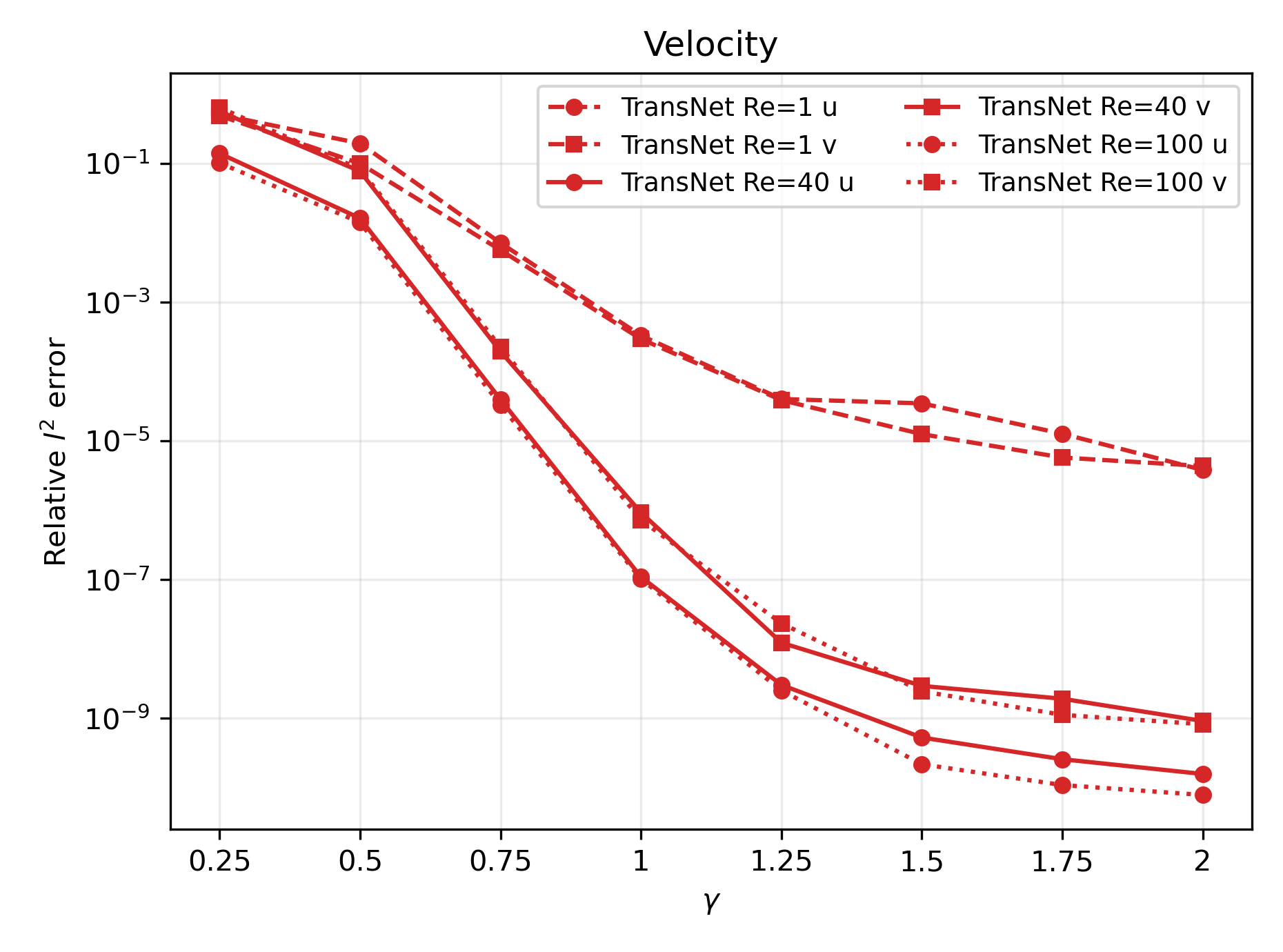}}
  \end{minipage}
  \begin{minipage}[t]{0.45\linewidth}
    {\includegraphics[scale=0.4]{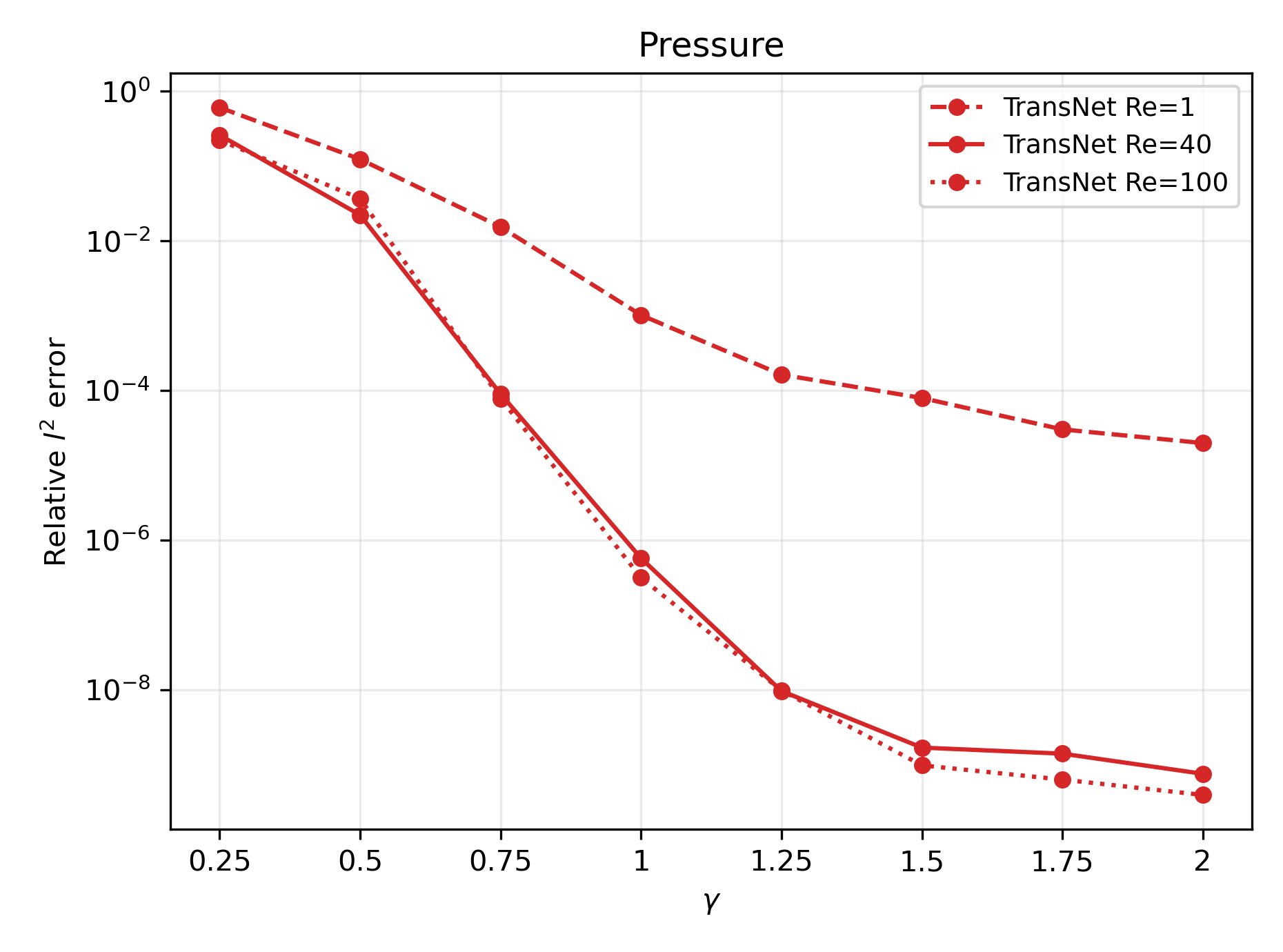}}
  \end{minipage}
  \begin{minipage}[t]{0.45\linewidth}
    {\includegraphics[scale=0.4]{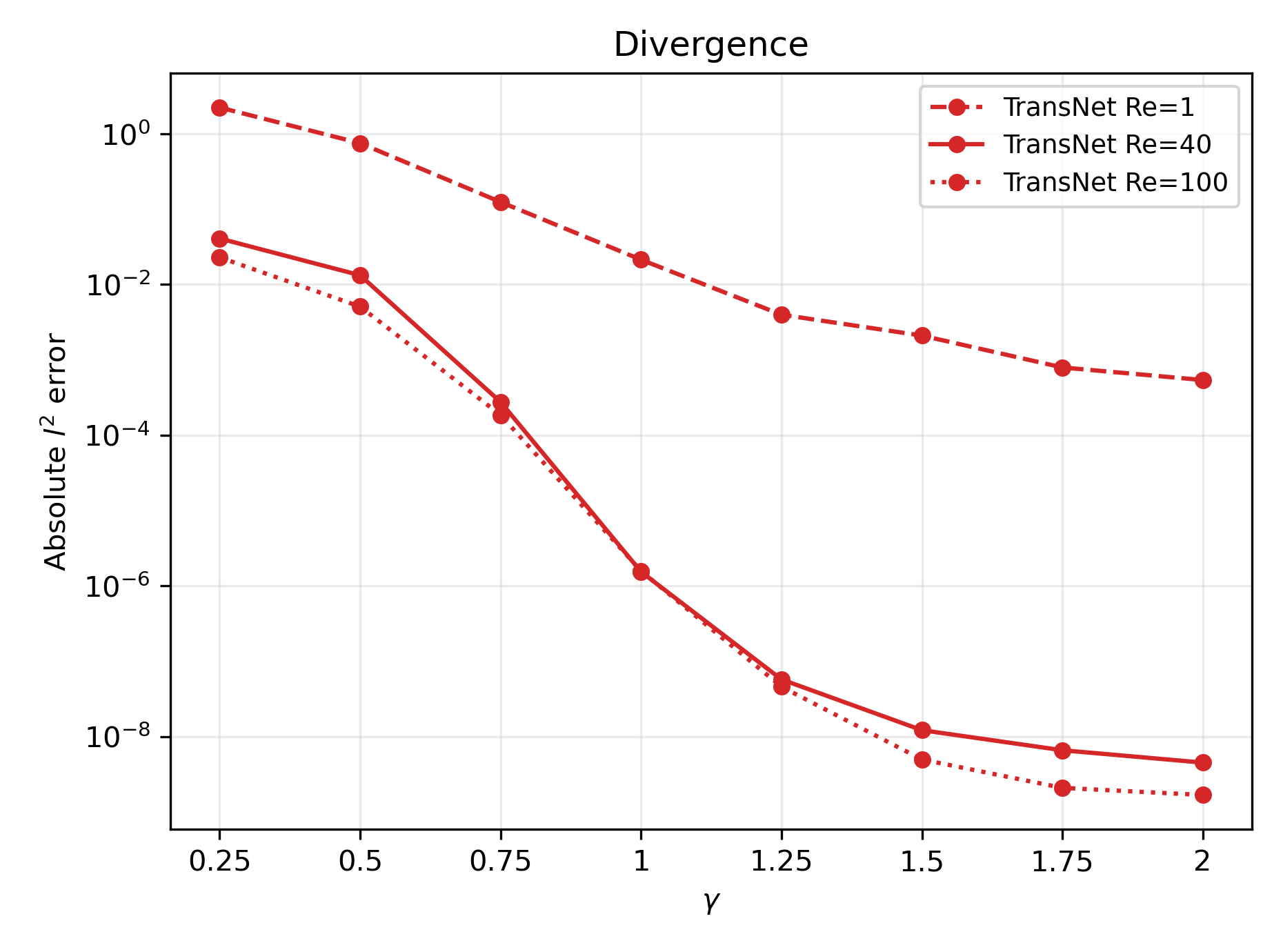}}
  \end{minipage}
  \begin{minipage}[t]{0.45\linewidth}
    {\includegraphics[scale=0.4]{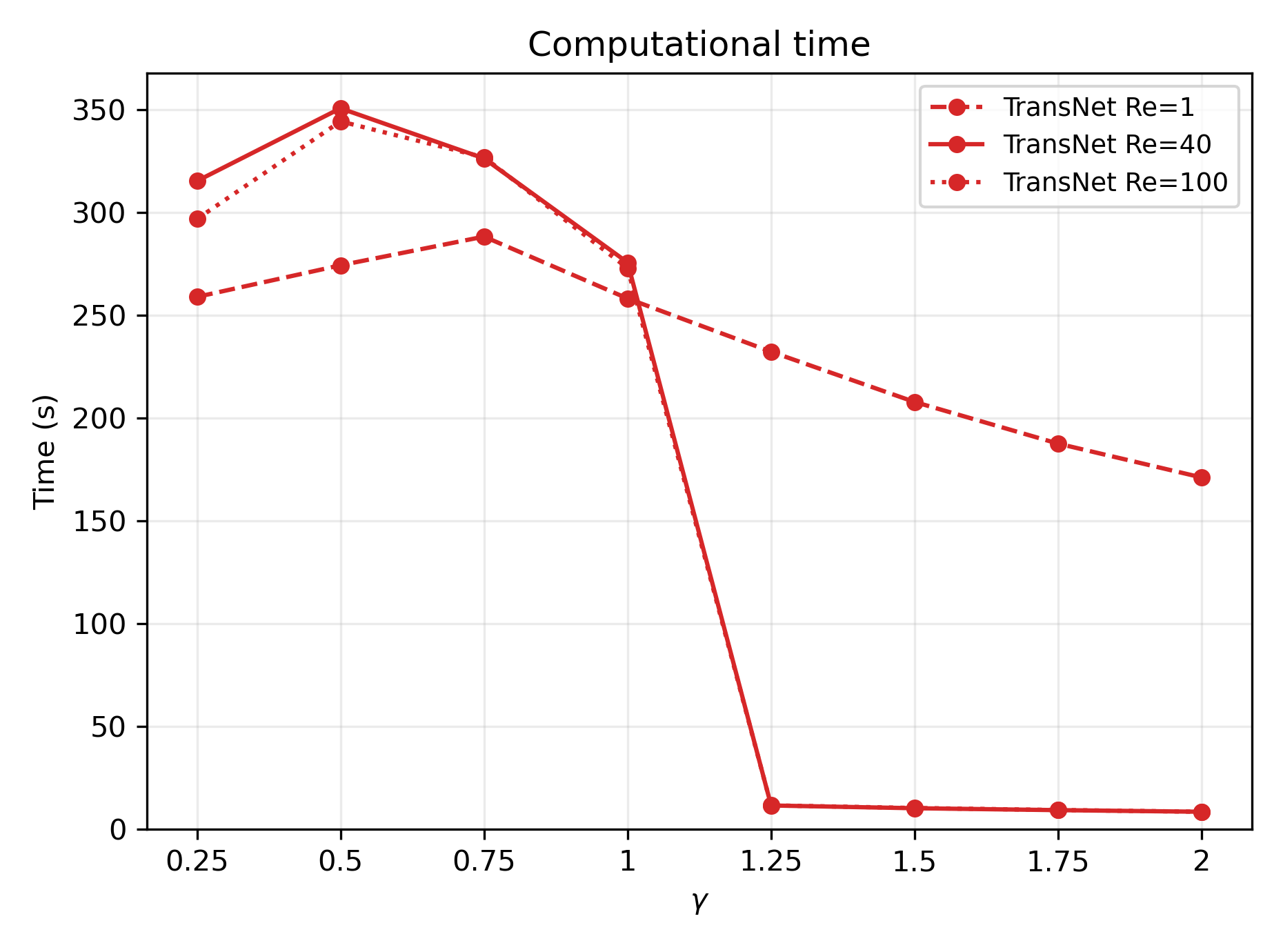}}
  \end{minipage}
    \caption{Numerical results of TransNet for the Kovasznay flow problem with $r = 2.5$ and varying shape parameter $\gamma$.}\label{fig-2d-9}
\end{figure}

 %%%%%%%%%%%%%%%%%%%%%%%%% r %%%%%%%%%%%%%%%%%%%%%%%
\begin{figure}[H]
\centering
    \begin{minipage}[t]{0.45\linewidth}
    {\includegraphics[scale=0.4]{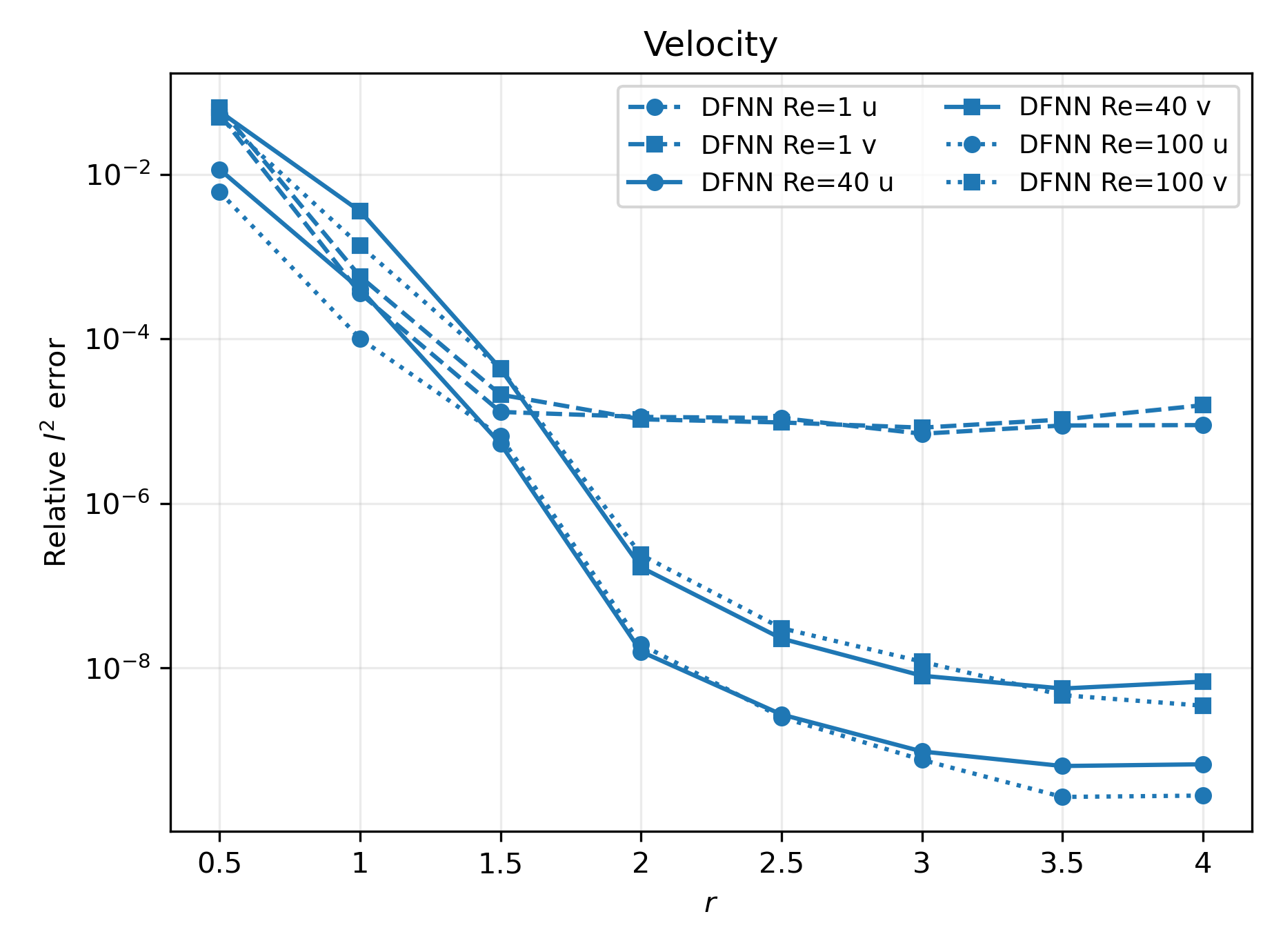} }
  \end{minipage}
  \begin{minipage}[t]{0.45\linewidth}
    {\includegraphics[scale=0.4]{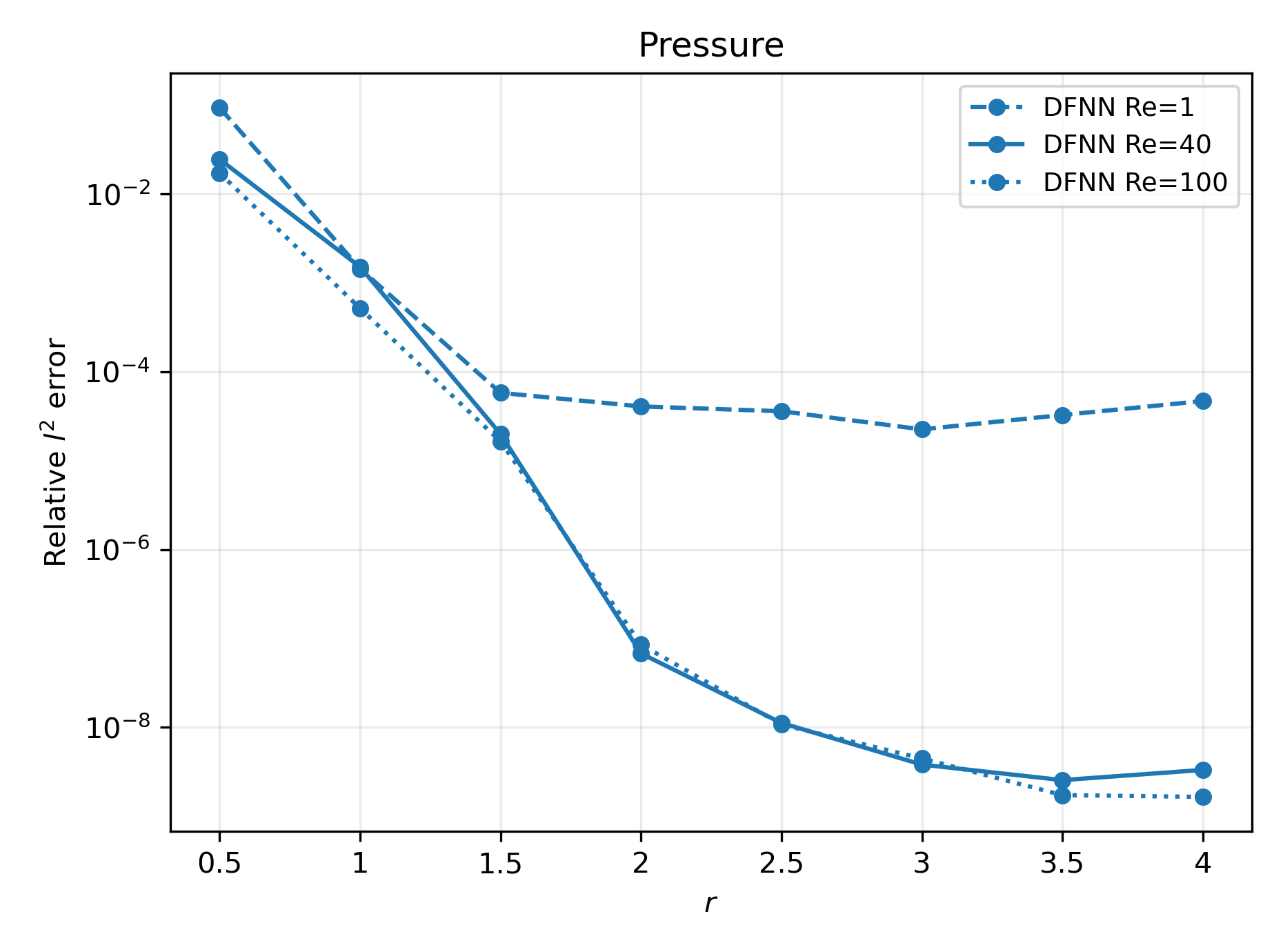} }
  \end{minipage}
  \begin{minipage}[t]{0.45\linewidth}
    {\includegraphics[scale=0.4]{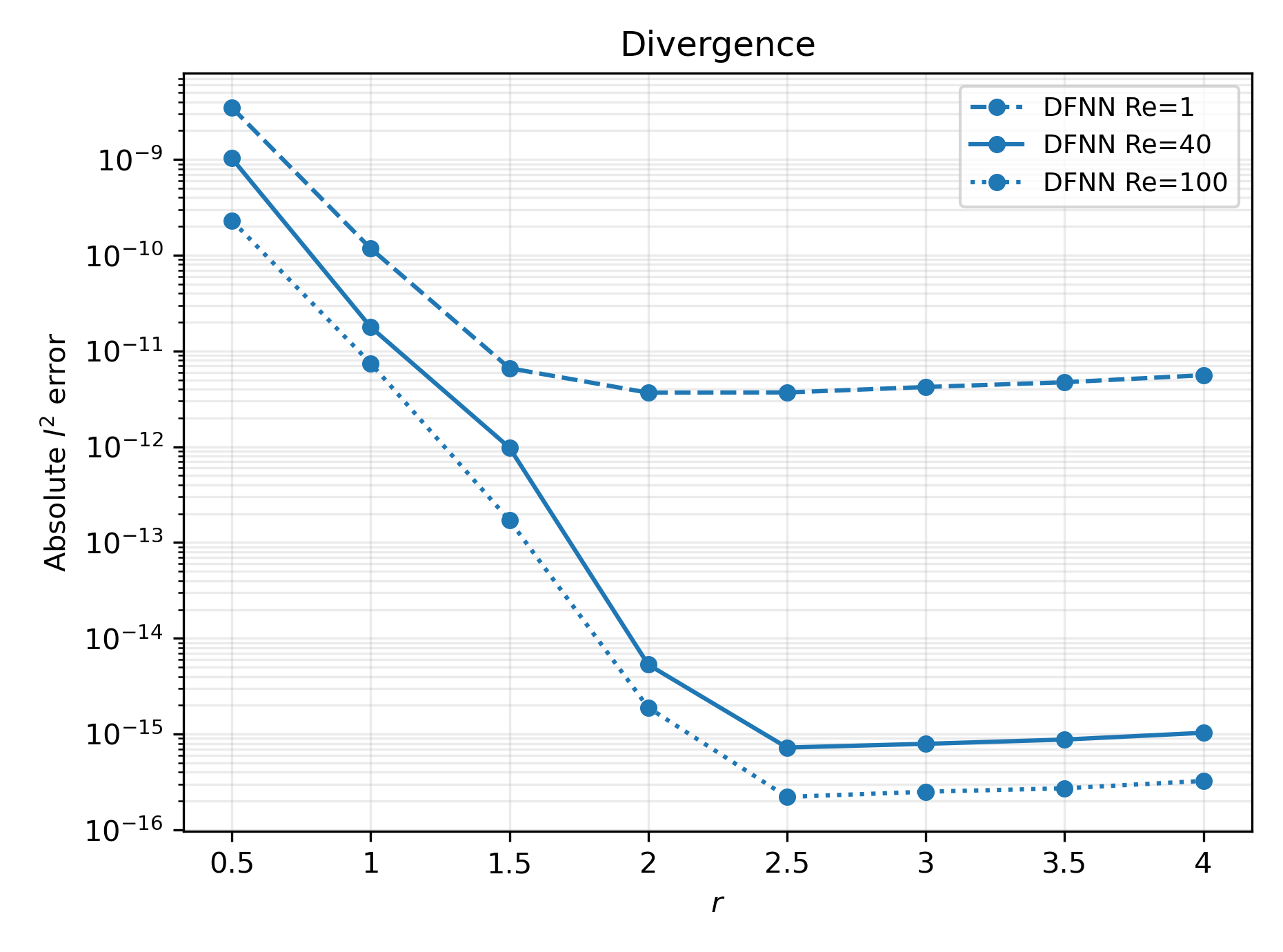} }
  \end{minipage}
  \begin{minipage}[t]{0.45\linewidth}
    {\includegraphics[scale=0.4]{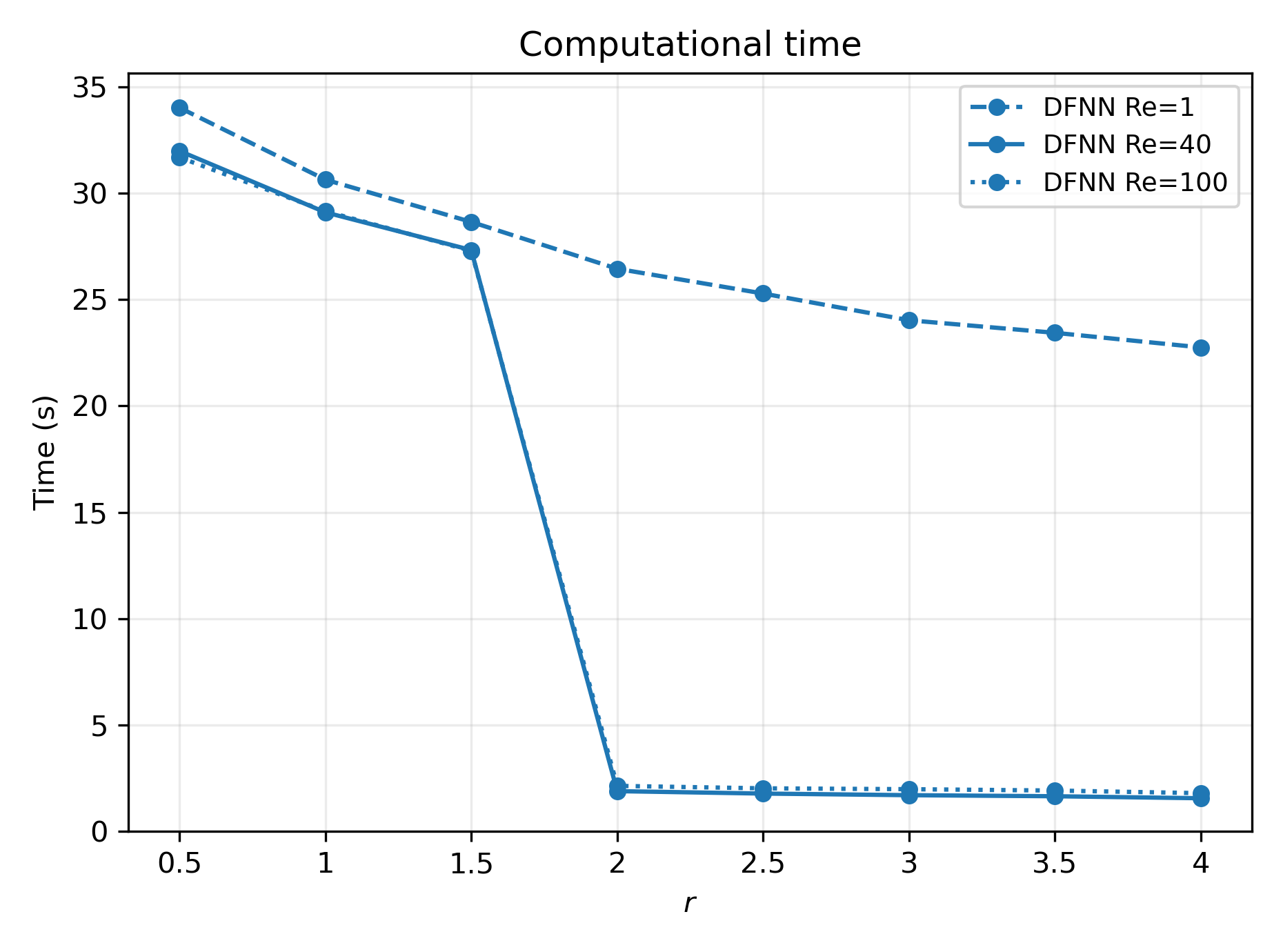} }
  \end{minipage}
    \caption{Numerical results of DFNN for the Kovasznay flow problem with $\gamma =1.5$ and varying shape parameter $r$.}\label{fig-2d-10}
\end{figure}

\begin{figure}[H]
\centering
    \begin{minipage}[t]{0.45\linewidth}
    {\includegraphics[scale=0.4]{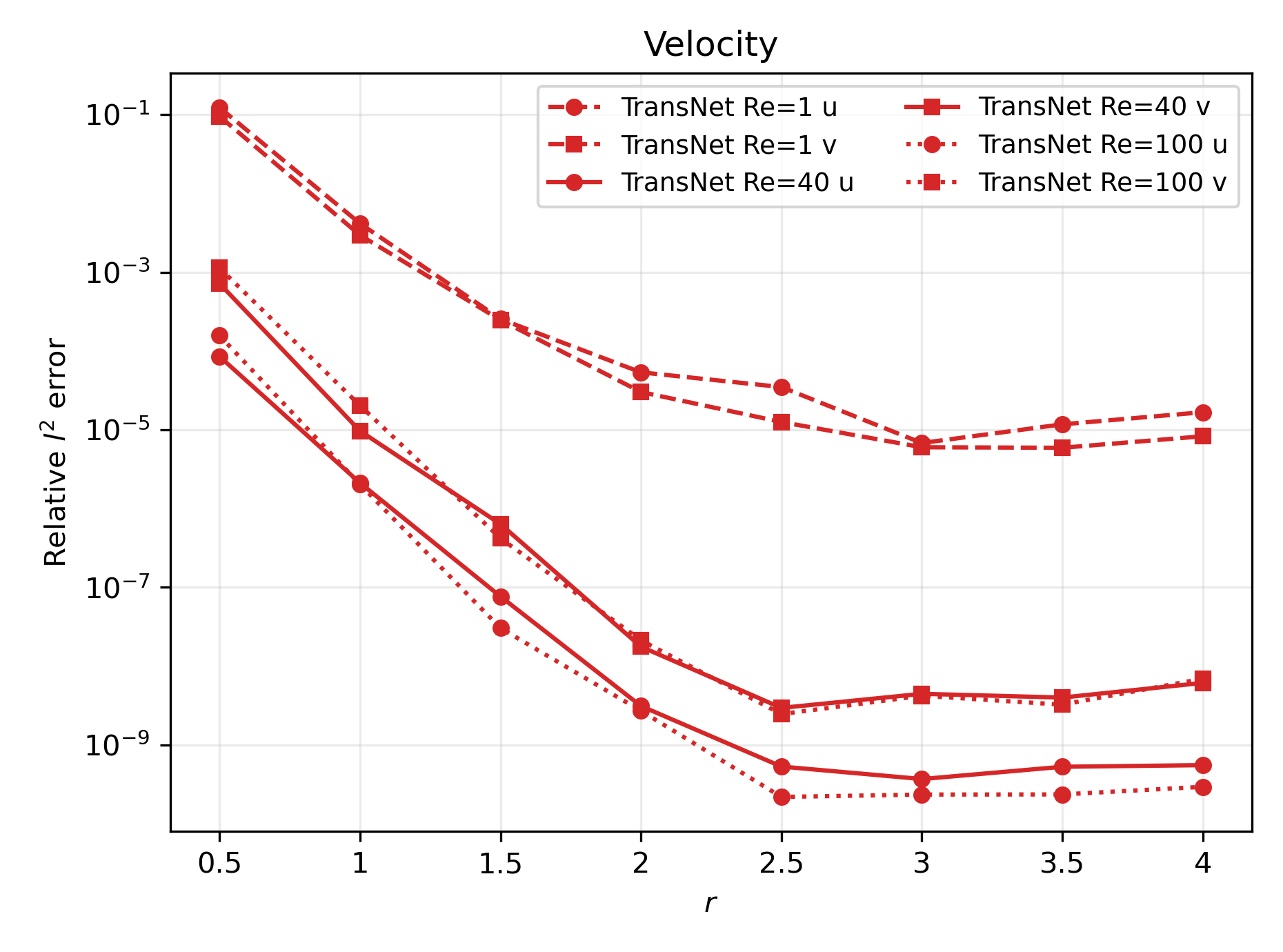} }
  \end{minipage}
  \begin{minipage}[t]{0.45\linewidth}
    {\includegraphics[scale=0.4]{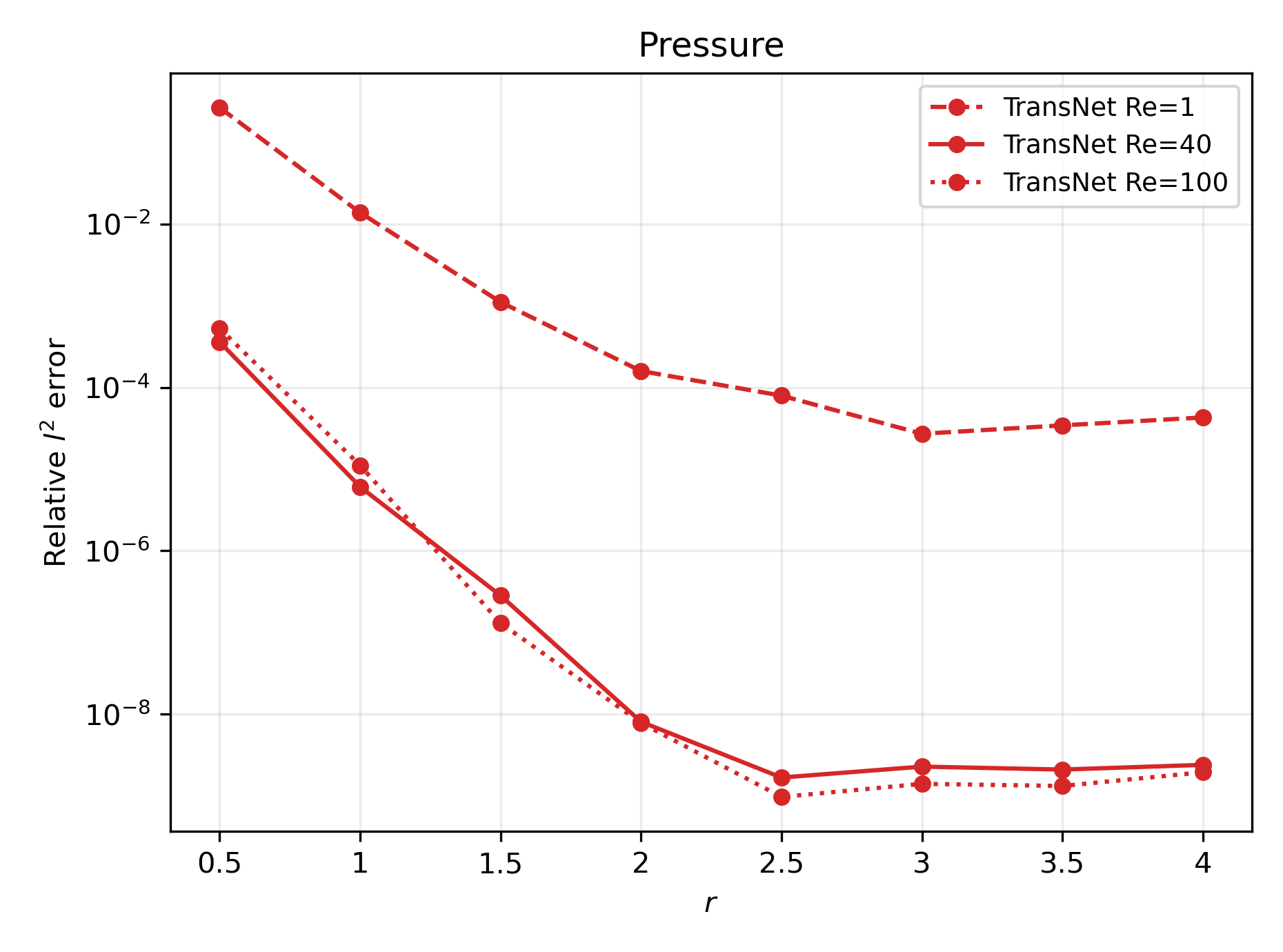} }
  \end{minipage}
  \begin{minipage}[t]{0.45\linewidth}
    {\includegraphics[scale=0.4]{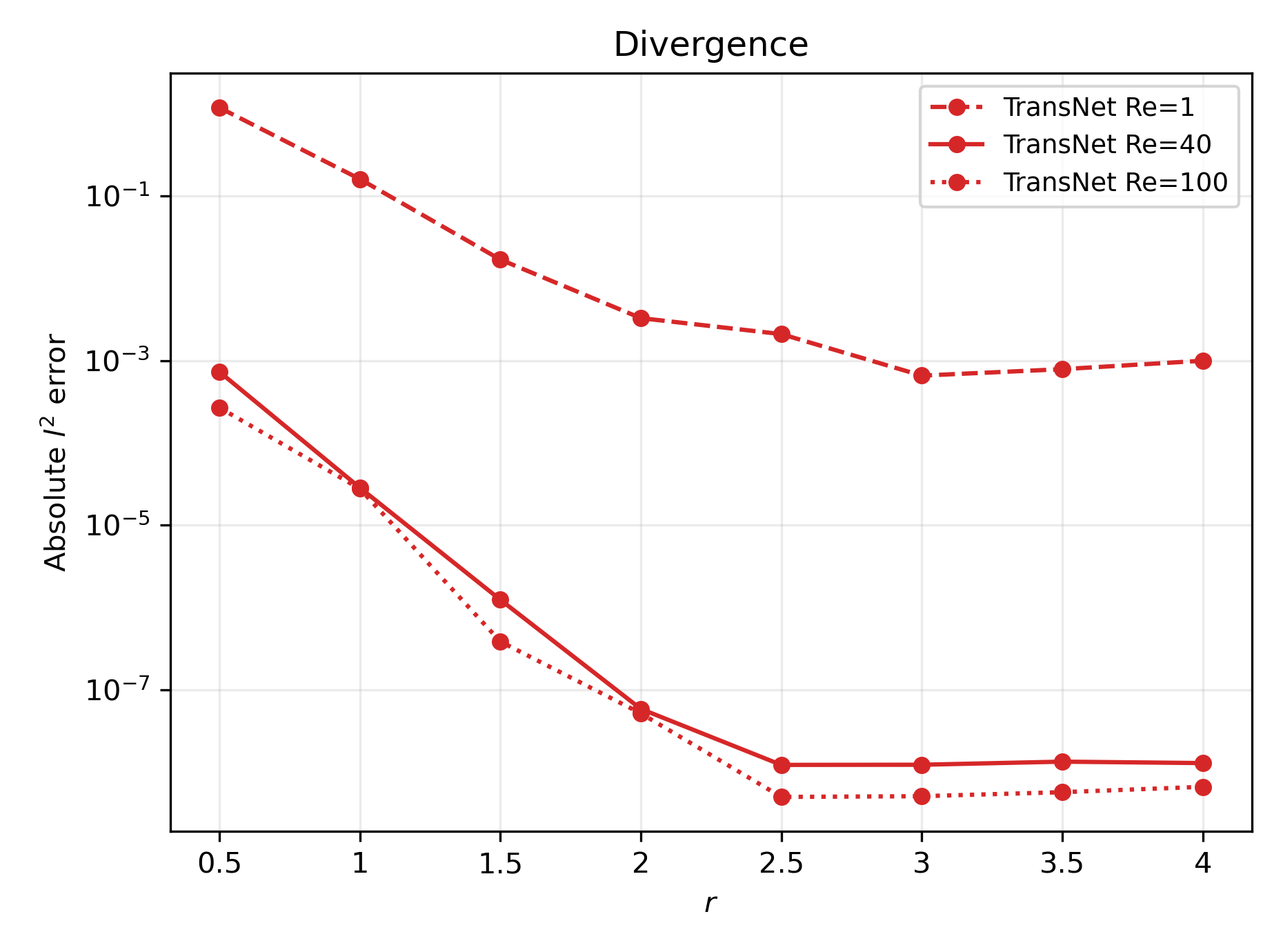} }
  \end{minipage}
  \begin{minipage}[t]{0.45\linewidth}
    {\includegraphics[scale=0.4]{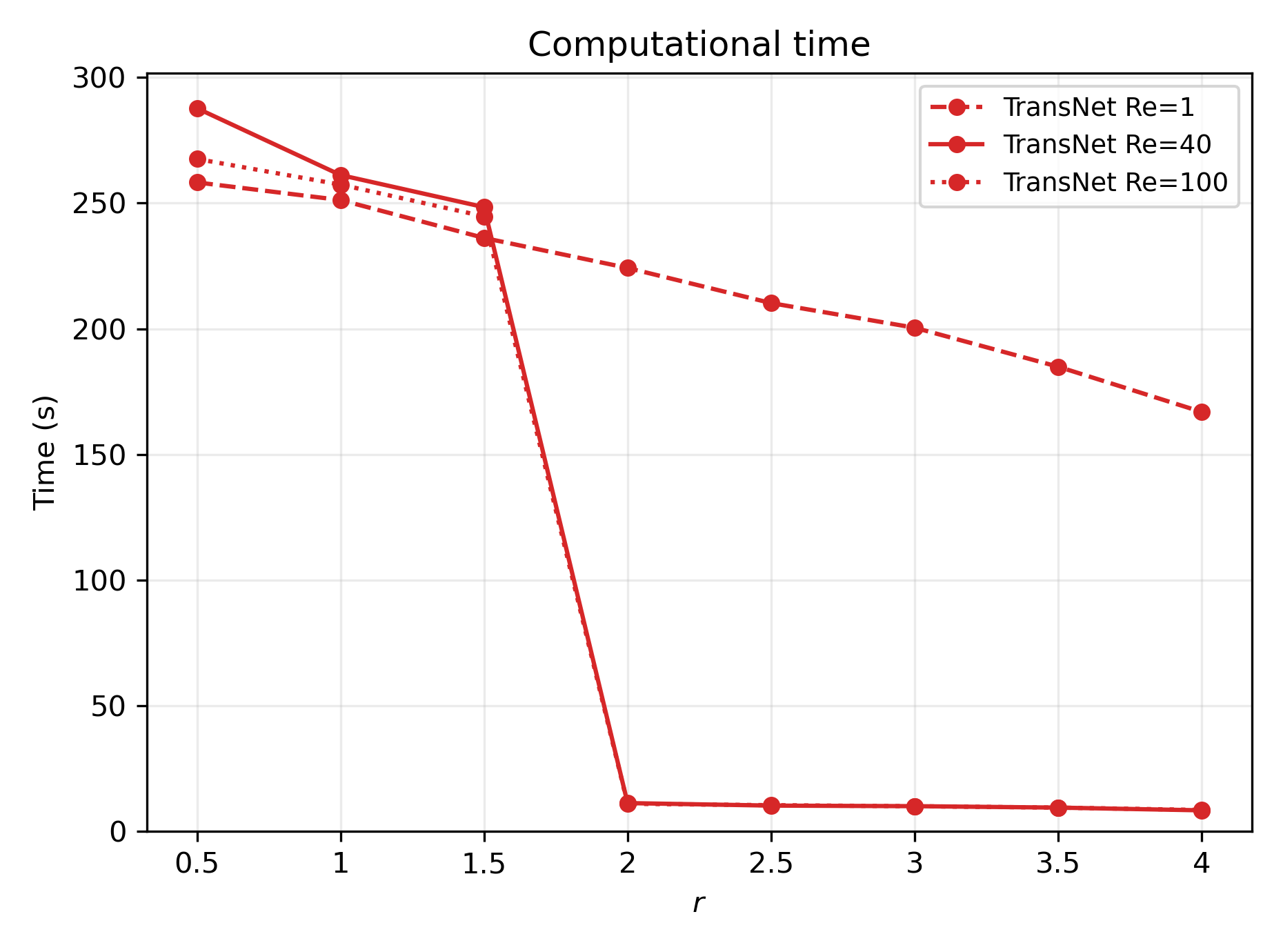} }
  \end{minipage}
    \caption{Numerical results of TransNet for the Kovasznay flow problem with $\gamma=1.5$ and varying shape parameter $r$.}\label{fig-2d-11}
\end{figure}

A similar trend is observed for the parameter $r$, as shown in Figs.~\ref{fig-2d-10}--\ref{fig-2d-11}. The accuracy improves as $r$ increases, while further improvement becomes marginal beyond $r=2.5$. For $\mathrm{Re}=40$ and $100$, both methods satisfy the prescribed stopping criterion for the tested values with $r\geq2$, whereas for $\mathrm{Re}=1$, both methods reach the maximum number of iterations before satisfying the stopping criterion, with the errors eventually decreasing to the level of $10^{-5}$--$10^{-4}$.

\subsection{3D Stokes problem}\label{sub:3d_stokes_problem}
Consider the Stokes problem defined on $\Omega = (0,1)^3$, for which the analytical solution is given by
\begin{align*}
&u = e^{\cos(\pi y)}\sin(\pi z),
\quad
v = e^{\cos(\pi z)}\sin(\pi x),
\quad
w = e^{\cos(\pi x)}\sin(\pi y),
\\
&p = e^{\cos(\pi x) + \sin(\pi y)} + e^{\cos(\pi z) + \sin(\pi x)}.
\end{align*}

In this case, we employ a Halton sampling strategy \cite{Caflisch1998} to select $10~000$ interior points and $5~400$ boundary points (distributed as $30 \times 30 \times 6$ uniform grids) for the training set. Additionally, $8~192$ interior Halton samples are chosen as the test set.

Consistent with the 2D results, we first compare the numerical results of Decoupled-DFNN and TransNet using varying numbers of basis functions with $\gamma=1$ and $r=3$, as shown in Fig.~\ref{fig-3d-1}. We observe that the approximate velocity field of Decoupled-DFNN is more accurate than that of TransNet for all values of $M$. For the pressure field, both methods show comparable accuracy. Furthermore, the Decoupled-DFNN exhibits a significant advantage in structure-preserving capability, achieving a divergence error of only $10^{-14}$ and thus satisfying the incompressibility condition almost exactly.

\begin{figure}[H]
\centering
    \begin{minipage}[t]{0.45\linewidth}
    {\includegraphics[scale=0.4]{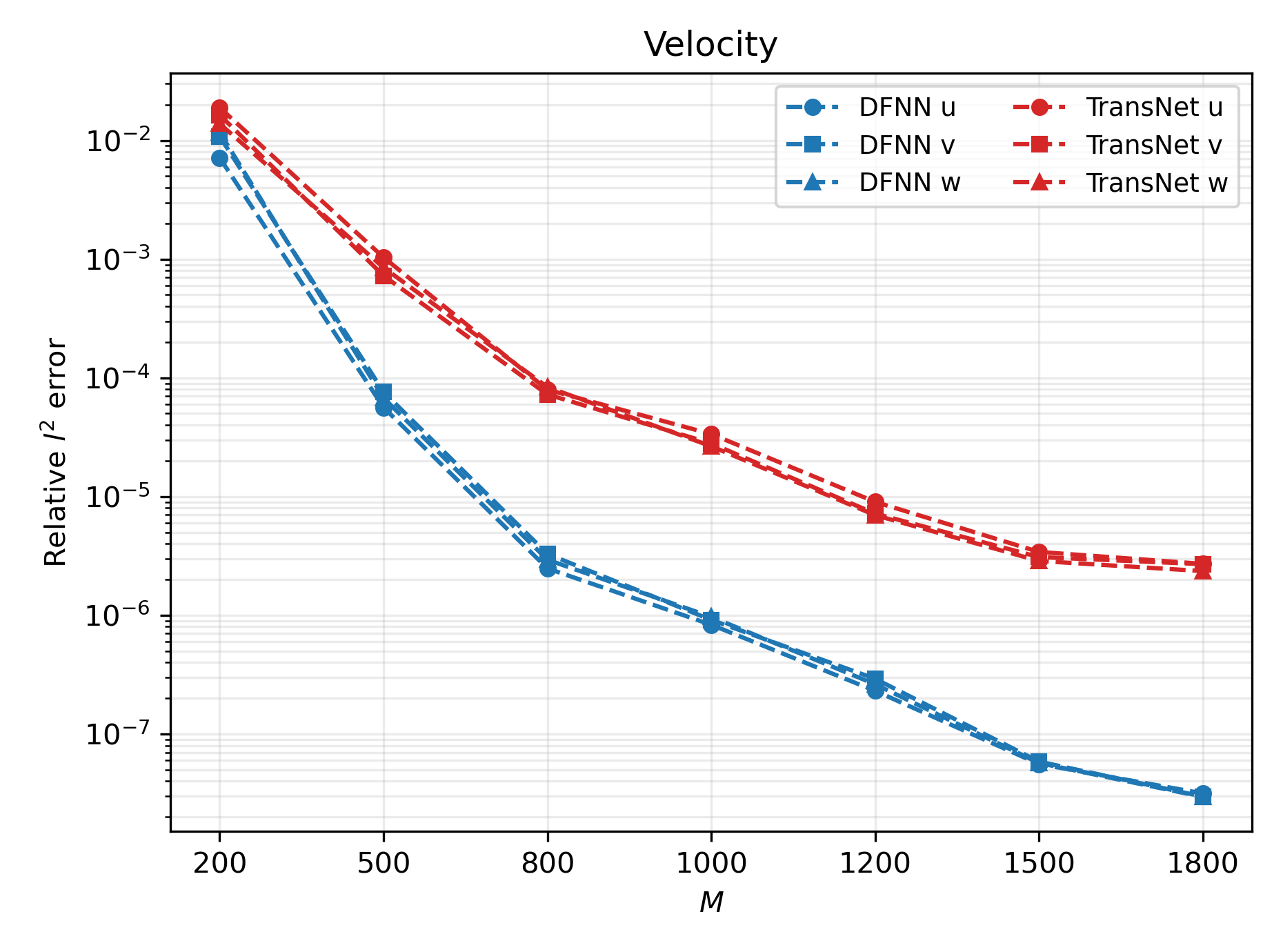}}
  \end{minipage}
  \begin{minipage}[t]{0.45\linewidth}
    {\includegraphics[scale=0.4]{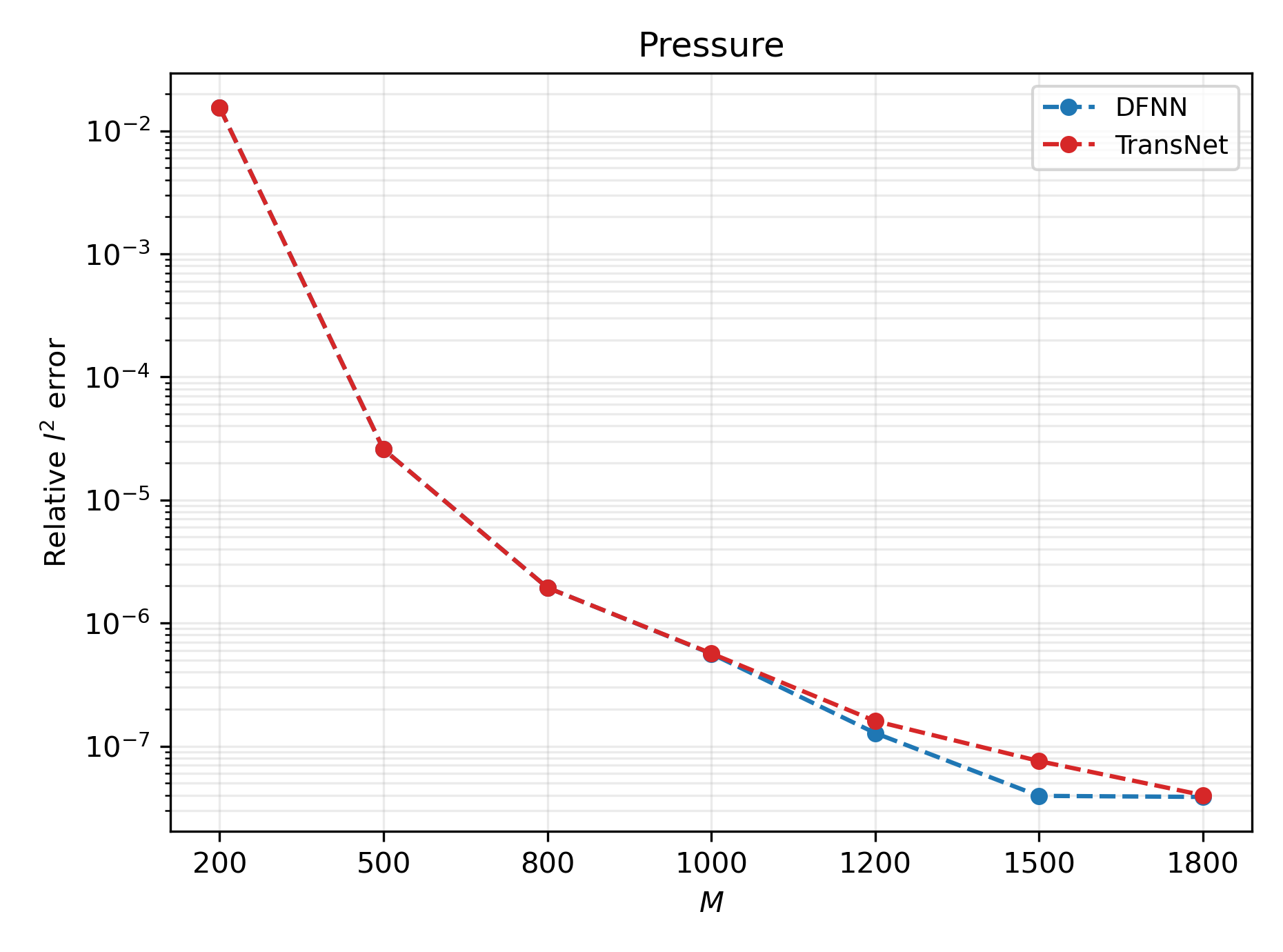}}
  \end{minipage}
  \begin{minipage}[t]{0.45\linewidth}
    {\includegraphics[scale=0.4]{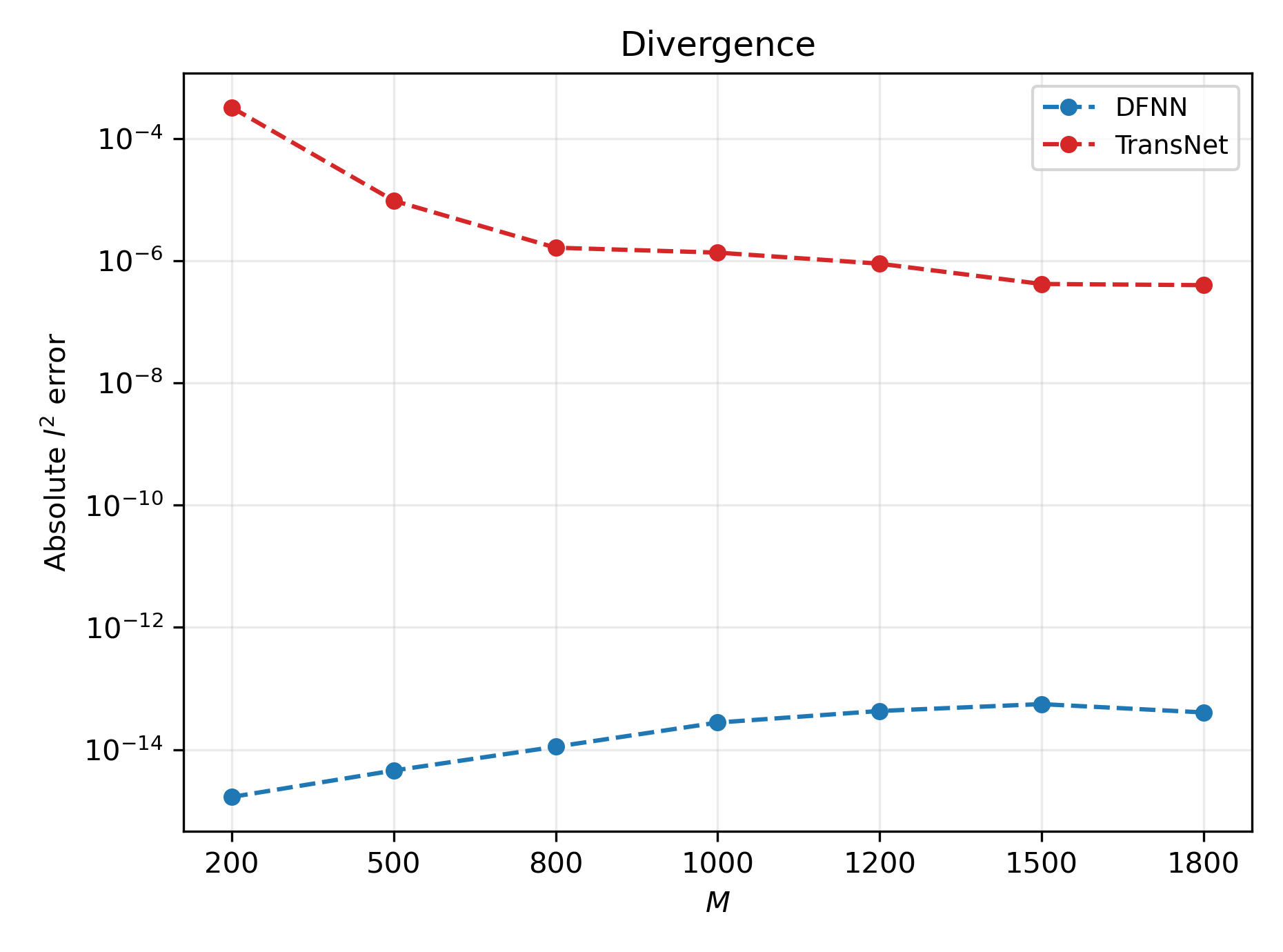}}
  \end{minipage}
  \begin{minipage}[t]{0.45\linewidth}
    {\includegraphics[scale=0.4]{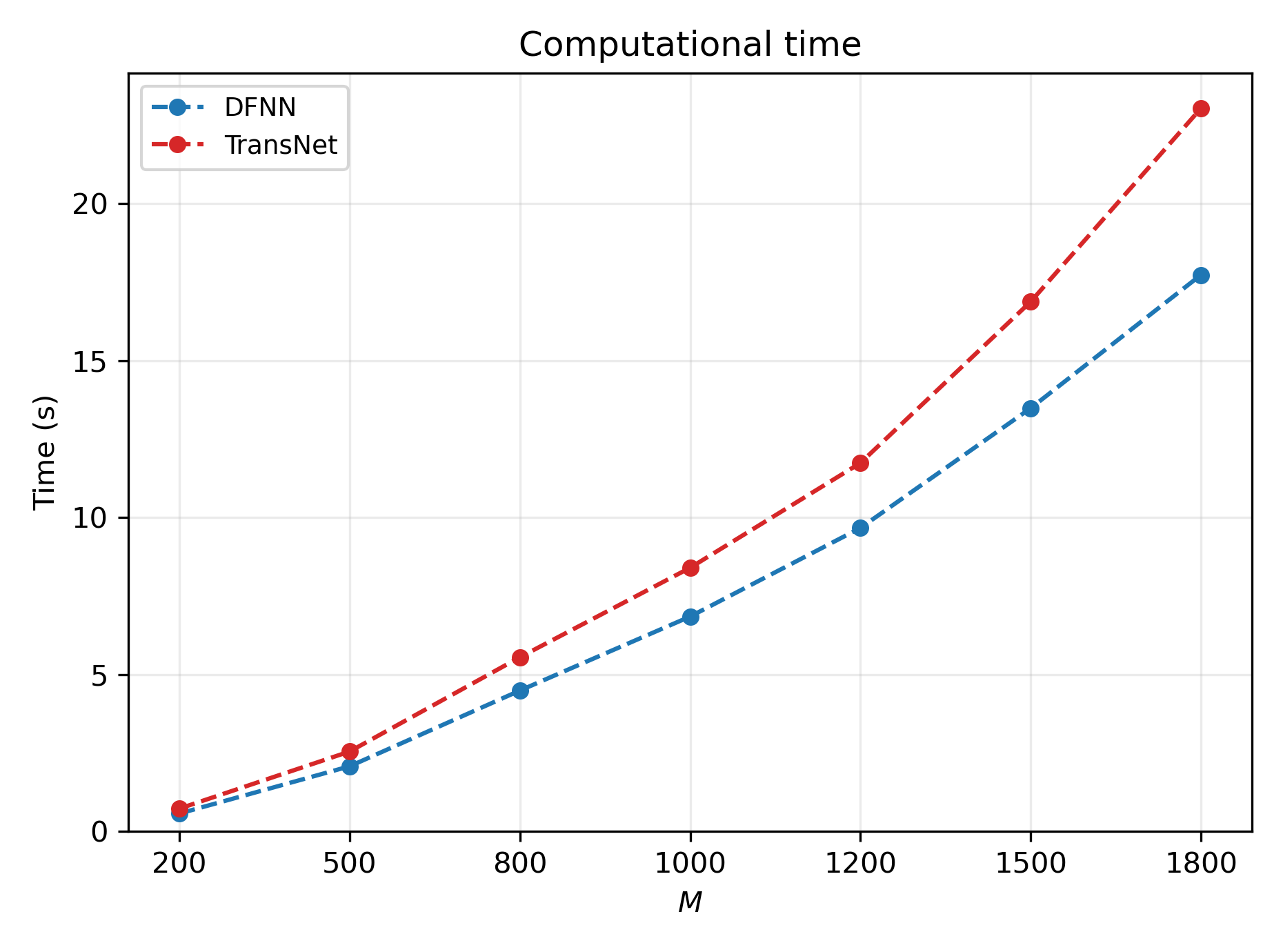}}
  \end{minipage}
    \caption{Numerical performance comparison for the 3D Stokes problem with $\nu=10^{-4}$ and varying numbers of basis functions $M$.}\label{fig-3d-1}
\end{figure}

We further study the effects of the number of training samples and basis functions on the accuracy and computational cost in Tab.~\ref{tab3}. The results show that increasing the number of training samples has almost no effect on the numerical performance, whereas increasing the number of basis functions improves the accuracy at the expense of additional computational cost.

\begin{table}[H]
  \centering
  \caption{Numerical performance of DFNN with $\gamma =1$ and $r=3$ at $\mathrm{Re}=10^{4}$}\label{tab3}
  \setlength{\tabcolsep}{4.5pt}
\renewcommand{\arraystretch}{1.2}
\footnotesize
    \begin{tabular}{cccccccccc}
    \hline
     &$M$  &$N_{pde}$  &$N_{bd}$   &$\mbox{error}_u$ &$\mbox{error}_v$  &$\mbox{error}_w$ &$\mbox{error}_p$  &$\mbox{error}_{div}$   &GPU time(s) \\
    \hline
    Baseline  & 1~000  & 8~000   & 2~400  & 8.0790e-07  & 9.3207e-07 & 9.7578e-07  & 5.1180e-07 & 2.7669e-14   & 6.5 \\
    Case I     & 1~000  & 10~000  & 2~400  & 7.9697e-07  & 9.3600e-07 & 9.7459e-07  & 5.1386e-07 & 2.7731e-14   & 6.4 \\
    Case II    & 1~000  & 8~000   & 3~750  & 8.0882e-07  & 9.2868e-07 & 9.6598e-07  & 5.1179e-07 & 2.7295e-14   & 6.4 \\
    Case III   & 1~000  & 10~000  & 3~750  & 7.8635e-07  & 9.1497e-07 & 9.5160e-07  & 5.1385e-07 & 2.7172e-14 & 6.3 \\
    Case IV    & 1~500  & 8~000   & 2~400  & 5.9470e-08  & 6.1063e-08 & 6.0425e-08  & 3.8441e-08 & 5.9341e-14 &12.4 \\
    \hline
    \end{tabular}
\end{table}

Next, we investigate the effects of the shape parameters $\gamma$ and $r$ on Decoupled-DFNN and TransNet with a fixed $M=1~500$ and $\mathrm{Re} = 10^4$, as shown in Figs.~\ref{fig-3d-2}--\ref{fig-3d-3}. For $\gamma$, unlike the behavior observed in Subsection~\ref{sub:2d_navier_stokes_problem}, the error first decreases and then increases as $\gamma$ increases, indicating that the accuracy is sensitive to the choice of $\gamma$. In contrast, the effect of $r$ is similar to that observed in subsection~\ref{sub:2d_navier_stokes_problem}: the error decreases as $r$ increases and then remains relatively stable for $r\geq2.5$.

\begin{figure}[H]
\centering
    \begin{minipage}[t]{0.32\linewidth}
    {\includegraphics[scale=0.3]{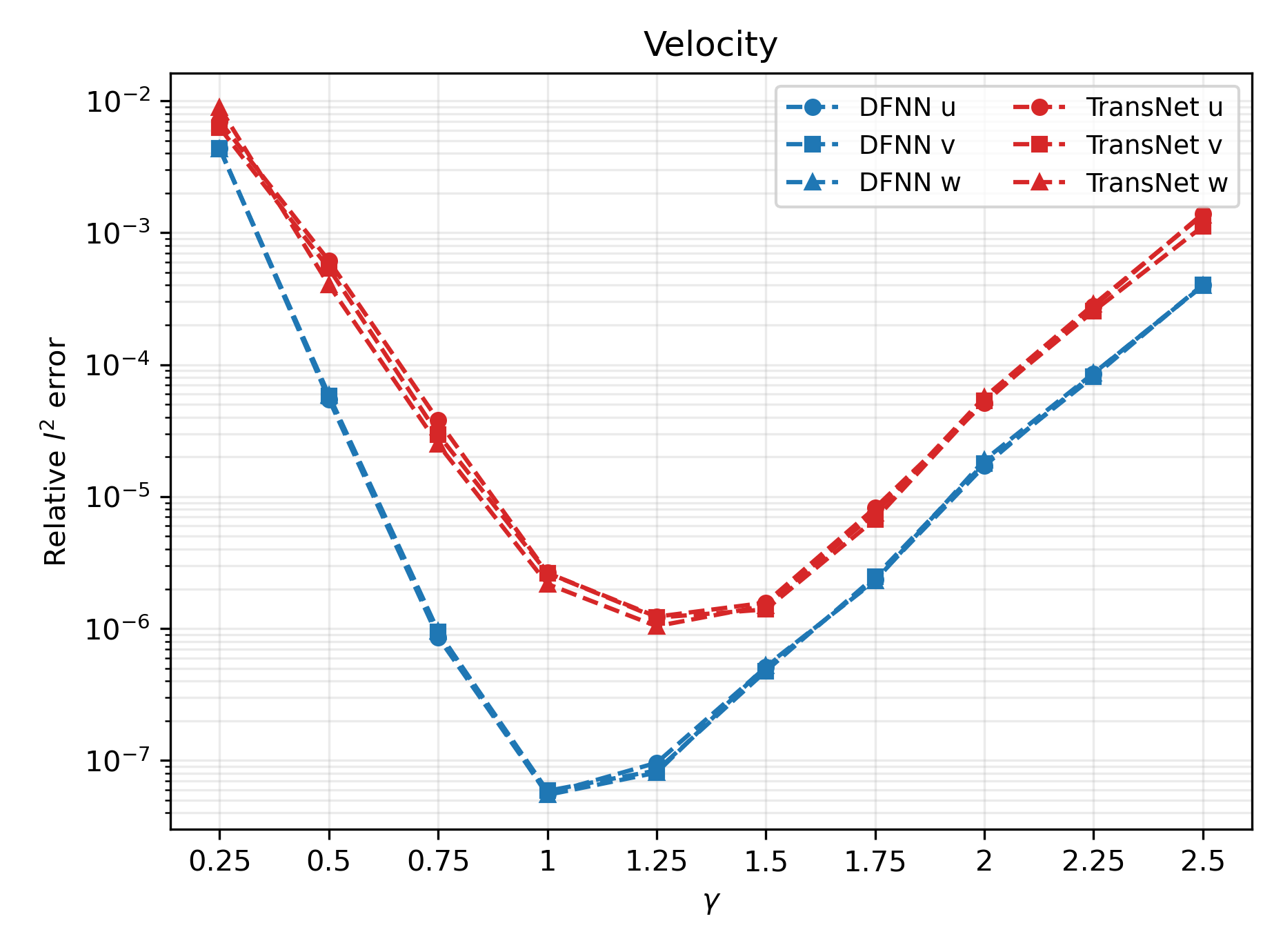} }
  \end{minipage}
  \begin{minipage}[t]{0.32\linewidth}
    {\includegraphics[scale=0.3]{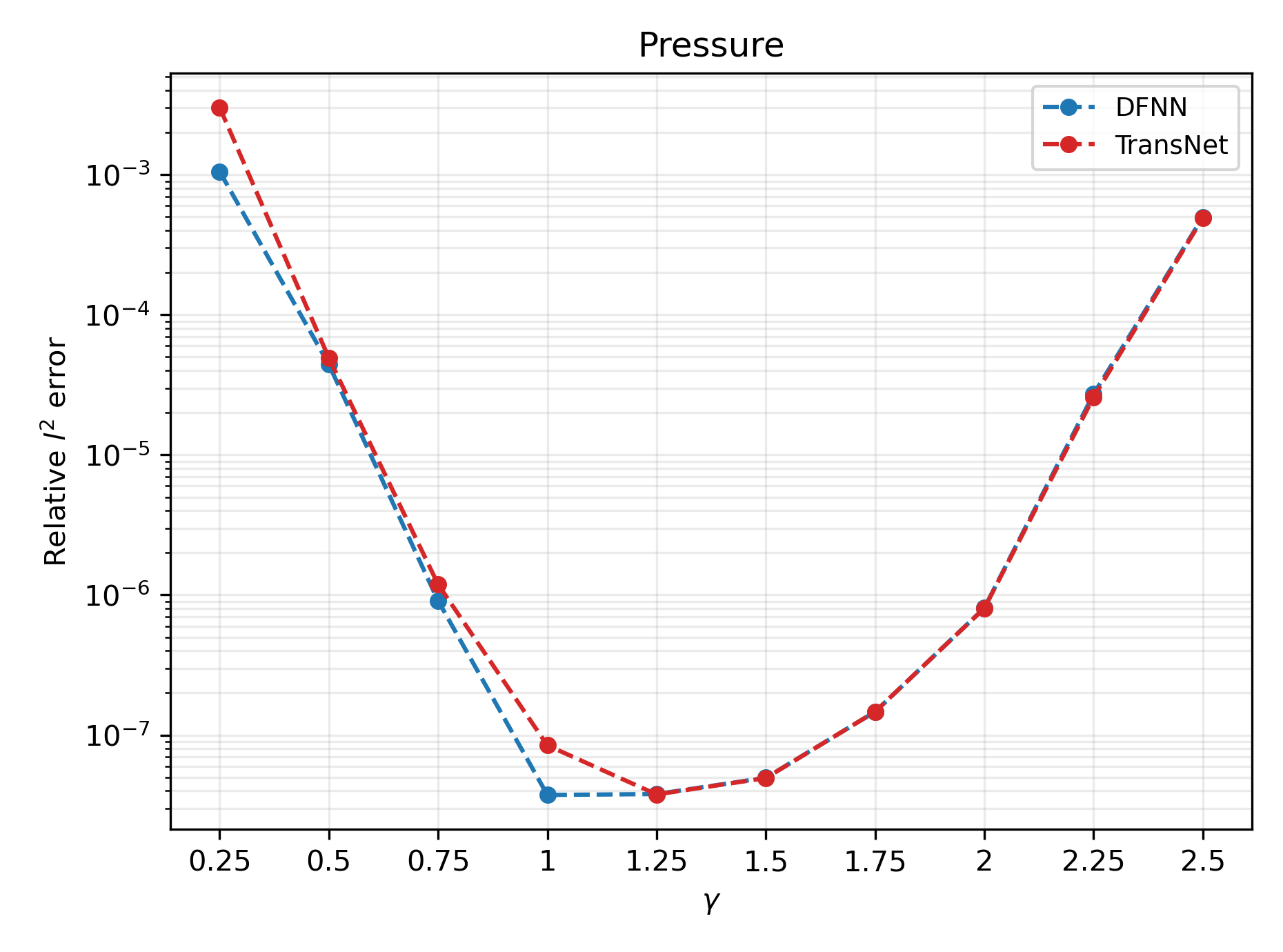} }
  \end{minipage}
  \begin{minipage}[t]{0.32\linewidth}
    {\includegraphics[scale=0.3]{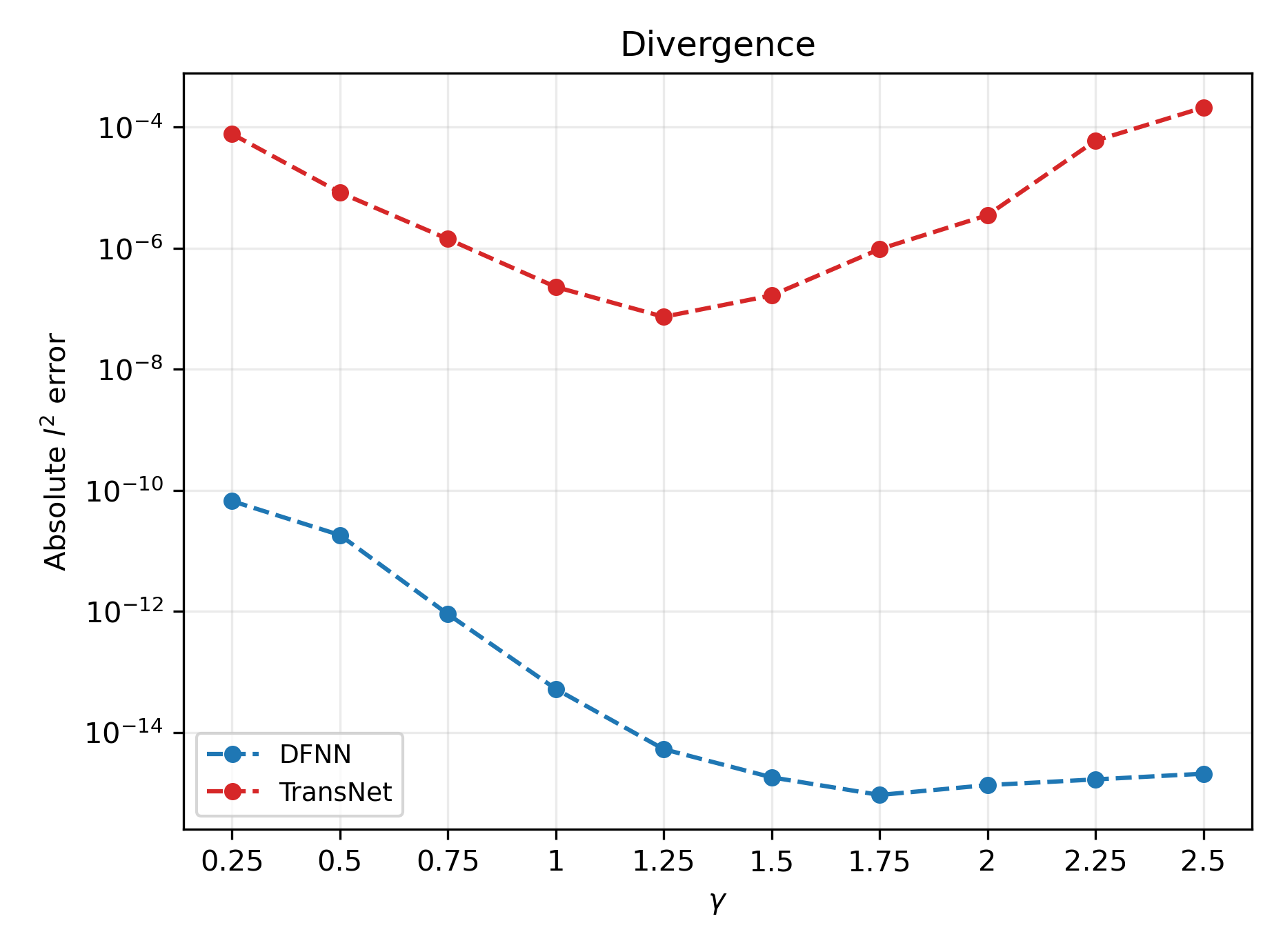} }
  \end{minipage}
    \caption{Numerical performance comparison for the 3D Stokes problem with $r=3$ and  varying shape parameter $\gamma$.}\label{fig-3d-2}
\end{figure}

\begin{figure}[H]
\centering
    \begin{minipage}[t]{0.32\linewidth}
    {\includegraphics[scale=0.3]{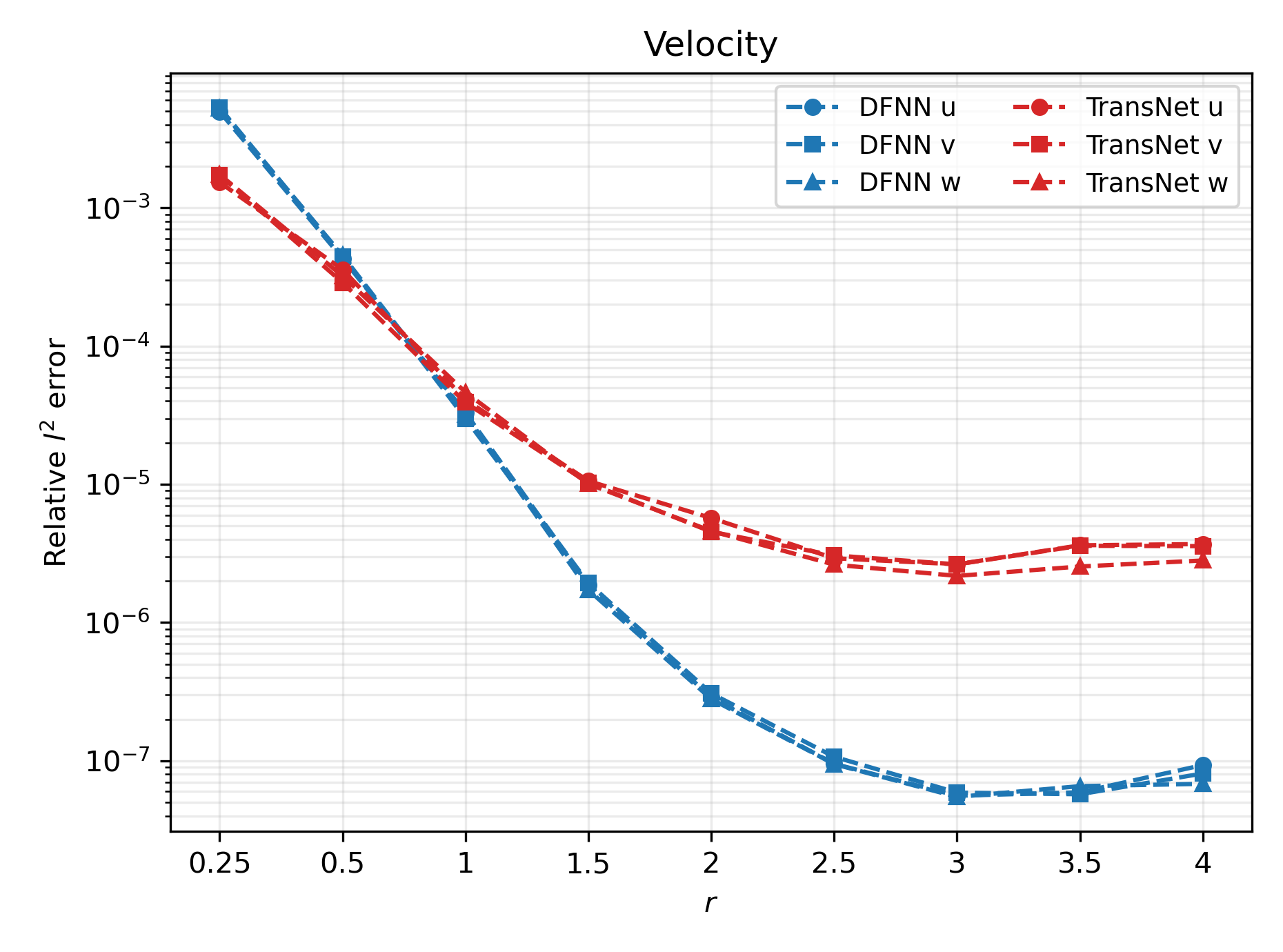} }
  \end{minipage}
  \begin{minipage}[t]{0.32\linewidth}
    {\includegraphics[scale=0.3]{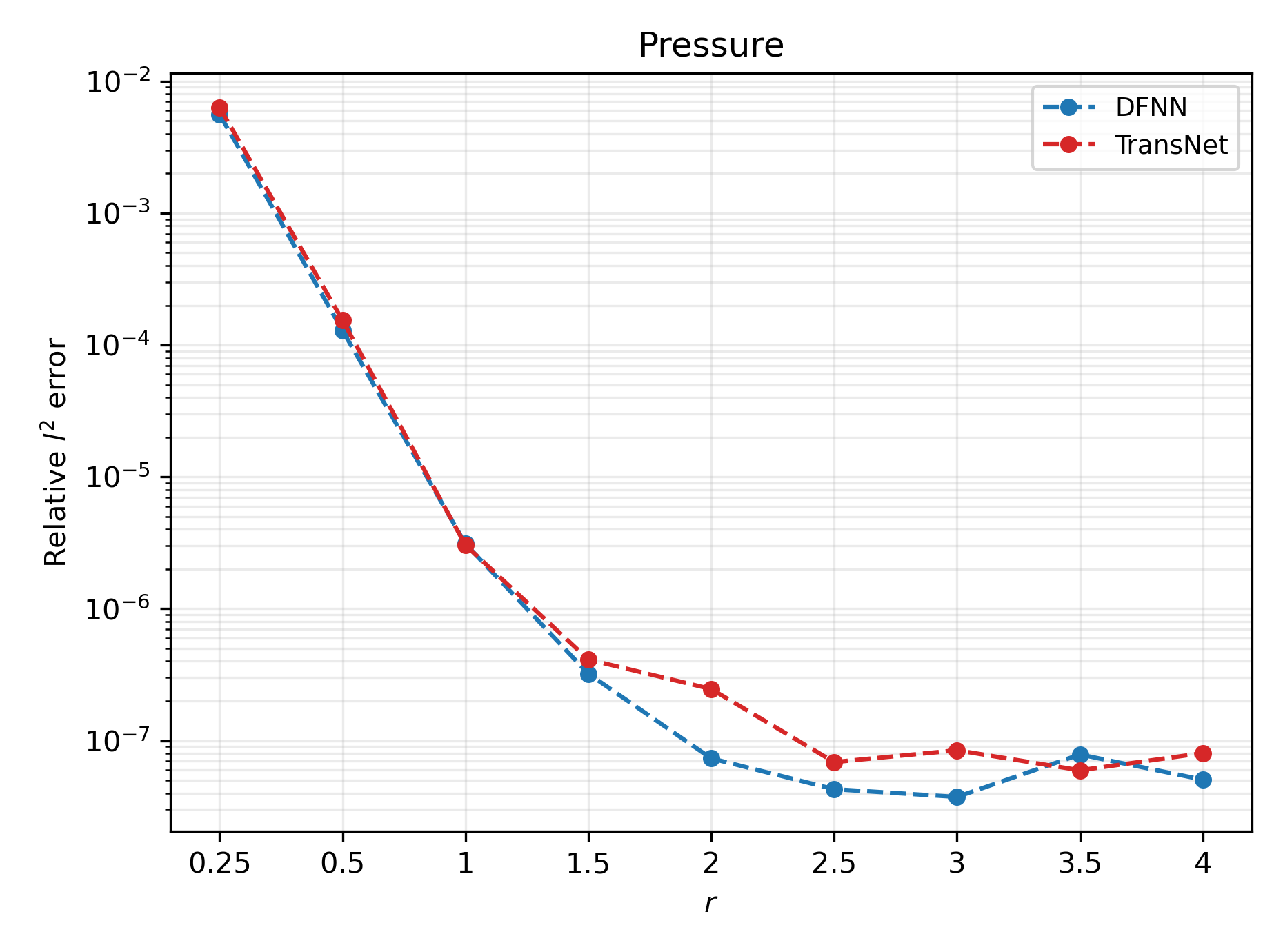} }
  \end{minipage}
  \begin{minipage}[t]{0.32\linewidth}
    {\includegraphics[scale=0.3]{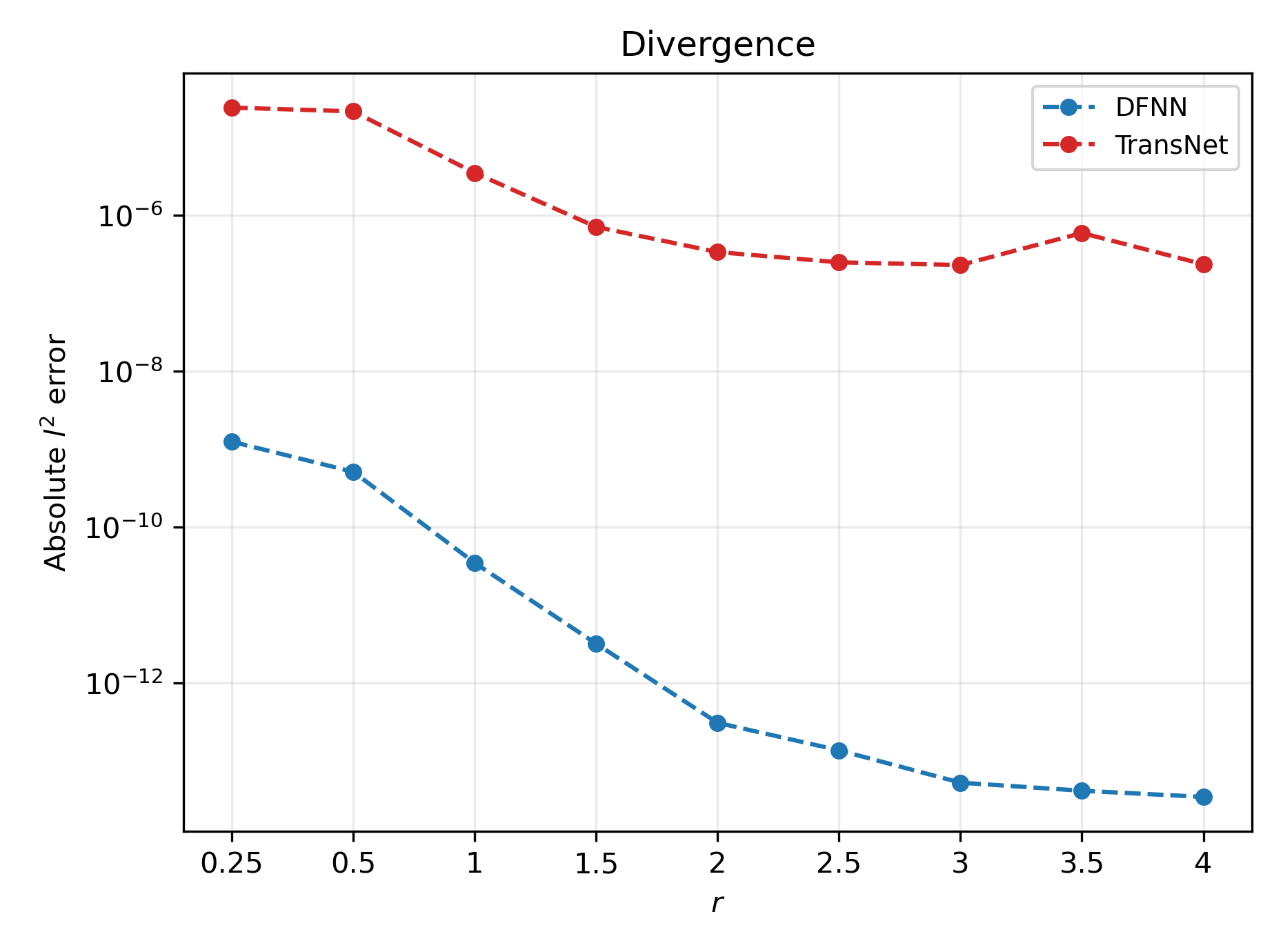} }
  \end{minipage}
    \caption{Numerical performance comparison for the 3D Stokes problem with $\gamma=1$ and varying shape parameter $r$.}\label{fig-3d-3}
\end{figure}

Finally, we investigate the numerical robustness of Decoupled-DFNN and TransNet with respect to the viscosity coefficient $\nu$, using fixed parameters $M=1500$, $\gamma=1$, and $r=3$, as shown in Fig.~\ref{fig-3d-4}. We find that the velocity errors of both Decoupled-DFNN and TransNet initially decrease and subsequently increase as $\nu$ is reduced. Nevertheless, Decoupled-DFNN demonstrates superior velocity accuracy in the low-viscosity regime ($\nu\leq10^{-3}$). Regarding pressure, the errors of both methods decrease as $\nu$ is reduced and remain nearly stable in the low-viscosity regime, with Decoupled-DFNN outperforming TransNet specifically when $\nu\leq10^{-3}$. Furthermore, Decoupled-DFNN consistently requires less computational time than TransNet over the tested Reynolds numbers.

\begin{figure}[H]
\centering
    \begin{minipage}[t]{0.45\linewidth}
    {\includegraphics[scale=0.4]{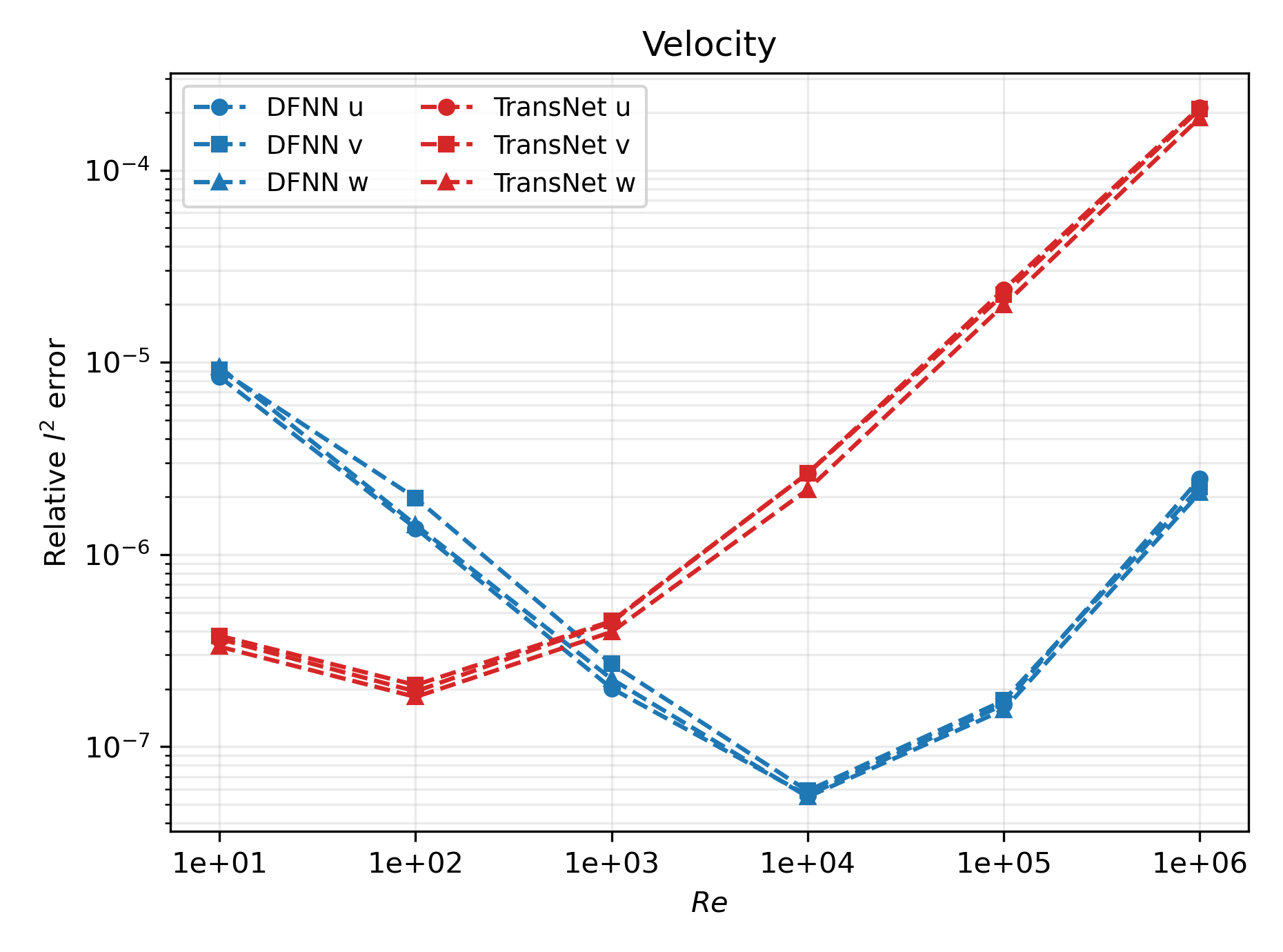}}
  \end{minipage}
  \begin{minipage}[t]{0.45\linewidth}
    {\includegraphics[scale=0.4]{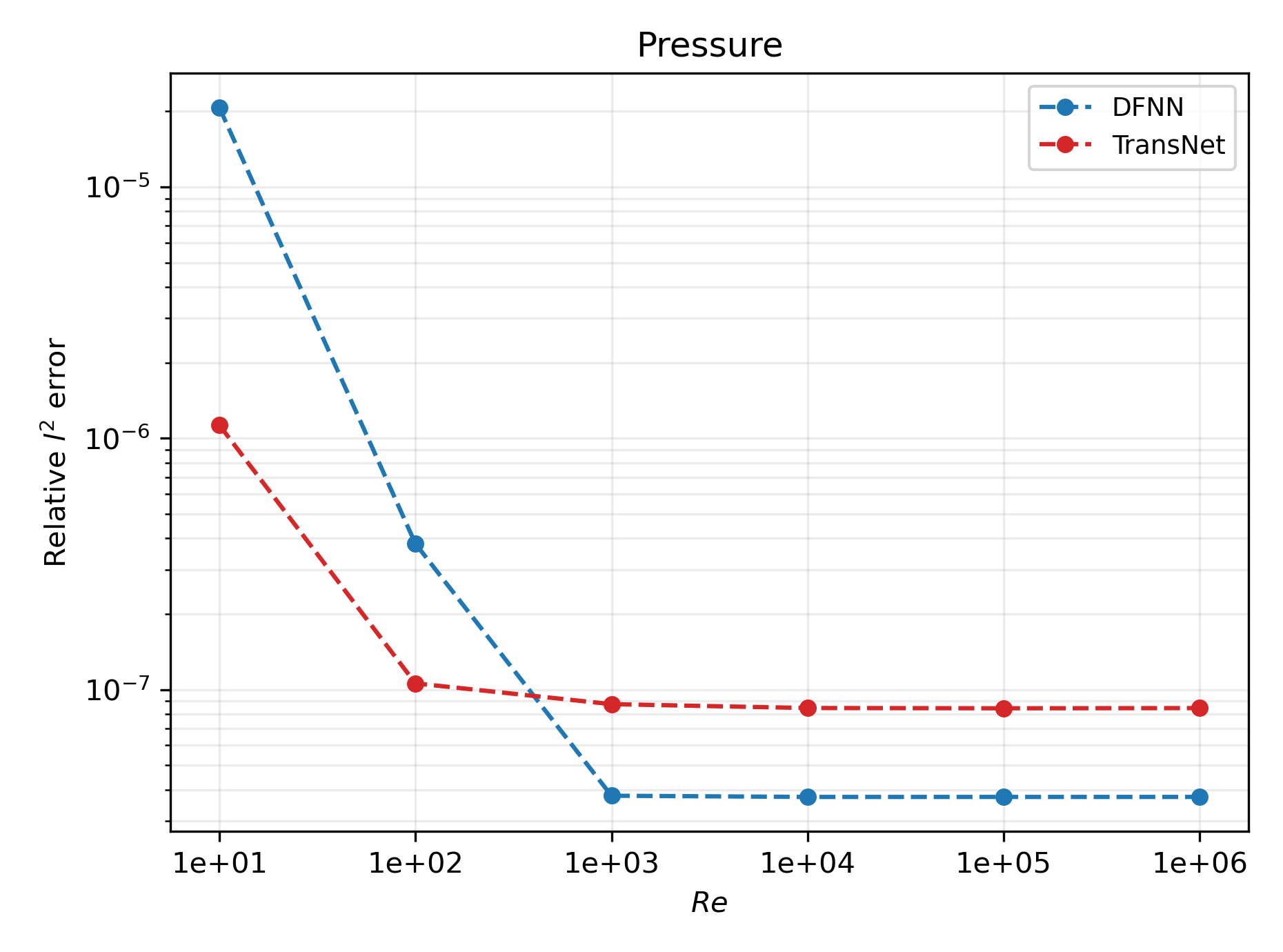}}
  \end{minipage}
  \begin{minipage}[t]{0.45\linewidth}
    {\includegraphics[scale=0.4]{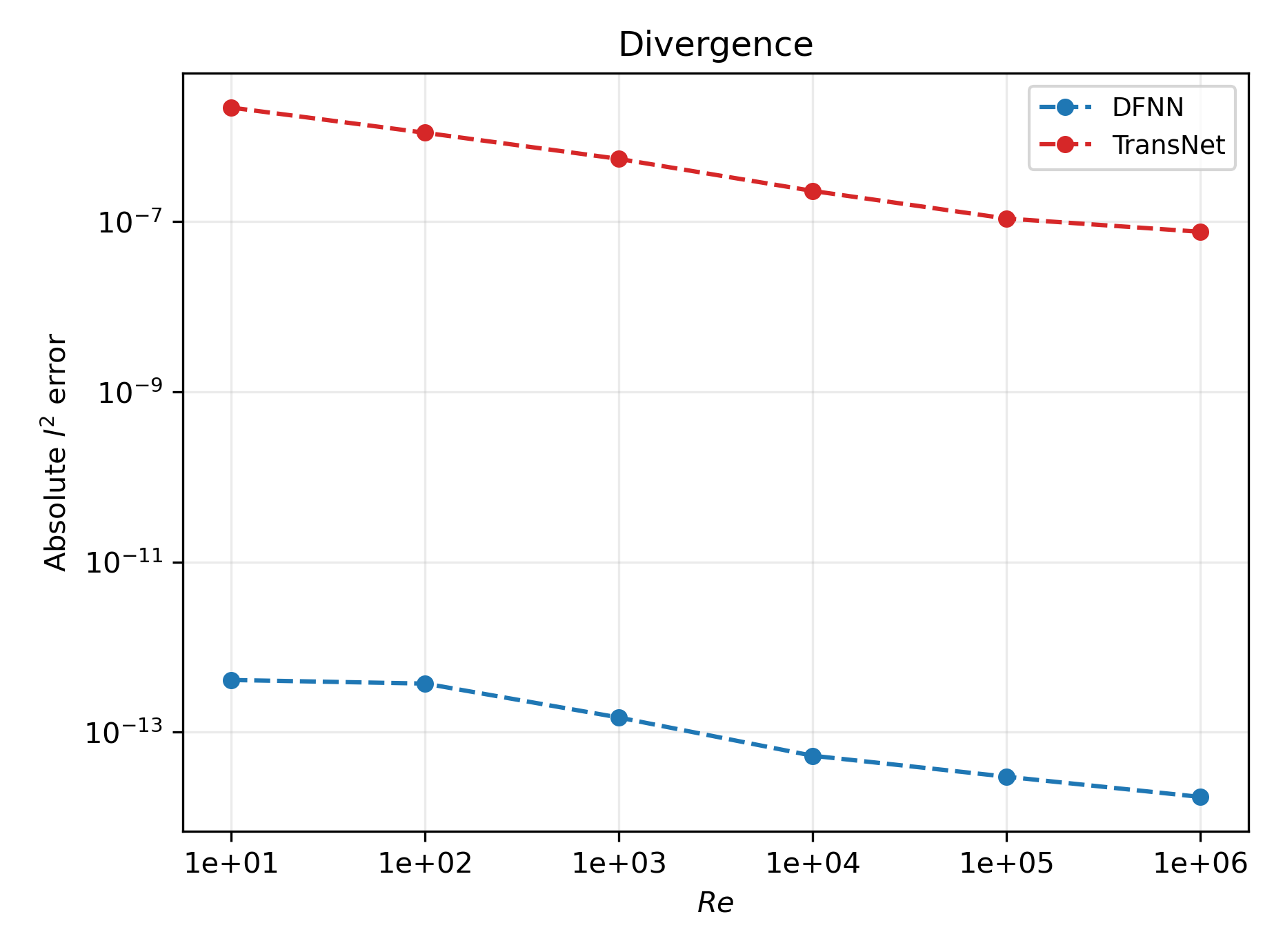}}
  \end{minipage}
  \begin{minipage}[t]{0.45\linewidth}
    {\includegraphics[scale=0.4]{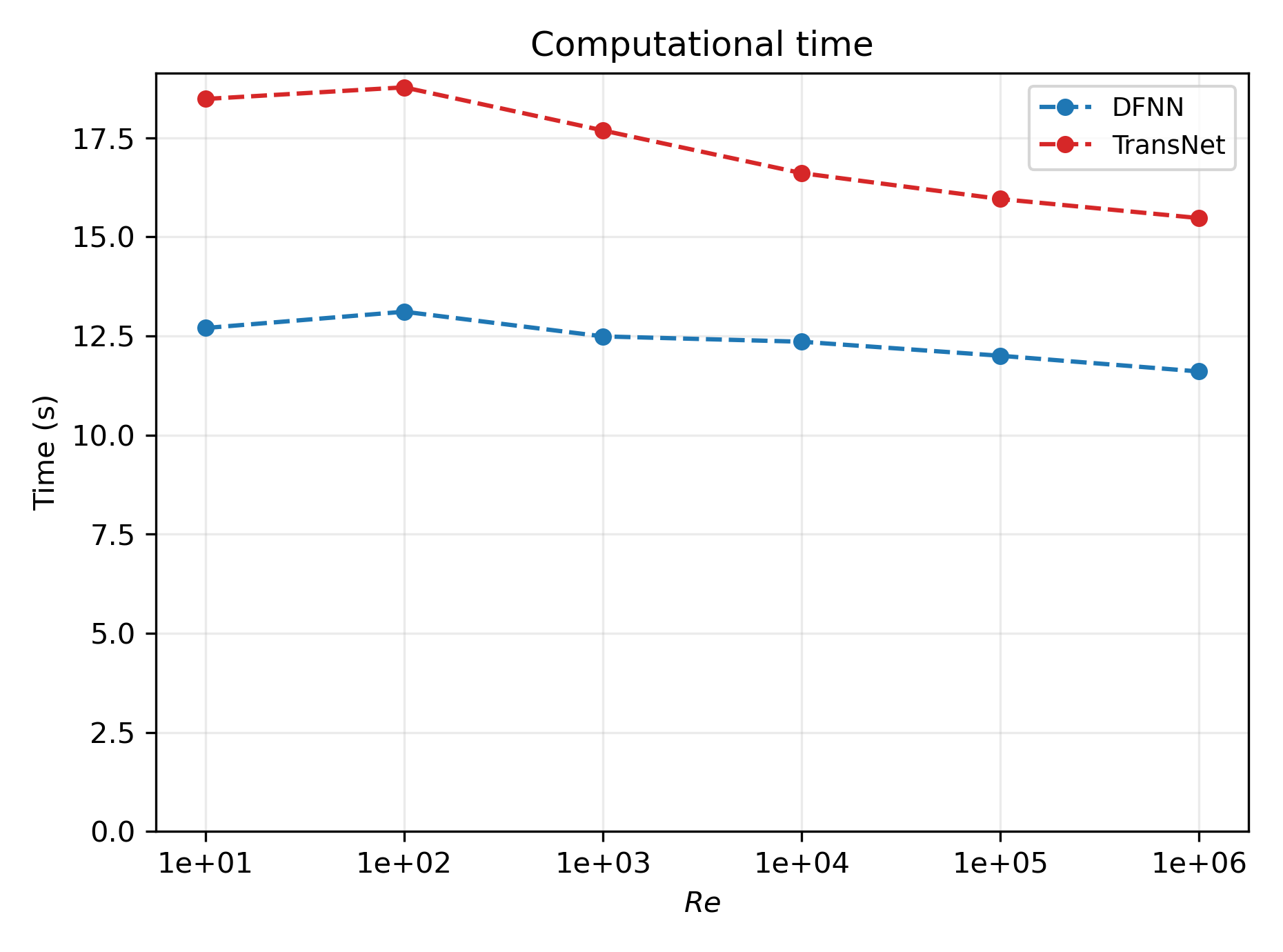}}
  \end{minipage}
    \caption{Numerical performance comparison for the 3D Stokes problem with $M=1~500$ and varying Reynolds number $\mathrm{Re}$.}\label{fig-3d-4}
\end{figure}

\subsection{3D Navier-Stokes problem}
Consider the steady Navier--Stokes problem on the unit cube $\Omega=(0,1)^3$. We prescribe a smooth, divergence-free velocity field combining trigonometric terms with localized Gaussian contributions. The exact velocity components are given by
\begin{align*}
u(x,y,z)
=&0.20\sin(2\pi x)\cos(\pi y)\sin(\pi z)
   -0.25\sin(\pi x)\sin(2\pi y)\cos(\pi z) \\
   &-1.68(y-0.65)G_3(x,y,z)
   -1.44(z-0.45)G_2(x,y,z),\\
v(x,y,z)
=&0.30\sin(\pi x)\sin(\pi y)\cos(\pi z)
   -0.40\cos(2\pi x)\sin(\pi y)\sin(\pi z)\\
 &-1.92(z-0.55)G_1(x,y,z)
   +1.68(x-0.45)G_3(x,y,z),\\
w(x,y,z)
=&0.25\cos(\pi x)\sin(2\pi y)\sin(\pi z)
   -0.30\sin(\pi x)\cos(\pi y)\sin(\pi z)\\
 &+1.44(x-0.60)G_2(x,y,z)
   +1.92(y-0.45)G_1(x,y,z).
\end{align*}
The corresponding pressure is prescribed as
\[
p(x,y,z)=10\sin(2\pi x)\sin(2\pi y)\sin(2\pi z).
\]
The Gaussian functions appearing in the velocity field are defined by
\begin{align*}
G_1(x,y,z)
&=\exp\!\left[-12\bigl((x-0.35)^2+(y-0.45)^2+(z-0.55)^2\bigr)\right],\\
G_2(x,y,z)
&=\exp\!\left[-12\bigl((x-0.60)^2+(y-0.35)^2+(z-0.45)^2\bigr)\right],\\
G_3(x,y,z)
&=\exp\!\left[-12\bigl((x-0.45)^2+(y-0.65)^2+(z-0.35)^2\bigr)\right].
\end{align*}

In this implementation, we use $5~000$ Halton samples in the interior and $3~750$ boundary points distributed on $25\times25\times6$ uniform grids for training. Additionally, $8~192$ interior Halton samples are used for testing. The resulting nonlinear system is solved by the Gauss--Newton method. To accelerate convergence, following the idea of \cite{Int-Deep}, we first obtain a PINN approximation and use it as the initial guess for the Gauss--Newton iteration. The maximum number of iterations is set to $20$, with a tolerance of $5\times10^{-8}$.

\begin{table}[H]
  \centering
  \caption{Numerical performance for the three-dimensional
  Navier--Stokes problem.} \label{tab4}
  \setlength{\tabcolsep}{4.5pt}
  \renewcommand{\arraystretch}{1.25}
  \footnotesize
  \begin{tabular}{cccccccccc}
    \hline
    Method & $M$ & $\mathrm{Re}$
    & $\mathrm{error}_u$ & $\mathrm{error}_v$
    & $\mathrm{error}_w$ & $\mathrm{error}_p$
    & $\mathrm{error}_{\mathrm{div}}$
    & Iterations & GPU time (s) \\
    \hline
    DFNN     & 1~000 & $10^2$
    & 3.7387e-03 & 1.7236e-03 & 2.3701e-03
    & 4.8831e-05 & 4.6504e-15 & 20 & 76.2 \\
    TransNet & 1~000 & $10^2$
    & 5.0311e-04 & 2.5167e-04 & 3.4968e-04
    & 1.1532e-05 & 2.1532e-04 & 20 & 117.7 \\
    PINN     & -- & $10^2$
    & 4.5100e-02 & 3.5927e-02 & 5.0904e-02 & 1.2228e-03  & 1.1764e-02 & $5000+500$ & 1411.9 \\
    \hline
    DFNN     & 500 & $10^3$
    & 1.6258e-02 & 7.0855e-03 & 9.8572e-03
    & 1.2996e-03 & 1.3394e-15 & 20 & 24.9 \\
    TransNet & 500 & $10^3$
    & 2.3176e-01 & 1.1780e-01 & 1.6871e-01
    & 4.5550e-03 & 1.0669e-02 & 20 & 37.7 \\
    DFNN     & 1~000 & $10^3$
    & 7.6657e-04 & 4.3365e-04 & 5.9413e-04
    & 1.0455e-05 & 4.5654e-15 & 20 & 77.4 \\
    TransNet & 1~000 & $10^3$
    & 1.1912e-03 & 7.5074e-04 & 8.9239e-04
    & 1.1797e-05 & 1.3559e-04 & 20 & 115.4 \\
    DFNN     & 1~500 & $10^3$
    & 9.0364e-05 & 5.1727e-05 & 7.2329e-05
    & 8.9506e-07 & 1.6127e-14 & 20 & 146.8 \\
    TransNet & 1~500 & $10^3$
    & 5.1112e-05 & 3.2256e-05 & 3.7486e-05
    & 7.0462e-07 & 1.8149e-05 & 20 & 250.6 \\
    PINN     & -- & $10^3$
    & 6.7918e-02 & 4.5556e-02 & 6.7155e-02
    & 1.2657e-03 & 1.2269e-02 & $5000+500$ & 1434.9 \\
    \hline
    DFNN     & 1~000 & $10^4$
    & 6.7979e-04 & 4.0948e-04 & 5.2833e-04
    & 9.7813e-06 & 4.4130e-15 & 20 & 76.9 \\
    TransNet & 1~000 & $10^4$
    & 1.2963e-03 & 7.8961e-04 & 9.7132e-04
    & 1.0558e-05 & 1.3242e-04 & 20 & 116.0 \\
    PINN     & -- & $10^4$
    & 6.5507e-02 & 4.8728e-02 &6.8361e-02 & 8.6678e-04 & 1.2169e-02 & $5000+500$ & 1403.8 \\
    \hline
  \end{tabular}
\end{table}

We set the shape parameters to $\gamma=1.5$ and $r=2$ and compare Decoupled-DFNN, TransNet, and PINN for different cases in Tab.~\ref{tab4}. Three main observations can be made. First, with a sufficient number of basis functions, both Decoupled-DFNN and TransNet outperform PINN in terms of accuracy and computational efficiency. PINN requires substantially more GPU time than the two neural network basis methods in all tested cases, while attaining velocity errors of order \(10^{-2}\). In contrast, Decoupled-DFNN and TransNet can reach error levels of \(10^{-4}\)-\(10^{-5}\) with a sufficient number of basis functions. Second, Decoupled-DFNN satisfies the divergence-free condition almost exactly, demonstrating its clear advantage in structure preservation. Third, Decoupled-DFNN and TransNet exhibit different advantages depending on $M$ and the Reynolds number. At $M=1~000$, Decoupled-DFNN provides higher accuracy than TransNet for $\mathrm{Re}=10^3$ and $10^4$, whereas TransNet is more accurate for $\mathrm{Re}=10^2$; Decoupled-DFNN requires less GPU time in all three cases. For $\mathrm{Re}=10^3$, increasing $M$ from $M=1~000$ to $1~500$ makes TransNet more accurate than Decoupled-DFNN, at the expense of a higher computational cost. On the other hand, Decoupled-DFNN may perform better at larger Reynolds numbers, whereas the velocity errors of TransNet increase as the Reynolds number grows for the tested cases with $M=1000$. Therefore, Decoupled-DFNN offers a favorable balance between accuracy and efficiency, particularly when computational resources are limited or the Reynolds number is relatively large.

% For example, if you use {\sf MATLAB} to do figure plots, axis labels should be at least point 18. Title should be 24 points or above. Tick marks labels better have 14 points or above. Line width should be 2 (or above).

\section{Conclusion}
In this work, we present a decoupled divergence-free neural network method for solving incompressible flow problems, including the Stokes and Navier-Stokes problems. To ensure physical consistency, we adopt the stream function formulation in two-dimensions and vorticity-vector potential formulations in three-dimensions, and further exploit the properties of curl operator to derive fully decoupled divergence-free formulations about velocity and pressure. These formulations lead to a sequential solution strategy: the velocity field is first computed without involving the pressure, and the pressure field is then recovered. These decoupled subproblems are solved using the TransNet framework. To address the inherent nonlinearity of the Navier-Stokes equations, we further employ the Gauss-Newton method to linearize the governing equations, thereby transforming the problem into a sequence of linear subproblems. Numerical experiments demonstrate several key advantages of the Decoupled-DFNN:
\begin{itemize}
\item \textbf{Exact incompressibility}:
The incompressibility constraint is satisfied up to nearly machine precision ($10^{-14}$) in the main benchmark tests, significantly outperforming standard PINNs and TransNets that enforce this condition via soft penalty terms.

\item \textbf{High efficiency}:
The proposed method exhibits substantially reduced computational cost. For the tested linear problems, it generally delivers higher accuracy with shorter execution time, while for nonlinear problems, it achieves comparable or competitive accuracy with lower computational cost in the tested cases. Moreover, it provides superior accuracy in several of the tested low-viscosity cases.
\end{itemize}
In addition, we find that the Gaussian activation function performs better than the tanh activation function for the Kovasznay flow problem. The numerical results also indicate that the accuracy is relatively sensitive to the shape parameter \(\gamma\), whose appropriate choice is problem-dependent. In the present experiments, moderate values of \(\gamma\in[1,1.5]\) generally provide good performance. By contrast, increasing \(r\) generally improves the accuracy until the results become relatively stable, with \(r\in[2,4]\) performing well across the tested cases.

Overall, the proposed method offers a favorable balance between structure preservation and computational efficiency: it enforces the divergence-free condition exactly and reduces the complexity of the coupled system through decoupling, at the cost of higher regularity requirements, which are well accommodated by neural network approximations. We also observe that the performance for the Stokes problem is generally better than that for the Navier-Stokes equations. One possible reason is the solver for nonlinear systems. How to choose an efficient solver for nonlinear system is a challenge work in the further. In addition, the present study focuses on smooth problems. Extending the proposed framework to more challenging problems in- volving discontinuities, particularly through adaptive strategies, constitutes another important direction for future research.

%%%% Acknowledgments %%%%%%%%
\section*{Declarations}
\begin{itemize}
\item {\bf Code availability.}  The codes during the current study are available from the corresponding author on reasonable request.

\item {\bf Authors' contributions.} All authors contributed equally to this manuscript.
\end{itemize}

\section*{Acknowledgments}
The authors are indebted to the anonymous reviewers for valuable comments which improved an earlier version of the paper. J. Huang and H. Wang were partially supported by NSFC (Grant No. 12571390). T. Zhou was supported by National Natural Science Foundation of China under grants 12288201 and 12461160275, and the Science challenge project No. TZ2025006.

%%%% Bibliography  %%%%%%%%%%

\end{document}